\documentclass[10pt]{article}

\usepackage{setspace}
\usepackage{comment}
\usepackage{amsmath, amssymb, amsthm}
\usepackage{fullpage}
\usepackage{float}
\usepackage[all]{xy}
\usepackage{graphicx, epstopdf}
\usepackage{enumerate}

\newtheorem{definition}{Definition}[section]
\newtheorem{theorem}[definition]{Theorem}

\newtheorem{lemma}[definition]{Lemma}
\newtheorem{corollary}[definition]{Corollary}
\newtheorem{proposition}[definition]{Proposition}
\newtheorem{remark}[definition]{Remark}

\DeclareMathOperator{\SL}{SL}
\DeclareMathOperator{\GL}{GL}
\DeclareMathOperator{\Sp}{Sp}
\DeclareMathOperator{\rk}{rk}
\DeclareMathOperator{\GSp}{GSp}
\DeclareMathOperator{\ord}{ord}
\DeclareMathOperator{\tr}{tr}
\DeclareMathOperator{\disc}{disc}
\DeclareMathOperator{\Cl}{Cl}
\DeclareMathOperator{\diag}{diag}
\DeclareMathOperator{\USp}{USp}
\DeclareMathOperator{\vol}{vol}
\DeclareMathOperator{\cont}{cont}
\DeclareMathOperator{\PGSp}{PGSp}
\DeclareMathOperator{\Imm}{Im}

\DeclareMathOperator{\Sh}{Sh}
\DeclareMathOperator{\SK}{SK}

\DeclareMathOperator{\PGL}{PGL}
\DeclareMathOperator{\St}{St}

\providecommand{\abs}[1]{\left\lvert#1\right\rvert}
\providecommand{\norm}[1]{\lvert\lvert#1\rvert\rvert}

\newenvironment{myitem}
{\begin{itemize}

\setlength{\itemsep}{0pt}
\setlength{\parsep}{0pt}}
{\end{itemize}}

\usepackage{titlesec}
\titleformat*{\section}{\normalsize\centering\scshape}  

\title{\textrm{\textbf{\normalsize LOCAL SPECTRAL EQUIDISTRIBUTION FOR DEGREE TWO SIEGEL MODULAR FORMS IN LEVEL AND WEIGHT ASPECTS}}}
\author{\textrm{\scriptsize MARTIN DICKSON}}
\date{}

\numberwithin{equation}{section}

\begin{document}
\maketitle

\quote{\textsc{Abstract}.  We prove an equidistribution result for the Satake parameters of the local representations attached to Siegel cusp forms of degree $2$ of increasing level and weight, counted with a certain arithmetic weight.  We then apply this to compute the symmetry type of a similarly weighted distribution of the low-lying zeros of $L$-functions attached to these cusp forms.}

\section{Introduction}\label{introduction}

Classical Siegel modular forms of degree $2$, weight $k$, and level $N$ may be viewed as vectors inside cuspidal automorphic representations of $\GSp_4$.  As one increases the weight and level one might expect that these representations vary in a \textit{family}.  A recent survey article of Kowalski (\cite{Kowalski2011}) proposes that a reasonable family of automorphic representations should satisfy the following condition: after defining an appropriate notion of a conductor and way to count (i.e., an appropriate measure on) cuspidal automorphic representations up to a given conductor, the local components of these representations should, as we increase the conductor, be equidistributed amongst all possibilities.  Not only this, but the distributions of the local components at different places should be independent.  The purpose of this paper is to show, building on the paper \cite{KowalskiSahaTsimerman2012}, that this property holds for the representations generated by holomorphic Siegel modular forms of degree $2$.  \\

\noindent  A motivating example for us of this behaviour is due to Serre (\cite{Serre1997}, Th\'{e}or\`{e}me 1) and, independently, Conrey--Duke--Farmer (\cite{ConreyDukeFarmer1997}, Theorem 1).  Fix a finite set $S$ of primes, and  consider the set of all holomorphic cusp forms on $\SL_2(\mathbb{Z})$ of weight $k$, level $N$, and trivial nebentypus, with $k$ varying over even positive integers and $N$ varying over integers with none of their prime factors in $S$.  At the places $p \in S$ the local representation attached to a holomorphic cuspidal eigenform is unramified and is determined by the Hecke eigenvalue $\lambda(p; f)$ of $f$ under $T(p)$.  The above papers prove that, as $k$ and $N$ vary as above, normalized eigenvalues $\lambda'(p; f) = \lambda(p; f)/2p^{(k-1)/2}$ are equidistributed with respect to the $p$-adic Plancherel measure on $[-1, 1]$.\begin{footnote}{This is related to, but easier than, the Sato--Tate problem, where one fixes the modular form but varies the prime.}\end{footnote}  This implies, without invoking the machinery of Deligne, that ``most'' cusp forms satisfy the Ramanujan--Petersson conjecture $\lambda'(p; f) \in [-1, 1]$; but it also ties down a precise measure with respect to which the points $\lambda'(p; f)$ equidistribute.  Moreover, the joint asymptotic distribution for different primes $p$ are independent (see \S5.2 of \cite{Serre1997}).  A related result of Sarnak (\cite{Sarnak1987}) states that if one considers the $p$th Fourier coefficient of Maass forms averaged over the Laplacian eigenvalues, one finds that they are equidistributed with respect to this same measure.  Another example, pertinent for us, is a weighted version of the quoted result of \cite{Serre1997} and \cite{ConreyDukeFarmer1997} which is implicit in the work of Bruggeman (\cite{Bruggeman1978}).\\

\noindent The above results can be understood as solutions to \textit{equidistribution problems}, which can be set up very generally as follows: let $X$ be a topological space, $V$ a finite dimensional complex vector space endowed with an inner product $\langle, \rangle$, $Q$ any fixed non-negative quadratic form on $V$, and $H$ a finitely generated commutative algebra of hermitian operators acting on $V$.  Suppose that whenever $v \in V$ is an eigenvector for $H$ it has associated to it a point $a(v) \in X$ such that if $v_1$ and $v_2$ lie in the same eigenspace then $a(v_1) = a(v_2)$.  For each $v \in V$ let $\omega(v) = Q(v)/\langle v, v \rangle$.  For each orthogonal basis $\mathcal{B}$ of $V$ consisting of eigenforms for $H$, consider the measure on $X$ given by $\nu_{V, \omega} = \sum_{v \in \mathcal{B}} \omega(v) \delta_{a(v)}$ (where $\delta$ is the Dirac mass).

\indent Now suppose we keep $X$ fixed but vary $V$ over a sequence of finite dimensional vector spaces: then it makes sense to ask whether the sequence of measure $\nu_{V, \omega}$ converges weakly to some canonical measure $\mu$ on $X$.  In other words, we ask whether the points $a(v)$ with $v \in \mathcal{B}$ counted with the ``weighting'' $\omega$ equidistribute.  We are specifically interested in the case when $V$ is the space of automorphic forms on some group of some fixed conductor and infinity type (as we allow the conductor to increase), $H$ is the local Hecke algebra at $p$ for some fixed prime $p$ (or more generally the algebra generated by finitely many Hecke operators), and for an eigenform $f$ of $H$ the point $a(f)$ is the local Satake parameter at the prime $p$.  The problem is then one of spectral equidistribution, and the results of Serre and Conrey--Duke--Farmer is spectral equidistribution of holomorphic cusp forms when the weighting $\omega$ is constant.\\

\noindent  As stated above we are concerned with the case of holomorphic cusp forms on $\PGSp_4$, i.e. holomorphic Siegel modular forms with trivial character; the relevant congruence group is
\[\Gamma_0(N) = \left\{\left(\begin{smallmatrix} A & B \\ C & D \end{smallmatrix}\right) \in \Sp_4(\mathbb{Z}); \: C \equiv 0 \bmod N\right\}.\]
We write $\mathcal{S}_k(N)$ for the space of Siegel modular forms of weight $k$, level $\Gamma_0(N)$, and trivial nebentypus.  In order to state a version of the main result we assume a certain amount of familiarity with automorphic representations; the relevant parts of the theory will be explained in \S\ref{reps-general} and \S\ref{the-equidistribution-problem}.  As before fix a finite set of primes $S$.  Let $f \in \mathcal{S}_k(N)$, and suppose that $f$ is an eigenform for the local Hecke algebras at all primes in $S$.  Then for each prime $p \in S$ there is a spherical principal series representation $\pi_{f, p}$ of $\GSp_4(\mathbb{Q}_p)$ generated by $f$ (see Remark \ref{well-defined-isom-class-of-spsr}).  The isomorphism class of $\pi_{f, p}$ is determined by the eigenvalues of $f$ for the Hecke operators $T(p)$ and $T_1(p^2)$. Equivalently, the isomorphism class of $\pi_{f, p}$ is determined by the orbit of the Satake parameters $(a_p(f), b_p(f))$ of $\pi_{f, p}$ under a certain Weyl group which we denote by $W$.  Now it is known that the Satake parameters satisfy $0 < \abs{a_p(f)}, \abs{b_p(f)} \leq \sqrt{p}$, and the very deep generalized Ramanujan conjecture for $\GSp_4$ (a proof of which has recently appeared in \cite{Weissauer2009}) implies that if $f$ is not a Saito-Kurokawa lift then $\abs{a_p(f)} = \abs{b_p(f)} = 1$.  (If $f$ \textit{is} a Saito--Kurokawa lift then it easily follows that for an appropriate representative in the Weyl group orbit we have $\abs{a_p(f)} = 1$ and $\abs{b_p(f)} = \sqrt{p}$.)  We can therefore regard the local components as points on the space $Y_p = \{(a, b) \in \mathbb{C}^\times \times \mathbb{C}^\times;\: 0 < \abs{a}, \abs{b} \leq \sqrt{p}\}/W$, and the local components of the representations attached to non-Saito--Kurokawa lifts actually lie inside $I_p$ the product of the two unit circles in $Y_p$ (or rather its image under the quotient by $W$; this is the subspace of \textit{tempered representations}).  The main theorem can then be stated as follows:

\begin{theorem}  [Local equidistribution and independence, prototypical version\begin{footnote}{Theorem \ref{local-equidistribution-theorem-qualitive} is a slightly more general version of this theorem, which in fact contains an infinite family of local equidistribution and independence statement indexed by fundamental discriminants $-d$ ($d>0$) and characters $\Lambda$ of the ideal class group of $\mathbb{Q}(\sqrt{-d})$.  The above is $d=4$, $\Lambda = \mathbf{1}$.  The measure $\nu_{S, N, k}$ does not depend on the choice of basis (see Lemma \ref{measure-well-defined}) and in Theorem \ref{local-equidistribution-theorem-qualitive} a slight relaxation on the basis is allowed.  Theorem \ref{local-equidistribution-theorem-quantitative} is a quantitative version of Theorem \ref{local-equidistribution-theorem-qualitive}.}\end{footnote}]\label{local-equidistribution-theorem-prototype}  Let $S$ be a finite set of primes.  Let $k \geq 6$ be even and let $N \geq 1$ have none of its prime factors in $S$.  Let $Y_S = \prod_{p \in S} Y_p$, and define a measure $\nu_{S, N, k}$ on $Y_S$ by
\[\nu_{S, N, k} = \sum_{f \in \mathcal{S}_k(N)^*} \omega_{f, N, k}\delta_{\pi_S}(f),\]
where
\[\omega_{f, N, k} = \frac{\sqrt{\pi}(4\pi)^{3-2k} \Gamma\left(k - \frac{3}{2}\right) \Gamma(k-2)}{\vol(\Gamma_0(N) \backslash \mathbb{H}_2)} \frac{\abs{a(1_2; f)}^2}{4\langle f, f \rangle},\]
$\mathcal{S}_k(N)^*$ is any orthogonal basis for $\mathcal{S}_k(N)$ consisting of eigenforms for the local Hecke algebra at all $p \in S$, $\pi_{S}(f) = \prod_{p \in S} (a_p(f), b_p(f)) \in Y_S$, and $\delta$ denotes Dirac mass.  Then, as $k +N \rightarrow \infty$ with $k \geq 6$ varying over even integers and $N \geq 1$ varying over integers with none of their prime factors in $S$, the measure $\nu_{S, N, k}$ converges weak-$*$ to a certain product measure $\mu_S = \prod_p \mu_p$ on $Y_S$, which is the measure\begin{footnote}{The measures $\mu_p$ and their product $\mu_S$ will be constructed in detail in \S\ref{the-equidistribution-problem}.  In particular we will see that $\mu_S$ is actually supported on $I_p$.}\end{footnote} referred to in \cite{FurusawaShalika2002} as the Plancherel measure for the local Bessel model associated to $(4, \mathbf{1})$.  That is, for any continuous function $\varphi$ on $Y_S$,
\[\lim_{k+N \rightarrow \infty} \sum_{f \in \mathcal{S}_k(N)^*} \omega_{f, N, k} \:\varphi((a_p(f), b_p(f))_{p \in S}) = \int_{Y_S} \varphi \:  d\mu_{S}.\]
In particular if $\varphi = \prod_{p \in S} \varphi_p$ is a product function then
\[\lim_{k+N \rightarrow \infty} \sum_{f \in \mathcal{S}_k(N)^*} \omega_{f, N, k} \varphi((a_p(f), b_p(f))_{p \in S}) = \prod_{p \in S} \int_{Y_p} \varphi_p\: d\mu_p.\]\end{theorem}

\noindent  Theorem \ref{local-equidistribution-theorem-prototype} is a generalization of Theorem 1.6 of \cite{KowalskiSahaTsimerman2012}, which deals with the case when $N$ is fixed equal to $1$.\begin{footnote}{Setting $N=1$ it appears we have an extra factor of $\vol(\Gamma_0(1) \backslash \mathbb{H}_2)$.  However, our Petersson norms are normalised whereas those of \cite{KowalskiSahaTsimerman2012} are not, so the weights are in fact the same.}\end{footnote}  Note that the cusp forms are only required to be eigenfunctions at $p \in S$, a point which does not seem to have been emphasised in previous work in equidistribution.  Our methods of proof follow those of \cite{KowalskiSahaTsimerman2012}, with some small changes when arguing with Bessel models in \S\ref{bessel-models} and the main modifications coming from the need to track the dependency on both the weight $k$ and level $N$ in certain estimates of Fourier coefficients of cusp forms, carried out in \S\ref{sctn:fourier-coeff-bound}.  As well as the case considered in \cite{KowalskiSahaTsimerman2012} ($N=1$, $k$ varying) this result generalizes one which recently appeared in \cite{ChidaKatsuradaMatsumoto2011} ($k$ treated constant, $N$ varying); in treating the mixed case we also obtain a better decay with respect to $N$ than that in \cite{ChidaKatsuradaMatsumoto2011}.  Allowing both the weight and level allows for the most general notion of conductor in this context, so Theorem \ref{local-equidistribution-theorem-prototype} settles the question of local equidistribution and independence (at unramified primes) for representations attached to classical Siegel modular forms.  We remark now that we will actually prove a quantitative version of this (Theorem \ref{local-equidistribution-theorem-quantitative}), which will be useful for applications.\\

\noindent  In recent work (\cite{Shin2012}, \cite{ShinTemplier2013}) Shin and Shin--Templier have proved a very general local equidistribution and independence statement.  For any cuspidal automorphic representation of a reductive group $G$ with a discrete series representation at the archimedean place\begin{footnote}{We have stated a weak version of their result relevant to our setup, but their theorem is much more general.  In particular $\mathbb{Q}$ can be replaced by any totally real field and what follows is true verbatim.  The condition that the archimedean component admits a discrete series representation is, however, important.}\end{footnote}, they are able to count cuspidal automorphic representations with their natural weight $1$ (in contrast to the weight $\omega_{f, N, k}$ appearing in ours) and prove a local equidistribution, with the limit measure a suitable normalization of Plancherel measure, on the unitary dual of $G(\mathbb{Q}_p)$; and moreover they prove the expected independence as well.  The limits are taken in either the increasing weight or level aspect, and it is expected an appropriate combination of their arguments would deal with the mixed case.  The groups satisfying these hypotheses include $\GSp_4$ and also higher rank symplectic groups, as well as $\GL_1$ and $\GL_2$ (but not $\GL_n$ for $n \geq 3$).  \\

\noindent  Our Theorem \ref{local-equidistribution-theorem-prototype} (even for either fixed level or weight aspect) is not contained the work of Shin--Templier, due to our different weights in counting.  In fact, we see that the presence of the arithmetic factor in our weight affects the limiting measure (that is, our limit measure is \textit{not} Plancherel measure -- for more details on our limit measure see \S\ref{the-equidistribution-problem}).  Such behaviour has been observed in local equidistribution problems of cuspidal automorphic representations on general linear groups in various families when the representations are counted weighted by special values of associated $L$-functions.  Our arithmetic factor $\abs{a(1_2; f)}^2/\langle f, f\rangle$ \textit{prima facie} does not appear to be so significant, but, at least when $f$ is an eigenform, a deep conjecture of B\"{o}cherer relates $\abs{a(1_2; f)}^2$ to $L(1/2, f)L(1/2, f \times \chi_{-4})$ (the $L$-functions are normalized to have functional equation relating $s$ with $1-s$).  In fact, the effect of the factor on the equidistribution problem can be interpreted as \textit{evidence} for B\"{o}cherer's conjecture -- see \cite{KowalskiSahaTsimerman2012} \S5.4 for a discussion of this when $N=1$.  \\

\noindent  After obtaining a quantitative version of Theorem \ref{local-equidistribution-theorem-prototype} we turn to the problem of low-lying zeros of $L$-functions of Siegel modular forms of weight $k$ and level $N$.  We attach the ``spin'' $L$-function $L(s, \pi_f)$ to an irreducible constituent $\pi_f$ of the cuspidal automorphic representation generated by $f$, and for any even Schwartz function $\Phi$ whose Fourier transform has compact support we consider
\[D(\pi_f; \Phi) = \sum_\rho \Phi\left(\frac{\gamma}{2\pi} \log C_{k, N}\right)\]
where $\rho = 1/2 + i\gamma$ varies over all zeros of $L(s, \pi_f)$ inside the critical strip with multiplicity, and $C_{k, N}$ is a certain analytic conductor as defined in \S\ref{sctn:low-lying-zeros}.  We assume the Riemann hypothesis: namely all $\gamma \in \mathbb{R}$.  $D(\pi_f; \Phi)$ reflects the distribution of the low-lying zeros of the single $L$-function $L(s, \pi_f)$.  We study an averaged version of this: let
\[D(N, k; \Phi) = \frac{1}{\sum_{f \in \mathcal{S}_k(N)^*}\omega_{f, N, k}} \sum_{f \in \mathcal{S}_k(N)^\#} \omega_{f, N, k} D(\pi_f; \Phi)\]
where $\omega_{f, N, k}$ is the weight from Theorem \ref{local-equidistribution-theorem-prototype}\begin{footnote}{Theorem \ref{local-equidistribution-theorem-prototype} has a version for more general weights.  In our treatment of low-lying zeros we stick to this special case for simplicity.}\end{footnote} and $\mathcal{S}_k(N)^{\#}$ consists of eigenfunctions of all Hecke operators at \textit{all} $p \nmid N$ (in contrast to $\mathcal{S}_k(N)^*$ above).  The distribution of the low-lying zeros is then described as follows:

\begin{theorem}\label{thm:low-lying-zeros}  Let $\Phi : \mathbb{R} \to \mathbb{R}$ be an even Schwartz function such that the Fourier transform $\widehat{\Phi}(t) = \int_\mathbb{R} \Phi(x) e^{-2\pi i x t} dx$ has compact supported contained in $[-\alpha, \alpha]$ where $\alpha < 2/9$.  Then
\[\lim_{k + N \rightarrow \infty} D(N, k; \Phi) = \int_\mathbb{R} \Phi(x) W(\Sp)(x) dx\]
as $k$ varies over even integers and $N$ varies over squarefree\begin{footnote}{Note that this was not assumed when considering the local equidistribution and independence.}\end{footnote} positive integers, and where $W(\Sp)$ is the kernel for symplectic symmetry
\[W(\Sp)(x) = 1 - \frac{\sin 2 \pi x}{2 \pi x}.\]\end{theorem}

\noindent The proof of Theorem \ref{thm:low-lying-zeros} is a fairly standard exercise, combining Theorem \ref{local-equidistribution-theorem-prototype} and explicit formulas for $L$-functions.  The first thing to notice about the result is that there is no restriction to newforms, so representations are counted with multiplicity, as in \cite{ShinTemplier2013} (this means that we must take our conductor $C_{k, N}$ to be a log-average one).  Once again we see the effect of the weight $\omega_{f, N, k}$, as \cite{ShinTemplier2013} (Theorem 1.5/11.5) shows that these low lying zeros with constant weight exhibit even orthogonal symmetry (in the weight or level aspect).\\

\noindent Another noteworthy feature of Theorem \ref{thm:low-lying-zeros} is the contribution of Saito--Kurokawa lifts at ramified primes, which does not appear in the work of \cite{KowalskiSahaTsimerman2012} (where there are no ramified primes) or \cite{ShinTemplier2013} (where transfer to $\GL_4$ is assumed, and thus the Saito--Kurokawa forms are not present because their transfer to $\GL_4$ is not cuspidal).  The point is that Saito--Kurokawa lifts do not satisfy the Ramanujan conjecture.  At unramified primes their contribution is handled already in Theorem \ref{local-equidistribution-theorem-prototype}, but at ramified primes we must show that their contribution in the explicit formula calculation can be neglected.  In order to get a handle on these exceptional cases we restrict to square-free level.  After doing so we prove that a cusp form which violates the Ramanujan conjecture at a single ramified prime gives rise to a vector in the same representation as that of a classical Saito--Kurokawa lift.  It is well-known that Saito--Kurokawa lifts are few amongst all Siegel cusp forms, but we require a quantitative estimate of how few they are when counted with the weight $\omega_{f, N, k}$.  To achieve this we combine classical and representation theoretic methods to show that the Saito--Kurokawa contribution can be neglected as desired.  We remark that, as suggested by the previous paragraph, our treatment of Saito--Kurokawa lifts is not restricted to newforms (as is often the case in the literature).    \\ 

\noindent  The layout of the paper is as follows:  after collecting notations in \S\ref{notation} we recall the adelization of Siegel modular forms and discuss the cuspidal automorphic representations attached to Siegel modular forms of weight $k$ and level $N$.  This discussion follows \cite{Saha2013}, which we refer to for proofs.  With enough notation in place we are then able in \S\ref{the-equidistribution-problem} to set up the equidistribution problem precisely and state the main result (Theorem \ref{local-equidistribution-theorem-qualitive}).  The first main tool to prove this is the theorem of Sugano (Theorem \ref{sugano-formula}) which explicitly relates the values of certain continuous functions on the space $Y_p$, evaluated at the Satake parameters of an unramified local representation $\pi_p$, to the values of the spherical vector in the local Bessel model of $\pi_p$.  In the case when these local components come from the representation attached to a Siegel cusp form $f$, a computation with the global Bessel model shows that these values of the spherical vector are in turn related to a certain sum of Fourier coefficients of the form $f$.  Then \S\ref{bessel-models} is devoted to explaining these two ingredients.  The other main tool is the subject of \S\ref{sctn:fourier-coeff-bound}, in which we bound the sum of Fourier coefficients we obtain in \S\ref{bessel-models} by bounding the Fourier coefficients of Poincar\'{e} series of weight $k$ and level $N$.  The quantitative estimates obtained therein are the key to obtaining a quantitative version of the local equidistribution statement, and the main new ingredient here is that we obtain an estimate that decreases with respect to \textit{both} the weight and level simultaneously.  We then combine the result of this technical computation with the theory of \S\ref{bessel-models} to obtain both the qualitative and quantitative versions of the main result in \S\ref{main-theorem}.  In the final sections \S\ref{sctn:background-on-l-functions}-\S\ref{sctn:low-lying-zeros} we treat Theorem \ref{thm:low-lying-zeros}.  First we describe the relevant L-function theory in \S\ref{sctn:background-on-l-functions} and the set-up for the low-lying zeros in more detail.  In \S\ref{sctn:saito-kurokawa-lifts} handles the Saito--Kurokawa contribution at ramified primes, assuming squarefree level, as described above.  The proof of Theorem \ref{thm:low-lying-zeros} occupies \S\ref{sctn:low-lying-zeros}.\\

\noindent \textbf{Acknowledgements.}  The author would like to thank Abhishek Saha for suggesting this topic and for his help with various technical questions.  We also thank the referee for their comments and careful reading of the manuscript, and for suggesting a better proof of Lemma \ref{measure-well-defined}.

\section{Notation}\label{notation}

The algebraic group $\GSp_4$ is defined as
\[\GSp_4 = \{g \in \GL_4;\: {}^t g J g = \lambda(g) J\text{ for some }\lambda(g) \in \GL_1\},\]
where
\[J = \left(\begin{matrix}  &  -1_2 \\  1_2 & \end{matrix}\right).\]
Throughout we write $G = \GSp_4$.  If $R$ is a subring of $\mathbb{R}$, we let $G^+(R)$ be the subgroup of $G(R)$ consisting of those $g$ with $\lambda(g) > 0$.  $\lambda : G \to \GL_1$ is a homomorphism, and the kernel is by definition $\Sp_4$.  The centre $Z_G$ of $G$ consists of the scalar matrices in $G$.  We often write an element of $G$ in block matrix form as $\left(\begin{smallmatrix} A & B \\ C & D \end{smallmatrix}\right)$.  \\

\noindent  For a ring $R$, $R^{n \times n}$ denotes the set of $n \times n$ matrices over $R$, and $R_{\text{sym}}^{n \times n}$ the subset of symmetric ones.  We say a matrix $S = (s_{ij}) \in \mathbb{Q}_{\text{sym}}^{n \times n}$ is semi-integral if $s_{ij} \in \frac{1}{2}\mathbb{Z}$ for all $i, j$ and $s_{ii} \in \mathbb{Z}$ for all $i$.  \\

\noindent We say a subgroup $\Gamma$ of $\Sp_4(\mathbb{Q})$ is a congruence subgroup if there exists an integer $N$ such that $\Gamma$ contains
\[\Gamma(N) = \left\{g \in \Sp_4(\mathbb{Z});\:g \equiv 1_4 \bmod N\right\}\]
as a subgroup of \textit{finite index}.  \\

\noindent Let $\mathbb{H}_2 = \{Z \in \mathbb{C}_{\text{sym}}^{2 \times 2};\: \Imm(Z)>0\}$ be the Siegel upper half space of degree $2$.  There is an action of $G^+(\mathbb{R})$ on $\mathbb{H}_2$, namely 
\[\left(\gamma, Z\right) \mapsto \gamma \langle Z\rangle = (AZ + B)(CZ+D)^{-1},\]
where $\gamma = \left(\begin{smallmatrix} A & B \\ C & D \end{smallmatrix}\right) \in G^+(\mathbb{R})$.  Let $\Gamma \subset \Sp_4(\mathbb{Q})$ be a congruence subgroup.  For $k$ a positive integer, $\mathcal{S}_k(\Gamma)$ denotes the space of Siegel cusp forms of degree $2$ and weight $k$ for $\Gamma$; that is, the space of holomorphic functions $f$ on $\mathbb{H}_2$ such that
\[f(\gamma \langle Z\rangle ) = j(\gamma, Z)^{k} f(Z)\]
for all $\gamma \in \Gamma$ (where the automorphy factor is given by $j\left(\left(\begin{smallmatrix} A & B \\ C & D \end{smallmatrix}\right), Z\right) = \det(CZ + D)$) as well as vanishing at all cusps (equivalently satisfying a \textit{moderate growth} condition).  We will be particularly interested in the case when $\Gamma = \Gamma_0(N)$, where $N$ is a positive integer, which is the congruence subgroup consisting of those matrices $\left(\begin{smallmatrix} A & B \\ C & D \end{smallmatrix}\right) \in \Sp_4(\mathbb{Z})$ such that $C \equiv 0 \bmod N$.  \\

\noindent We write $\mathcal{S}_k = \bigcup_{\Gamma} \mathcal{S}_k(\Gamma)$, where the union is over all congruence subgroups.  Equivalently, $\mathcal{S}_k = \bigcup_{N \geq 1} \mathcal{S}_k(\Gamma(N))$.\\

\noindent An element $f \in \mathcal{S}_k(\Gamma)$ possesses a Fourier expansion of the form
\[f(Z) = \sum_{T} a(T; f) e(\tr(TZ))\]
where $e(z) = e^{2 \pi i z}$, and the matrices $T$ are positive semi-definite elements of $\mathbb{Q}^{2 \times 2}_{\text{sym}}$ constrained to lie in some lattice, which depends on $\Gamma$.  In the case of $\Gamma = \Gamma_0(N)$ this lattice is simply the lattice of semi-integral matrices in $\mathbb{Q}^{2 \times 2}_{\text{sym}}$. \\

\noindent For $f, g \in \mathcal{S}_k(\Gamma)$ we define the Petersson inner product
\begin{equation}\label{petersson-definition} \langle f, g \rangle = \frac{1}{\vol(\Gamma \backslash \mathbb{H}_2)} \int_{\Gamma \backslash \mathbb{H}_2} f(Z) \overline{g(Z)} \det(Y)^{k-3} dX dY,\end{equation}
where $Z = X + iY$.  Note that if $\Gamma'$ is a subgroup of $\Gamma$ then $f, g \in \mathcal{S}_k(\Gamma')$ as well, so the modularity group $\Gamma$ is not uniquely determined.  With our normalization, however, the inner product $\langle f, g\rangle$ is independent of the choice of $\Gamma$.  \\ 

\noindent If $S$ is a finite set of primes and $N$ is a positive integer, we write $\gcd(N, S)=1$ to mean that no prime factor of $N$ lies in $S$.

\section{The representation attached to a Siegel modular form}\label{reps-general}

\textbf{Adelization.}  We begin by recalling the construction of the adelization of a Siegel cusp form.  Our setup follows that of \cite{Saha2013}, which we refer to for proofs.  Let $f \in \mathcal{S}_k$, so there is a positive integer $N$ such that $f \in \mathcal{S}_k(\Gamma(N))$.  Define, for each prime $p$, a compact open subgroup $K_p^N$ of $G(\mathbb{Z}_p)$ by
\begin{equation}\label{open-compact-def} K_p^N = \left\{g \in G(\mathbb{Z}_p);\: g \equiv \left(\begin{matrix} 1_2 & \\ & a1_2 \end{matrix}\right) \bmod N\mathbb{Z}_p, a \in \mathbb{Z}_p^{\times} \right\}.\end{equation}
Note that $K_p^N = G(\mathbb{Z}_p)$ for all primes $p \nmid N$, and that the multiplier map $\lambda : G(\mathbb{Z}_p) \to \mathbb{Z}_p^\times$ is surjective for every prime $p$.  Thus strong approximation applies, hence
\[G(\mathbb{A}) = G(\mathbb{Q}) G^+(\mathbb{R}) \prod_{p < \infty} K_p^N.\]
Define the adelization $\Phi_f$ of $f$ by
\[\Phi_f(g_\mathbb{Q} g_\infty h) = \lambda(g_\infty)^{k} j(g_\infty, i1_2)^{-k} f(g_\infty\langle i1_2\rangle)\]
where $g_\mathbb{Q} \in G(\mathbb{Q})$, $g_\infty \in G^+(\mathbb{R})$, $h \in \prod_{p<\infty} K_p^N$.  Since $G(\mathbb{Q}) \cap G^+(\mathbb{R})\prod_p K_p^N = \Gamma(N)$, the modularity of $f$ implies that $\Phi_f$ is well-defined.  Moreover, one can easily check that $\Phi_f$ is independent of the choice of $N$ made in its construction.  \\

\noindent  The map $f \mapsto \Phi_f$ injectively assigns to each degree $2$ Siegel modular form a function on $G(\mathbb{A})$.  Immediately from the definition it is clear that $\Phi_f(g_\mathbb{Q} g) = \Phi_f(g)$ for all $g_\mathbb{Q} \in G(\mathbb{Q})$, $g \in G(\mathbb{A})$.  Also $\abs{\Phi_f(g)}^2$ is invariant under the centre $Z_G(\mathbb{A})$, and so we can form the following integral, which will in fact be finite by the moderate growth of $f$:
\[\int_{Z_G(\mathbb{A})G(\mathbb{Q}) \backslash G(\mathbb{A})} \abs{\Phi_f(g)}^2 dg < \infty.\]
For the remaining constraints on the image $V_k \subset L^2(G(\mathbb{Q}) \backslash G(\mathbb{A}))$ (as usual we mean square-integrable modulo $Z(\mathbb{A})$) we refer to \cite{Saha2013}; and in particular Theorem 1 which shows that the map $\mathcal{S}_k \to V_k$ is an isometry of vector spaces, the inner product on $\mathcal{S}_k$ being (\ref{petersson-definition}). \\

\noindent \textbf{Hecke operators}:  Fix a prime $p$.  We say that $f \in \mathcal{S}_k$ is $p$-spherical if there exists $N$ such that $p \nmid N$ and $f \in \mathcal{S}_k(\Gamma(N))$.  For any $N$ with $p \nmid N$ we define $\mathcal{H}_{p, N}$, the (classical) local Hecke algebra at $p$, to be the ring of $\mathbb{Z}$-linear combinations of double cosets $\Gamma(N)M\Gamma(N)$ where
\[M \in \Delta_{p, N} = \left\{g \in G^{+}(\mathbb{Z}[p^{-1}]);\:g \equiv \left(\begin{matrix} 1_2 &  \\  & \lambda(g)1_2 \end{matrix}\right) \bmod N\right\}.\]
The multiplication is defined in the usual manner for a Hecke algebra.  We will abbreviate $\mathcal{H}_p := \mathcal{H}_{p, 1}$.

\begin{lemma}\label{classical-hecke-algebras}  The ring $\mathcal{H}_{p}$ is commutative and genereated by 
\[\begin{aligned} T(p) &= \Gamma(1)\left(\begin{matrix} 1_2 &  \\ & p1_2 \end{matrix}\right)\Gamma(1), \\
T_1(p^2) &= \Gamma(1)\left(\begin{matrix} 1 &  &  &  \\  & p &  &  \\  &  & p^2 &  \\  &  &  & p \end{matrix}\right)\Gamma(1).\end{aligned}\]
Moreover, for any $p \nmid N$, the natural map $\iota_{p, N} : \mathcal{H}_{p, N} \to \mathcal{H}_{p}$ defined by $\Gamma(N) M \Gamma(N) \mapsto \Gamma(1) M \Gamma(1)$ is an isomorphism.
\end{lemma}
\begin{proof}  See \cite{AndrianovZhuravlev1995} Lemma 3.3, Theorem 3.7, and Theorem 3.23.  \end{proof}  

\begin{remark}\label{hecke-for-different-congruence-subgroups}  The isomorphism of Lemma \ref{classical-hecke-algebras} also holds with the modified Hecke algebra $\widetilde{\mathcal{H}}_{p, N}$, defined to be the Hecke algebra generated by $\Gamma_0(N)M\Gamma_0(N)$ with $M$ lying in the $\Gamma_0(N)$-analogue of $\Delta_{p, N}$.  This follows since $\Gamma_0(N)$ satisfies the ``$q$-symmetry condition'' of \cite{AndrianovZhuravlev1995} -- see Lemma 3.5 there.  We thus identify $\widetilde{\mathcal{H}}_{p, N}$ with $\mathcal{H}_{p, N}$ (when $p \nmid N$).  In particular we will use the notation $T(p)$ and $T_1(p^2)$ for the standard Hecke operators on $\mathcal{S}_k(N)$.  \end{remark}

\noindent  Let $f \in \mathcal{S}_k$ be $p$-spherical, say $f \in \mathcal{S}_k(\Gamma(N))$ with $p \nmid N$.  The action of $T = \Gamma(1)M\Gamma(1) \in \mathcal{H}_{p}$ is defined as follows:  write $\iota_{p, N}^{-1}(\Gamma(1) M \Gamma(1)) = \bigsqcup_i \Gamma(N)M_i$, and set
\[(f_k|T)(Z) = \sum_{i} \lambda(M_i)^{k}j(M_i, Z)^{-k}	f(M_i \langle Z\rangle).\]
This extends by linearity to a right-action of $\mathcal{H}_{p}$ on the set of $p$-spherical elements of $\mathcal{S}_k$.  The composition of operators agrees with the multiplication in $\mathcal{H}_{p}$.  Thus by Lemma \ref{classical-hecke-algebras} the action is commutative, and determined by the action of $T(p)$ and $T_1(p^2)$. \\

\noindent  Continue with $p$ a fixed prime, the adelic counterpart to $\mathcal{H}_p$ is $\mathfrak{h}_p$, the set of locally constant compactly supported functions on $G(\mathbb{Q}_p)$ which are both left and right invariant by $G(\mathbb{Z}_p)$, equipped with convolution product.  It acts on $p$-spherical elements $\Phi \in V_k$, that is those $\Phi \in V_k$ such that $\Phi(gh) = \Phi(g)$ for all $h \in G(\mathbb{Z}_p)$ (so that by definition $\Phi \in V_k$ is $p$-spherical for almost all $p$).  There is a canonical map $\mathcal{H}_p \to \mathfrak{h}_p$ defined to be the $\mathbb{Z}$-linear extension of the map $\Gamma(1)M\Gamma(1) \mapsto \mathbf{1}_{G(\mathbb{Z}_p) M G(\mathbb{Z}_p)}$.  Denoting the image of an arbitrary element $T \in \mathcal{H}$ by $\widetilde{T}$, the map $T \mapsto \widetilde{T}$ is an isomorphism of rings $\mathcal{H}_p \otimes \mathbb{C} \to \mathfrak{h}_p$.  Furthermore, the map $f \mapsto \Phi_f$ restricts to a map between $p$-spherical elements, and is Hecke equivariant in the sense that
\[\Phi_{f | T} = \widetilde{T} \Phi_f\]
for any $p$-spherical $f \in \mathcal{S}_k$ and $T \in \mathcal{H}_p$.  Again we refer to \cite{Saha2013} for a proof of these facts. \\

\noindent \textbf{The automorphic representation corresponding to $f$}:  Let $f \in \mathcal{S}_k$, and let $\Phi_f \in V_k$ be its adelization.  Letting $G(\mathbb{A})$ act on $\Phi_f$ by the right regular action $\Phi_f(g) \mapsto \Phi_f(gh)$ for $h \in G(\mathbb{A})$ we generate a cuspidal automorphic representation $\pi_f$ of $G(\mathbb{A})$.  As usual this decomposes as a direct sum of finitely many irreducible cuspidal automorphic representations of $G(\mathbb{A})$, say 
\begin{equation}\label{decompose-pi-f-into-irreps} \pi_f = \bigoplus_{i=1}^m \pi_f^{(i)}\end{equation} 
with each $(\pi_f^{(i)}, V^{(i)})$ irreducible.\begin{footnote}{Such a decomposition need not be unique, since an irreducible constituent may occur with multiplicity greater than one.  However, it is expected that weak multiplicity one holds, which would rule this possibility out.  We do not need to assume anything about the uniqueness of this decomposition, since the local components we are interested in will always turn out to be isomorphic.}  \end{footnote}  Let $\pi = \pi_f^{(i)}$ be any irreducible constituent of $\pi_f$.  By the tensor product theorem there exist irreducible, unitary, admissible representations $\pi_v$ of $G(\mathbb{Q}_v)$ (one for each place $v$ of $\mathbb{Q}$) such that
\begin{equation}\label{tensor-product-isomorphism} \pi \simeq \otimes'_v \pi_v,\end{equation}
where the prime denotes a restricted tensor product, and for almost all $v$ the local representation $\pi_v$ is spherical.  Since $f \in \mathcal{S}_k$ the archimedean component $\pi_\infty$ is a certain lowest weight representation as described in \cite{AsgariSchmidt2001}.  Similarly, the following proposition describes $\pi_p$ when $f$ is an eigenfunction for $\mathcal{H}_{p}$:

\begin{proposition} \label{hecke-evals-determine-spsr}  Let $p$ be a prime and suppose $f \in \mathcal{S}_k$ is $p$-spherical.  Assume moreover that $f$ is an eigenfunction for the Hecke operators $T(p)$ and $T_1(p^2)$ (and hence, by Lemma \ref{classical-hecke-algebras}, an eigenfunction for $\mathcal{H}_{p}$), with corresponding eigenvalues $\lambda(p)$ and $\lambda_1(p^2)$.  Then, for any irreducible constituent $\pi$ of $\pi_f$, the local component $\pi_p$ in any isomorphism of the form (\ref{tensor-product-isomorphism}) is a spherical principal series representation\begin{footnote}{We will recall the construction of these representations in \S\ref{the-equidistribution-problem}.}\end{footnote} of $G(\mathbb{Q}_p)$ whose isomorphism class is determined uniquely by $\lambda(p)$ and $\lambda_1(p^2)$. \end{proposition}
\begin{proof} See \cite{Saha2013} Proposition 3.9.  \end{proof}

\begin{remark}\label{well-defined-isom-class-of-spsr}  It follows that there is a well-defined isomorphism class of local representations at $p$ (which are necessarily spherical principal series) attached to a $p$-spherical element $f \in \mathcal{S}_k$ under the assumption that $f$ is an eigenfunction of $T(p)$ and $T_1(p^2)$.  This is well-defined in the sense that it is independent of the (possible) choice of decomposition in (\ref{decompose-pi-f-into-irreps}), the choice of irreducible constituent $\pi = \pi_f^{(i)}$ from this decomposition, and the choice of isomorphism in (\ref{tensor-product-isomorphism}).  \end{remark}

\section{The equidistribution problem}\label{the-equidistribution-problem}

\noindent We now describe in detail the equidistribution problem addressed in the paper.  Fix a finite set of primes $S$, let $k$ be any even integer $\geq 6$, and let $N$ be a positive integer with $\gcd(N, S)=1$.  Let $\mathcal{S}_k(N)^*$ denote any\begin{footnote}{The definitions we make in the following appear to depend on the choice of basis.  However we will show in Lemma \ref{measure-well-defined} that this is not the case.}\end{footnote} orthogonal basis of $\mathcal{S}_k(N)$ consisting of forms that are eigenfunctions of $T(p)$ and $T_1(p^2)$ whenever $p \in S$ (there is no ambiguity in our notation for the Hecke operators -- see Remark \ref{hecke-for-different-congruence-subgroups}).  Let $f \in \mathcal{S}_k(N)^*$.  By Remark \ref{well-defined-isom-class-of-spsr} we can attach to $f$ an isomorphism class of spherical principal series representations of $G(\mathbb{Q}_p)$ for each $p \in S$.  Since $f \in \mathcal{S}_k(N)^*$ has trivial character, the central character of the corresponding representation will be trivial.  \\

\noindent  We now recall the construction of the spherical principal series representations of $G(\mathbb{Q}_p)$ with trivial central character.  Let $\chi_1, \chi_2, \sigma$ be unramified quasi-characters of $\mathbb{Q}_p^{\times}$, and define a character of the Borel subgroup
\[\left(\begin{matrix} a_1 & * & * & * \\   & a_2 & * & * \\  &  & \lambda a_1^{-1} &  \\  &  & * & \lambda a_2^{-1} \end{matrix}\right) \mapsto \chi_1(a_1) \chi_2(a_2) \sigma(\lambda).\]
We require the central character to be trivial, so $\chi_1\chi_2\sigma^2=1$.  Via normalized induction we obtain a representation of $G(\mathbb{Q}_p)$, and this has a unique spherical constituent, denoted $\chi_1 \times \chi_2 \rtimes \sigma$ as in the notation of \cite{SallyTadic1993}.  Since the quasi-characters $\chi_1, \chi_2, \sigma$ are unramified they are completely determined by their values on $p \in \mathbb{Q}_p^\times$.  Since the central character is trivial, $\chi_1 \times \chi_2 \rtimes \sigma$ is therefore determined by $(a, b) = (\sigma(p), \sigma(p)\chi_1(p)) \in \mathbb{C}^\times \times \mathbb{C}^\times$.  We refer to $(a, b)$ as the \textit{Satake parameters} of $\chi_1 \times \chi_2 \rtimes \sigma$.  By the classification in \cite{PitaleSchmidt2009}, it follows that $0 < \abs{a}, \abs{b} \leq \sqrt{p}$.  The form of the generalized Ramanujan conjecture for $\GSp_4$ proved by Weissauer (see \cite{Weissauer2009}) states that if the global representation $\pi$ is not CAP then in fact $\abs{a}=\abs{b} = 1$.  We will discuss this further in \S\ref{sctn:background-on-l-functions} and \S\ref{sctn:saito-kurokawa-lifts}. \\

\noindent  Any spherical principal series representation of $G(\mathbb{Q}_p)$ with trivial central character is isomorphic to some $\chi_1 \times \chi_2 \rtimes \sigma$.  Moreover, the representations $\chi_1 \times \chi_2 \rtimes \sigma$ and $\chi_1' \times \chi_2' \rtimes \sigma'$, with associated $(a, b)$ and $(a', b')$ respectively, are isomorphic if and only if $(a, b)$ and $(a', b')$ lie in the same orbit under the action of the Weyl group $W$ of order $8$ generated by the transformations
\begin{equation}\label{weyl-group-generators} (a, b) \mapsto (b, a), \:\:(a, b) \mapsto (a^{-1}, b),\:\:(a,b) \mapsto (a, b^{-1}).\end{equation}

\noindent  Let $X_p$ be the set of isomorphism classes of spherical principal representations of $G(\mathbb{Q}_p)$.  Let $Y_p = \{(a, b) \in \mathbb{C}^\times \times \mathbb{C}^\times;\: 0 < \abs{a}, \abs{b} \leq \sqrt{p}\}/W$.  Then we have a well-defined injection $X_p \to Y_p$.  $Y_p$ therefore provides a natural choice of co-ordinates on $X_p$.  Fix a finite set of primes $S$.  We also form the product spaces
\begin{equation}\label{XS-YS-definition} X_S = \prod_{p \in S} X_p, \:\:\:\:Y_{S} = \prod_{p \in S} Y_p.\end{equation}
We form the natural injection $X_S \to Y_S$, which allows us to view $X_S$ as a subspace of $Y_S$.  We will formulate our equidistribution problem on $Y_S$, doing so by defining two measures $\nu_{S, N, k}$ and $\mu_S$ on $Y_S$ and showing that these agree in an appropriate weak-$*$ limit.  The measure $\mu_S$ is a certain natural measure on $Y_S$.  The measure $\nu_{S, N, k}$ reflects the distribution of the spherical principal series representations attached to eigenforms in $S_k(N)$.\\

\noindent  \textbf{The measure $\nu_{S, N, k}$.}   As mentioned in the introduction, our distribution will be weighted by a certain ``arithmetic factor''; our first task is to define this.  Let $k \geq 6$ be even and $N$ a positive integer with $\gcd(N, S)=1$.  Let $d$ be a positive integer such that $-d$ is the discriminant of $\mathbb{Q}(\sqrt{-d})$.  Let $w(-d)$ denote the number of roots of unity in $\mathbb{Q}(\sqrt{-d})$.  Let $\Cl_d$ denote the ideal class group of $\mathbb{Q}(\sqrt{-d})$, and let $\Lambda$ be any character of $\Cl_d$.  Recall the isomorphism between $\Cl_d$ and the set of $\SL_2(\mathbb{Z})$ equivalence classes of primitive, semi-integral, positive definite matrices with determinant $d/4$.  We write this map from $\Cl_d$ to the set of (equivalence classes of) such matrices as $c \mapsto \mathsf{S}_c$.  Define
\[c_{k}^{d, \Lambda} = \sqrt{\pi} (4\pi)^{3-2k} \Gamma\left(k - \frac{3}{2}\right) \Gamma(k-2) \left(\frac{d}{4}\right)^{-k+\frac{3}{2}}\frac{d_\Lambda}{w(-d)\abs{\Cl_d}},\]
where
\[d_\Lambda = \begin{cases} 1 & \text{if }\Lambda^2 = 1, \\ 2 & \text{otherwise.} \end{cases}\]
Define also
\begin{equation}\label{fourier-coeff-class-gp-average} a(d, \Lambda; f) = \sum_{c \in \Cl_d} \overline{\Lambda(c)} a(\mathsf{S}_c; f),\end{equation}
which is well-defined since the Fourier coefficients $a(T; f)$ depend only on the equivalence class of $T$ modulo $\SL_2(\mathbb{Z})$-conjugation (the same is even true for $\GL_2(\mathbb{Z})$-conjugation, since $k$ is even).  The weight\begin{footnote}{When $N=1$ this is the weight used in \cite{KowalskiSahaTsimerman2012}, though one must recall that we normalize our Petersson inner products differently.}\end{footnote} we use is
\begin{equation}\label{weight-def} \omega_{f, N, k}^{d, \Lambda} = \frac{c_k^{d, \Lambda}}{ \vol(\Gamma_0(N) \backslash \mathbb{H}_2)}\frac{\abs{a(d, \Lambda; f)}^2}{\langle f, f \rangle}.\end{equation}

\noindent Recall that the Petersson inner product, defined by (\ref{petersson-definition}), is independent of the choice of congruence subgroup.  The dependence on $N$ is therefore solely via $\vol(\Gamma_0(N) \backslash \mathbb{H}_2)$, in the sense that if $f \in \mathcal{S}_k(N) \subset \mathcal{S}_k(NN_1)$, then
\[\omega_{f, NN_1, k}^{d, \Lambda} = \frac{\vol(\Gamma_0(N) \backslash \mathbb{H}_2)}{\vol(\Gamma_0(NN_1) \backslash \mathbb{H}_2)}\omega_{f, N, k}^{d, \Lambda}.\]
The asymptotics as a function of $N$ is therefore determined by the index of $\Gamma_0(N)$ inside $\Sp_4(\mathbb{Z})$, and one can easily check $\vol(\Gamma_0(N) \backslash \mathbb{H}_2) \asymp [\Gamma_0(N) : \Sp_4(\mathbb{Z})] \asymp N^3$.  The dependency on $k$ is already explicit from the form of $c_k^{d, \Lambda}$.  \\

\noindent  A more subtle point is the dependency of this weight on $f$.  In the parlance of general equidistribution problems from the introduction we have chosen the quadratic form $Q$ to be 
\[f \mapsto \frac{c_{k}^{d, \Lambda} \abs{a(d, \Lambda; f)}^2}{\vol(\Gamma_0(N) \backslash \mathbb{H}_2)}.\]  
It is believed that the term $\abs{a(d, \Lambda; f)}^2$ carries deep arithmetic information: when $f$ is an eigenform, a conjecture of B\"{o}cherer relates this quantity to the central value $L(1/2, \pi_f \times \chi_{-d})$ of the Langlands $L$-function $L(s, \pi_f \times \chi_{-d})$, where $\chi_{-d}$ is the character corresponding to the imaginary quadratic extension $\mathbb{Q}(\sqrt{-d})$.  This deep conjecture can be viewed as an analogue of Waldspurger's famous theorem in the case of elliptic modular forms.  To the best of the author's knowledge this has only been proved for certain ``lifts'' (e.g. Saito--Kurokawa and Yoshida lifts).  \\

\noindent  In our investigation of the asymptotics of this measure we will work with a fixed but arbitrary choice of $d$ and $\Lambda$.  Consequently we will often abbreviate $\omega_{f, N, k}^{d, \Lambda}$ to $\omega_{f, N, k}$.  The limiting distribution, $\mu_S$ defined below, will also depend on the choice of $d, \Lambda$.  In order to simplify notation one may wish to focus on the simplest case, which is $d=4$ and $\Lambda=\mathbf{1}$, giving the weight used in Theorem \ref{local-equidistribution-theorem-prototype}.  We will also restrict to this weight in \S8-\S10.\\

\noindent  With $d$ and $\Lambda$ fixed, now fix $S$ and form $\mathcal{S}_k(N)^*$ as we did at the beginning of this section.  To each $f \in \mathcal{S}_k(N)^*$ we have associated a tuple $\pi_S(f) = (\pi_p(f))_{p \in S}$, where each $\pi_p(f)$ is an isomorphism class of spherical principal series representations of $G(\mathbb{Q}_p)$.  We also write $\pi_S(f) \in Y_S$ for the image of this tuple under the map $X_S \hookrightarrow Y_S$.  The measure $\nu_{S, N, k}$ on $Y_S$, which is supported on (the image of) $X_S$, is defined by
\begin{equation}\label{nu-def} \nu_{S, N, k} = \sum_{f \in \mathcal{S}_k(N)^*} \omega_{f, N, k} \delta_{\pi_S(f)},\end{equation}
where $\delta$ denotes Dirac mass.\\

\noindent In a moment we will compare this with the general equidistribution set up in the introduction.  First we prove, in this generality, that the measure is independent of the choice of basis:

\begin{lemma}\label{measure-well-defined}  Let $X$ be a topological space, $V$ a finite dimensional complex inner product space, $Q$ a fixed non-negative hermitian form on $V$, and $H$ a finitely generated commutative algebra of hermitian operators acting on $V$.  Suppose that whenever $v \in V$ is an eigenvector for $H$ is has associated to it a point $a(v) \in X$ such that if $v_1, v_2$ lie in the same eigenspace then $a(v_1) = a(v_2)$.  For each $v \in V$ let $\omega(v) = Q(v)/\langle v, v \rangle$.  For each orthogonal basis $\mathcal{B}$ of $V$ consisting of eigenforms of $H$ define a measure $X$ by $\nu_\mathcal{B} = \sum_{v \in \mathcal{B}} \omega(v) \delta_{a(v)}$.  Then $\nu_\mathcal{B}$ is independent of the choice of $\mathcal{B}$.\end{lemma}
\begin{proof}  $V$ can be written as a direct sum of $H$-eigenspaces, different eigenspaces necessarily being orthogonal, and hence we reduce to the case when all $v \in V$ have the same $a(v)$.  Let $A$ denote the linear operator such that $Q(v) = \langle Av, v\rangle$.  Take a function $f : X \to \mathbb{C}$, then
\[\int f d\nu_\mathcal{B} = f(a) \sum_{v \in \mathcal{B}} \frac{Q(v)}{\langle v, v\rangle} = f(a) \sum_{v \in \mathcal{B}} \frac{\langle Av, v\rangle}{\langle v, v\rangle} = f(a)\tr(A),\]
which is independent of $\mathcal{B}$.  Thus $\nu_\mathcal{B}$ is independent of $\mathcal{B}$.\begin{footnote}{This proof actually shows how we can define $\nu_\mathcal{B}$ without picking a basis: namely $\nu_\mathcal{B} := \sum_E \tr(A_E) \delta_{a(E)}$ where $E$ varies over the distinct $H$-eigenspaces, $A_E$ is the operator representing $Q$ restricted to $E$, and $a(E)=a(v)$ for any $v \in E$.}\end{footnote} \end{proof}

\noindent For our present situation in the notation of the lemma we have:
\begin{myitem}
\item  the topological space $X$ is $Y_S$,
\item  the finite dimensional vector space $V$ is $\mathcal{S}_k(N)$, equipped with the Petersson inner product,
\item  the algebra of operators $H$ consists of the local Hecke algebras for $p \in S$,
\item  the point $a(f) \in X$ for $f \in V$ is the Satake parameters of the local representation, $\pi_S(f)$, 
\item  the quadratic form $Q$ is $f \mapsto c_k^{d, \Lambda} \abs{a(d, \Lambda; f)}^2 / \vol(\Gamma_0(N) \backslash \mathbb{H}_2)$.
\end{myitem}
 
\noindent \textbf{The measure $\mu_S$.}  \noindent  Our limiting measure is the measure referred to in \cite{FurusawaShalika2002} as \textit{the Plancherel measure for the local Bessel model associated to $(d, \Lambda)$}.  In \cite{KowalskiSahaTsimerman2012} is appears as the limiting measure for $\nu_{S, 1, k}$ as $k \rightarrow \infty$ over even integers.  We follow this paper for our definition now.  Let 
\[I_p = \{(a, b) \in \mathbb{C}^\times \times \mathbb{C}^\times;\: \abs{a} = \abs{b} = 1\}/W \subset Y_p,\] 
where $W$ is the Weyl group generated by (\ref{weyl-group-generators}).  We write (a representative of) the point $(a, b) \in I_p$ using the co-ordinates $(a, b) = (e^{i\theta_1}, e^{i\theta_2})$.  We define a measure $d\widetilde{\mu}_p$ on $I_p$ by 
\[d\widetilde{\mu}_p(\theta_1, \theta_2) = \frac{4}{\pi^2}(\cos(\theta_1)-\cos(\theta_2))^2 \sin^2(\theta_1) \sin^2(\theta_2) \: d\theta_1 d\theta_2.\]
Note that this is independent of the choice of representative $(e^{i\theta_1}, e^{i\theta_2})$.  This can be obtained as a pushforward of the probability Haar on $\USp_4$ (the compact form of $\Sp_4$) to $I_p$, in analogy with the construction of the classical Sato--Tate measure.  We extend $\widetilde{\mu}_p$ to a measure on $Y_p$, also denoted $\widetilde{\mu}_p$, by extending by zero.  The measure $\mu_p = \mu_{p, d, \Lambda}$ is now defined by
\[d\mu_{p} = \left(1 - \left(\frac{-d}{p}\right)\frac{1}{p}\right) \Delta_{p, d, \Lambda}^{-1} d\widetilde{\mu}_p.\]
This measure is also supported on $I_p \subset Y_p$.  The function $\Delta_{p, d, \Lambda}$ is given by
\[\Delta_{p, d, \Lambda}(\theta_1, \theta_2) = \prod_{i=1}^2 \begin{cases} \left(\left(1+\frac{1}{p}\right)^2 - \frac{4\cos^2(\theta_i)}{p}\right) & \text{if 
}p\text{ is inert in }\mathbb{Q}(\sqrt{-d}), \\
\left(\left(1 - \frac{1}{p}\right)^2 + \frac{1}{p}\left(2\cos(\theta_i)\sqrt{p} - \lambda_p\right)\left(\frac{2\cos(\theta_i)}{\sqrt{p}} - \lambda_p\right)\right) & \text{if }p\text{ is split in }\mathbb{Q}(\sqrt{-d}), \\
\left(1-\frac{2\lambda_p\cos(\theta_i)}{\sqrt{p}} + \frac{1}{p}\right) & \text{if }p\text{ is ramified in }\mathbb{Q}(\sqrt{-d}), \end{cases}\]
where $\lambda_p = \sum_{N(\mathfrak{p}) = p} \Lambda(\mathfrak{p})$ (a sum over the one or two prime ideals in $\mathbb{Q}(\sqrt{-d})$ of norm $p$).  Note that $\Delta_{p, d, \Lambda}$ is again independent of the choice of Weyl group orbit representative.  Finally, we define the measure $\mu_S = \mu_{S, d, \Lambda}$ on $X_S$ by
\begin{equation}\label{mu-def} d\mu_S = \prod_{p \in S} d\mu_{p}.\end{equation}
Although the definition is rather complicated this measure is at least completely explicit.  Along with the fact that the measure is supported on $I_p$, it is perhaps also worth noting that $d\mu_p$ tends towards the Sato--Tate measure as $p \rightarrow \infty$.  

\begin{theorem}[Local equidistribution and independence, qualitative version]\label{local-equidistribution-theorem-qualitive}  Fix any $d>0$ such that $-d$ is the discriminant of $\mathbb{Q}(\sqrt{-d})$, and let $\Lambda$ be any character of $\Cl_d$.  For any finite set of primes $S$, the measure $\nu_{S, k, N}$ converges weak-$*$ to $\mu_{S}$ as $k+N \rightarrow \infty$ with $k\geq 6$ varying over positive even integers and $N\geq 1$ varying over positive integers with $\gcd(N, S)=1$.  That is, for any continuous function $\varphi$ on $Y_S$,
\[\lim_{k+N \rightarrow \infty} \sum_{f \in \mathcal{S}_k(N)^*} \omega_{f, N, k} \:\varphi((a_p(f), b_p(f))_{p \in S}) = \int_{Y_S} \varphi \:  d\mu_{S}.\]
In particular if $\varphi = \prod_{p \in S} \varphi_p$ is a product function then
\[\lim_{k+N \rightarrow \infty} \sum_{f \in \mathcal{S}_k(N)^*} \omega_{f, N, k} \varphi((a_p(f), b_p(f))_{p \in S}) = \prod_{p \in S} \int_{Y_p} \varphi_p\: d\mu_p.\]\end{theorem}

\noindent The proof of (a quantitative version of) this theorem is the goal of the next three sections.\\

\noindent Before proceeding, let us remark on the cases of low (even) weight which are not covered by Theorem \ref{local-equidistribution-theorem-qualitive} (i.e. $k=2, 4$).  As we shall see, the condition $k \geq 6$ is necessary for absolute convergence of a certain Poincar\'{e} series (required also in related calculations in \cite{KowalskiSahaTsimerman2012} and \cite{ChidaKatsuradaMatsumoto2011}) and is an artefact of our method.  In \cite{KowalskiSahaTsimerman2012} this condition is not an issue as they work in the limit $k \rightarrow \infty$.  However, in our context, the level aspect for fixed small weight is an interesting case which is not addressed by our results.  Note that the weight $k=4$ (for which the $\infty$-type is cohomological) in the level aspect is included in the work of \cite{ShinTemplier2013}.

\section{Bessel models}\label{bessel-models}

\noindent  \textbf{Global Bessel models.}  We begin by recalling the definition of the global Bessel model for a cuspidal representation of $G(\mathbb{A})$ in the fashion of \cite{Furusawa1993}, \cite{KowalskiSahaTsimerman2012}.  Let $\mathsf{S} \in \mathbb{Q}_{\text{sym}}^{2 \times 2}$ be positive definite.\begin{footnote}{It is customary to use $S$ for this matrix in the definition of the Bessel model.  However this clashes with our earlier notation for our finite set of primes.  To minimize confusion we use the standard letter but in sanserif font.}\end{footnote}  Let $\disc(\mathsf{S}) = -4\det(\mathsf{S}) < 0$ and $d = 4 \det(\mathsf{S}) > 0$.  If we write $\mathsf{S} = \left(\begin{smallmatrix} a & b/2 \\ b/2 & c \end{smallmatrix}\right)$, then we define $\xi = \xi_\mathsf{S}$ by
\[\xi = \left(\begin{matrix} b/2 & c \\ -a & -b/2 \end{matrix}\right).\]
Let $L = \mathbb{Q}(\sqrt{-d})$.  We have an isomorphism
\[\mathbb{Q}(\xi) \to L\]
defined by 
\[a + b\xi \mapsto a + b\frac{\sqrt{-d}}{2}.\] 
Now define the algebraic group
\[T = \{g \in \GL_2;\: {}^t g \mathsf{S} g = \det(g)\mathsf{S}\}.\]
A straightforward computation shows that $\mathbb{Q}(\xi)^\times = T(\mathbb{Q})$, and hence we can identity $T(\mathbb{Q})$ with $L^\times$.  We embed $T$ as a subgroup of $G$ via
\begin{equation}\label{embed-gl-in-gsp4} g \mapsto \left(\begin{matrix} g & 0 \\ 0 & \det(g) {}^tg^{-1}\end{matrix}\right).\end{equation}
Define another subgroup of $G$ by
\[U = \left\{u(X) = \left(\begin{matrix} 1_2 & X \\ 0_2 & 1_2 \end{matrix}\right);\: {}^tX = X\right\},\]
and let $R = TU$. \\

\noindent Let $\psi = \prod_v \psi_v$ be a character of $\mathbb{A}$ such that the conductor of $\psi_p$ is $\mathbb{Z}_p$ for all finite primes $p$, $\psi_\infty(x) = e(x)$ for $x \in \mathbb{R}$, and $\psi|_\mathbb{Q} = 1$.  Define a character $\theta$ of $U(\mathbb{A})$ by
\[\theta(u(X)) = \psi(\tr(\mathsf{S}X)).\]  Let $\Lambda$ be a character of $T(\mathbb{A})/T(\mathbb{Q})$ such that $\Lambda|_{\mathbb{A}^\times} = 1$.  Using the above isomorphism we see that this can be thought of as a character of $\mathbb{A}_{L^\times}/L^\times$ such that $\Lambda|_{\mathbb{A}^\times} = 1$.  Define a character $\Lambda \otimes \theta$ of $R(\mathbb{A})$ by $(\Lambda \otimes \theta)(tu) = \Lambda(t)\theta(u)$ for $t \in T(\mathbb{A})$, $u \in U(\mathbb{Q})$. \\

\noindent Now let $\pi$ be a cuspidal representation of $G(\mathbb{A})$ with trivial central character, and let $V_\pi$ be its space of automorphic forms.  For $\Phi \in V_\pi$, we define a function $B_\Phi$ on $G(\mathbb{A})$ by
\begin{equation}\label{bessel-function-def} B_\Phi(g) = \int_{R(\mathbb{Q}) Z_G(\mathbb{A}) \backslash R(\mathbb{A})} \overline{(\Lambda \otimes \theta)(r)} \Phi(rg) dr. \end{equation}
Note that the complex vector space $\mathbb{C}\langle B_\Phi;\: \Phi \in V_\pi \rangle$ is preserved under the right regular action of $G(\mathbb{A})$, since $\{\Phi \in V_\pi\}$ is. \\

\noindent  Consider the case that $\pi = \bigotimes_v \pi_v$ is an \textit{irreducible} cuspidal representation with trivial central character, with space of automorphic forms $V_\pi$.  If $\mathbb{C}\langle B_\Phi; \:\Phi \in V_\pi \rangle$ is nonzero then the representation afforded by the right regular action of $G(\mathbb{A})$ on this space is isomorphic to $\pi$.  We call the resulting representation a \textit{global Bessel model of type $(\mathsf{S}, \theta, \Lambda)$ for $\pi$}. \\

\noindent  \textbf{Local Bessel models.}  Let $\pi$ be an irreducible cuspidal representation of $G(\mathbb{A}$) with trivial central character.  Fix an isomorphism $\pi \simeq \otimes'_v \pi_v$, where the $\pi_v$ are irreducible, unitary, admissible representations of $G(\mathbb{Q}_v)$.  Let $\Omega$ be a finite set of places, containing $\infty$, such that if $p \notin \Omega$ then $\pi_p$ is a spherical principal series representation.  We now describe the local Bessel function on $G(\mathbb{Q}_p)$ associated to $\pi_p$ for $p \notin \Omega$.  From the character data $\Lambda$, $\theta$ for the global Bessel model we have induced characters $\Lambda_p, \theta_p$ of $T(\mathbb{Q}_p)$, $U(\mathbb{Q}_p)$ respectively.  Let $\mathcal{B}$ be the space of locally constant functions $\varphi$ on $G(\mathbb{Q}_p)$ such that
\[\varphi(tug) = \Lambda_p(t) \theta_p(u) \varphi(g),\text{ for }t\in T(\mathbb{Q}_p), u \in U(\mathbb{Q}_p), g \in G(\mathbb{Q}_p).\]
From the results of \cite{NovodvorskiPiatetski-Shapiro1973} we know that there is a unique subspace $\mathcal{B}(\pi_p)$ of $\mathcal{B}$ such that the right regular action of $G(\mathbb{Q}_p)$ on $\mathcal{B}(\pi_p)$ is isomorphic to $\pi_p$.  Let $B_p$ be the unique $G(\mathbb{Z}_p)$-fixed vector in $\mathcal{B}(\pi_p)$ such that $B_p(1_4) = 1$.  As explained in \cite{Furusawa1993}, $B_p$ is completely determined by the values $B_p(h_p(l, m))$ where
\begin{equation}\label{hp-def} h_p(l, m) = \diag(p^{l+2m}, p^{l+m}, 1, p^m)\end{equation}
for $l, m \geq 0$.  The following theorem of Sugano gives a formula for these values:

\begin{theorem}[Sugano, \cite{Sugano1985} p544; see also \cite{Furusawa1993} (3.6)]\label{sugano-formula}  Let $\pi_p$ be a spherical principal series representation of $G(\mathbb{Q}_p)$ with associated parameters $(a, b) = (\sigma(p), \sigma(p)\chi_1(p))$ as described in \S\ref{the-equidistribution-problem}.  Let $B_p$ be the normalized spherical vector in the local Bessel model.  Let $l, m \geq 0$ be integers, and $h_p(l, m) \in G(\mathbb{Q}_p)$ be defined by (\ref{hp-def}).  Then
\[B_p(h_p(l, m)) = p^{-2m - \frac{3l}{2}} U_p^{l, m}(a, b),\]
for $U_p^{l, m}$ given by the coefficients of the power series in \cite{Furusawa1993} (3.6).  The set of functions $\{U_p^{l, m};\: l, m \geq 0\}$ linearly generate a dense subspace of the space $C(Y_p)$ of continuous functions on $Y_p$.  \end{theorem}

\noindent  The point of Theorem \ref{sugano-formula} is, of course, that we have an explicit formula for $B_p(h_p(l, m))$.  The formula is fairly involved (for an exposition in a situation similar to our own, see \cite{Furusawa1993} (3.6)) and so we do not recall it here.  For the proof of our local equidistribution statement we only require two properties, namely the (already stated) fact that that $U_p^{l, m}$ generate (a dense subspace of) $C(Y_S)$ (see \cite{KowalskiSahaTsimerman2012} Proposition 2.7), and our Proposition \ref{convergence-on-U_p} (for which we will refer to \cite{FurusawaShalika2002} or \cite{KowalskiSahaTsimerman2012}).  For our application to low-lying zeros we will also use the formulas for the first few $U_p^{l, m}$, given as follows: as in the definition of $\mu_S$ write $\lambda_p = \sum_{N(\mathfrak{p}) = p} \Lambda(p)$ where $\Lambda$ is our fixed character of $\Cl_d$ and $\mathfrak{p}$ is prime in $\mathbb{Q}(\sqrt{-d})$.  Let $\left(\frac{d}{\cdot}\right)$ be the character of the extension $\mathbb{Q}(\sqrt{-d})/\mathbb{Q}$, which takes the value $1, 0, -1$ on a rational prime $p$ according to whether $p$ is split, ramified, or inert in $\mathbb{Q}(\sqrt{-d})$.  Set 
\[\begin{aligned} \sigma(a, b) &= a + b + a^{-1} + b^{-1}, \\ 
\tau(a, b) &= 1 + ab + a^{-1}b + ab^{-1} + a^{-1}b^{-1}. \end{aligned}\]
Then
\begin{equation}\label{eqn:U_p-definition} \begin{aligned} U_p^{0, 0}(a, b) &= 1, \\
U_p^{1, 0}(a, b) &= \sigma(a, b) -p^{-1/2} \lambda_p, \\
U_p^{2, 0}(a, b) &= a^2 + b^2 + a^{-2} + b^{-2} + 2\tau(a, b) + 2 - p^{-1/2}\lambda_p\sigma(a, b) + p^{-1} \left(\frac{d}{p}\right),\\
U_p^{0, 1}(a, b) &= \tau(a, b) - \left(p - \left(\frac{d}{p}\right)\right)^{-1}\left(p^{1/2} \lambda_p \sigma(a, b) - \left(\frac{d}{p}\right)(\tau(a, b) - 1) - \lambda_p^2\right).\end{aligned}\end{equation}\\

\noindent  \textbf{Local-global compatibility.}  Recall that $\pi \simeq \bigotimes'_v \pi_v$ is an irreducible cuspidal representation with trivial central character.  Suppose further that $\Phi = \otimes_v \Phi_v$ is a pure tensor in $V_\pi$.  Let $\Omega$ be as above, and for $g = (g_v) \in G(\mathbb{A})$ let $g_\Omega = \prod_{v \in \Omega} g_v$.  Then by uniqueness of local Bessel models
\begin{equation}\label{bessel-global-to-local} B_{\Phi}(g) = B_{\Phi}(g_\Omega) \prod_{p \notin \Omega} B_p(g_p).\end{equation}
Note that (\ref{bessel-global-to-local}) makes sense even if both sides are zero.\\

\noindent  \textbf{A computation with global Bessel models.}  We now consider the implications of (\ref{bessel-global-to-local}) for the class of Siegel modular forms we are interested in.  Let $S$ be a finite set of primes and let $f \in \mathcal{S}_k(N)$ where $\gcd(N, S)=1$.  Assume that $f$ is an eigenform for the local Hecke algebras at $p \in S$.  Recall the representation $\pi_f$ attached to $f$ decomposes\begin{footnote}{Still not necessarily uniquely, and this is still not a problem.}\end{footnote} as $\pi_f = \bigoplus_{i=1}^m \pi_f^{(i)}$, where each $\pi_f^{(i)}$ is an irreducible cuspidal representation of $G(\mathbb{A})$.  Thus each vector $\Phi \in \pi_f$ is a sum of vectors $\Phi_i$ in the irreducible cuspidal representations $\pi_f^{(i)}$.  Also, for each $1 \leq i \leq m$, by the tensor product theorem, we have $\pi_f^{(i)} \simeq \otimes'_v \pi_{f. v}^{(i)}$ where the $\pi_{f, v}^{(i)}$ are irreducible, unitary, admissible representations of $G(\mathbb{Q}_v)$.  Thus each vector $\Phi_i \in \pi^{(i)}$ is in turn a sum of pure tensors $\otimes_v \Phi_{i, v}^{(j)} \in \otimes'_v \pi_{f, v}^{(i)}$.  In particular, suppressing the subscript $i$, we can write 
\begin{equation}\label{Phif-as-sum-of-pure-tensors} \Phi_f = \sum_{j=1}^n \otimes_v \Phi_{f, v}^{(j)}\end{equation}
where each $\otimes_v \Phi_{f, v}^{(j)}$ is a pure tensor in some irreducible cuspidal representation $\pi_f^{(i)}$ with $1 \leq i \leq m$.  \\

\noindent  Let $\mathsf{S}, \theta, \Lambda$ be given.  For the representation $\pi = \pi_f$ we can define, for any vector $\Phi \in V_\pi$, the Bessel functional $B_\Phi$ by (\ref{bessel-function-def}).  We ease notation by temporarily writing $B_{\Phi}(\cdot) = B(\cdot;\:\Phi)$.  From the definition and (\ref{Phif-as-sum-of-pure-tensors}) it is clear that
\begin{equation}\label{bessel-of-f-decomposition} B(\cdot;\:\Phi_f) = \sum_{j=1}^n B(\cdot;\:\otimes \Phi_{f, v}^{(j)}).\end{equation}

\noindent  Fix some $1 \leq j \leq n$ and consider $B(\cdot;\; \otimes_v \Phi_{f, v}^{(j)})$.  Let $\Omega = \{\infty\} \cup \{p \mid N\}$.  All of the local components $\pi_{f, p}^{(i)}$ at $p \notin \Omega$ are spherical principal series so $\Omega$ satisfies the hypotheses necessary for (\ref{bessel-global-to-local}).  Thus we have, for any $g \in G(\mathbb{A})$,
\begin{equation}\label{bessel-global-to-local-f-component} B(g;\; \otimes_v \Phi_{f, v}^{(j)}) = B(g_\Omega;\; \otimes_v \Phi_{f, v}^{(j)}) \prod_{p \notin \Omega} B_p^{(i)}(g_p),\end{equation}
where $B_p^{(i)}$ is the spherical vector in the Bessel model for the spherical principal series representation $\pi_{f, p}^{(i)}$ (recall $\otimes_v \Phi_v^{(j)} \in \otimes'_v \pi_{f, v}^{(i)} \simeq \pi_f^{(i)}$).  As $i$ varies, the local representations $\pi_{f, p}^{(i)}$ for $p \in S$ lie in the same isomorphism class.  In particular, as $i$ varies, the associated Bessel models to $\pi_{f, p}^{(i)}$ is the same space of functions on $G(\mathbb{Q}_p)$, and each $B_p^{(i)}$ is the same vector $B_p$.  So (\ref{bessel-global-to-local-f-component}) becomes
\begin{equation}\label{bessel-global-to-local-f-component2}B(g;\; \otimes_v \Phi_{f, v}^{(j)}) = B(g_\Omega;\; \otimes_v \Phi_{f, v}^{(j)}) \prod_{p \in S} B_p(g_p) \prod_{p \notin (\Omega \cup S)} B_p^{(i)}(g_p),\end{equation}
and putting these in to (\ref{bessel-of-f-decomposition}) we obtain 
\begin{equation}\label{split-bessel-global-to-local-f} B(g;\; \Phi_f) = \prod_{p \in S} B_p(g_p) \left(\sum_{j=1}^n B(g_\Omega;\; \otimes \Phi_{f, v}^{(j)}) \prod_{p \notin (\Omega \cup S)} B_p^{(i)}(g_p)\right) \end{equation}
where $i = i(j)$ is such that $\otimes_v \Phi_v^{(j)} \in \otimes'_v \pi_{f, v}^{(i)}$.  
In particular, if $g$ has the form
\[g_v = \begin{cases} 1_4 & v \notin S \\
g_p & v \in S \end{cases}\]
then, by our normalisation of the $B_p^{(i)}$, (\ref{split-bessel-global-to-local-f}) reads
\begin{equation}\label{split-bessel-global-to-local-f-nice-adeles} B(g;\: \Phi_f) = \prod_{p \in S} B_p(g_p) \left(\sum_{j=1}^n B(1_4;\: \otimes \Phi_{f, v}^{(j)})\right). \end{equation}
We will use (\ref{split-bessel-global-to-local-f-nice-adeles}) by explicitly computing the left hand side for certain $g \in G(\mathbb{A})$.  Namely, let $L, M$ be integers with all their prime factors in $S$, and define $H(L, M) \in G(\mathbb{A})$ by
\[H(L, M)_v = \begin{cases} \diag(LM^2, LM, 1, M) & v \in S, \\
1_4 & v \notin S. \end{cases}\]
In particular, $H(1, 1) = 1_4$.  The first step is to reduce the computation of $B(H(L, M);\:\Phi_f)$ to the computation for $H(1, 1)$ with a possibly different modular form:

\begin{lemma}\label{global-bessel-reduction-lemma}  Let $S$ be a finite set of primes, $N$ be a positive integer with $\gcd(N, S)=1$, and $L, M$ positive integers with all their prime factors in $S$.  Let $f \in \mathcal{S}_k(N)$.  Then there exists $f' \in \mathcal{S}_k$ such that
\[B(H(L, M);\; \Phi_f) = B(H(1, 1);\; \Phi_{f'})\]  \end{lemma}

\begin{proof}  Define $\Phi_f^{L, M}(g) = \Phi_f(gH(L, M))$.  Then clearly $B(H(L, M);\; \Phi_f) = B(H(1, 1);\;\Phi_f^{L, M})$.  Now let $H_{\infty} = \diag(LM^2, LM, 1, M) \in G^+(\mathbb{R})$ and define
\[f'(Z) = (LM)^{-k}f(H_{\infty}^{-1}\langle Z\rangle).\]
One easily checks that $f'(\gamma \langle Z\rangle) = j(\gamma, Z)^kf(Z)$ for
\begin{equation}\label{modularity-gp-f'} \gamma \in H_\infty \Gamma_0(N) H_\infty^{-1} = \left\{\left(\begin{matrix}* & M* & LM^{2}* & LM* \\ M^{-1}* & * & LM* & L* \\ L^{-1}M^{-2}N* & L^{-1}M^{-1}N* & * & M^{-1}* \\ L^{-1}M^{-1}N* & L^{-1}N* & M* & * \end{matrix}\right) \in \Sp_4(\mathbb{Q});\: * \in \mathbb{Z}\right\}.\end{equation}
This contains $\Gamma(NLM^2)$ as a subgroup of finite index, so is a congruence subgroup and $f' \in \mathcal{S}_k$.  Recall the choice of open compact subgroups (\ref{open-compact-def}).  The adelization $\Phi_{f'}$ of $f'$ is left invariant under $G(\mathbb{Q})$ and right invariant under $\prod_{p < \infty} K_p^{NLM^2}$.  \\

\noindent  We claim that $\Phi_{f'} = \Phi_{f}^{L, M}$.  Now one easily checks that $\Phi_{f}^{L, M}$ is also left-invariant under $G(\mathbb{Q})$ and right-invariant under $\prod_{p < \infty} K_p^{NLM^2}$, so it suffices to show that $\Phi_{f'}$ and $\Phi_f^{L, M}$ agree as functions on $G^+(\mathbb{R})$.  For $g_\infty \in G^+(\mathbb{R})$
\begin{equation}\label{tildeH-introduction} \Phi_{f}^{L, M}(g_\infty) = \Phi_f(g_\infty H(L, M)) = \Phi_f\left(\left(\begin{smallmatrix} L^{-1}M^{-2} & & & \\ & L^{-1}M^{-1} & & \\ & & 1 & \\ & & & M^{-1}\end{smallmatrix}\right) g_\infty H(L, M)\right)\end{equation}
where the first equality is the definition and the second follows from left-invariance of $\Phi_f$ under $G(\mathbb{Q})$.  Similarly using the right-invariance by $\prod_{p<\infty} K_p^{NLM^2}$ we can right-multiply the variable by the adele which is $\diag(LM^2, LM, 1, M)$ when $v \notin S \cup \infty$ and is $1_4$ otherwise (note that we are using the restriction on the prime factors of $L$ and $M$ here) to obtain
\[\Phi_{f}^{L, M}(g_\infty) = \Phi_f(H_\infty^{-1}g_\infty).\] 
A simple computation (see \cite{KowalskiSahaTsimerman2012} Proposition 2.1) then yields
\[\Phi_{f}^{L, M}(g_\infty) = \Phi_{f'}(g_\infty).\]\end{proof}

\noindent  Lemma \ref{global-bessel-reduction-lemma} reduces the computation of $B(H(L, M);\;\Phi_f)$ to the computation of $B(H(1, 1);\;\Phi_{f'})$, which is precisely the approach taken in \cite{KowalskiSahaTsimerman2012} (although note the slight change in our definition of $H(L, M)$).  In order to quote the result of the latter computation, we introduce their notation.  Given $M$, define
\[\Cl_d(M) = T(\mathbb{A}) / T(\mathbb{Q}) T(\mathbb{R}) \prod_{p < \infty} (T(\mathbb{Q}_p) \cap K_p^{(0)}(M)),\]
where $K_p^{(0)}(M) = \left\{g \in \GL_2(\mathbb{Z}_p);\: g \equiv \left(\begin{smallmatrix} * & 0 \\ * & * \end{smallmatrix}\right) \bmod M\right\}$.  $\Cl_d(1)=\Cl_d$ is isomorphic to the ideal class group of $\mathbb{Q}(\sqrt{-d})$; in general $\Cl_d(M)$ is the ray class group for the modulus $M$.  Pick coset representatives $t_c \in T(\mathbb{A})$ (indexed by $c \in \Cl_d(M)$) for this quotient, and write (by strong approximation for $T$)
\[t_c = \gamma_c m_c \kappa_c\]
with $\gamma_c \in \GL_2(\mathbb{Q})$, $m_c \in \GL_2^+(\mathbb{R})$, $\kappa_c \in \prod_{p<\infty} K_p^{(0)}(M)$.  Let
\begin{equation}\label{S_c-definition} \mathsf{S}_c := \frac{1}{\det(\gamma_c)}{}^t\gamma_c \mathsf{S} \gamma_c,\end{equation}
where $\mathsf{S}$ is the matrix for our choice of Bessel model.  We also define, for any symmetric matrix $Q$, the matrix
\begin{equation}\label{LM-mod-def} Q^{L, M} := \left(\begin{matrix} L & \\ & L \end{matrix}\right) \left(\begin{matrix} M & \\ & 1 \end{matrix}\right) Q \left(\begin{matrix} M & \\ & 1 \end{matrix}\right).\end{equation} 

\begin{proposition}\label{compute-global-bessel}  Suppose we have the same hypotheses as Lemma \ref{global-bessel-reduction-lemma}.  Then
\[B(H(L, M); \Phi_f) = \frac{re^{-2\pi \tr(\mathsf{S})} (LM)^{-k}}{\abs{\Cl_d(M)}} \sum_{c \in \Cl_d(M)} \overline{\Lambda(c)} a(\mathsf{S}_{c}^{L, M}; f)\]
where $r$ is a nonzero constant depending only on the normalization of Haar measure on the Bessel subgroup $R$.  \end{proposition}
\begin{proof}  Lemma \ref{global-bessel-reduction-lemma} reduces this to the case in \cite{KowalskiSahaTsimerman2012} Proposition 2.1.  Note that this computation uses the fact that $\Phi_{f'}$ is right invariant under $\left\{g \in \GL_2(\mathbb{Z}_p);\: g \equiv \left(\begin{smallmatrix} * & 0 \\ * & * \end{smallmatrix}\right) \bmod M \right\}$, embedded as a subgroup of $G(\mathbb{Z}_p)$ via (\ref{embed-gl-in-gsp4}).  That this still holds in our case is clear from (\ref{modularity-gp-f'}).  \end{proof}

\noindent  Let $L, M$ be integers with all their prime factors in $S$ and $H(L, M)$ be as above.  By (\ref{split-bessel-global-to-local-f-nice-adeles}) we have
\[B(H(L, M); \Phi_f) = \prod_{p \in S} B_p(h_p(l_p, m_p)) \left(\sum_{j=1}^n B(1_4; \otimes_v \Phi_{f, v}^{(j)})\right)\]
where $l_p = \ord_p(L)$, $m_p = \ord_p(M)$ and $h_p(l_p, m_p) = \diag(p^{l_p + 2m_p}, p^{l_p + m_p}, 1, p^{m_p})$.  Also from (\ref{split-bessel-global-to-local-f-nice-adeles}) we have
\[B(H(1, 1);\: \Phi_f) = \left(\sum_{j=1}^n B (1_4;\:\otimes_v \Phi_{f, v}^{(j)})\right),\]
so
\[B(H(L, M);\: \Phi_f) = B(H(1, 1);\: \Phi_f)\prod_{p \in S} B_p(h_p(l_p, m_p)).\]
Using Proposition \ref{compute-global-bessel} twice we obtain
\begin{equation}\label{local-bessels-in-terms-of-fourier-coeff} \frac{(LM)^{-k}}{\abs{\Cl_d(M)}} \sum_{c \in \Cl_d(M)} \overline{\Lambda(c)}a(\mathsf{S}_c^{L,M}; f) = \prod_{p \mid LM} B_p(h_p(l_p, m_p)) \times \frac{1}{\abs{\Cl_d}} \sum_{c \in \Cl_d} \overline{\Lambda(c)}a(\mathsf{S}_c; f), \end{equation}
and hence using Theorem \ref{sugano-formula}
\begin{equation}\label{local-bessels-in-terms-of-sugano} \frac{\abs{\Cl_d}}{\abs{\Cl_d(M)}} \sum_{c \in \Cl_d(M)} \overline{\Lambda(c)}a(\mathsf{S}_c^{L,M}; f) = L^{k-\frac{3}{2}} M^{k-2} \sum_{c \in \Cl_d} \overline{\Lambda(c)}a(\mathsf{S}_c; f) \prod_{p \mid LM} U_p^{l_p, m_p}(a_p(f), b_p(f)). \end{equation}
Equation (\ref{local-bessels-in-terms-of-sugano}) is crucial to our argument.  It allows us to reduce the study of certain continuous functions $U_p^{l_p, m_p}$ on the space $X_p \subset Y_p$ at the parameters corresponding to $f$ to the study of certain sums of the Fourier coefficients of $f$.  In the next section we will prove a result that allows us to do the latter.

\section{Estimates for sums of Fourier coefficients of cusp forms}\label{sctn:fourier-coeff-bound}

They key to estimating (\ref{local-bessels-in-terms-of-sugano}) is the following proposition: 
\begin{proposition}\label{crly:applicable-version}  Let $k \geq 6$ be even, $N \geq 1$, and let $\mathcal{S}_k(N)^*$ be any orthogonal basis of $\mathcal{S}_k(N)$.  Let $d<0$ be a fundamental discriminant, $L$ and $M$ positive integers.  Recall the definition of $\Cl_d(M)$; for $c' \in \Cl_d(M)$ and $c \in \Cl_d$ recall also the matrices $\mathsf{S}_{c'}$ and $\mathsf{S}_c^{L, M}$ defined by (\ref{S_c-definition}) and (\ref{LM-mod-def}).  Then
\[\begin{aligned} &\frac{2}{\vol(\Gamma_0(N) \backslash \mathbb{H}_2)}\sqrt{\pi} (4\pi)^{3-2k} \Gamma\left(k - \frac{3}{2}\right) \Gamma(k-2) \left(\frac{d}{4}\right)^{-k + \frac{3}{2}} \sum_{f \in \mathcal{S}_k(N)^*} \frac{\overline{a(\mathsf{S}_{c'}; f)}a(\mathsf{S}_{c}^{L, M}; f)}{\langle f, f \rangle} \\
&\qquad\qquad= \delta(c, c', L, M) + E(N, k; c, c', L, M),\end{aligned}\]where
\[\delta(c, c', L, M) = \#\{U \in \GL_2(\mathbb{Z}); U\mathsf{S}_{c'}{}^tU = \mathsf{S}_c^{L, M}\}\]
(which may equal zero), and the error term satisfies
\[E(N, k; c, c', L, M) \ll_\epsilon N^{-1}k^{-\frac{2}{3}} (LM)^{k- \frac{1}{2} + \epsilon}.\]\end{proposition}

\noindent We will prove this using estimates for Fourier coefficients of Poincar\'{e} series.  Given $Q \in \mathbb{Q}^{2 \times 2}_{\text{sym}}$ positive definite and semi-integral and a positive even integer $k$, we define the associated Poincar\'{e} series of weight $k$ and level $N$ by
\begin{equation}\label{poincare-def} G_{Q, N, k}(Z) = \sum_{M \in \Delta \backslash \Gamma_0(N)} j(M, Z)^{-k} e(\tr(Q\cdot M\langle Z \rangle)),\end{equation}
where $\Delta = \left\{\left(\begin{smallmatrix} 1_2 & U \\ 0_2 & 1_2 \end{smallmatrix} \right) \in \Sp_4(\mathbb{Z})\right\}$.  This series converges uniformly and absolutely on compact subsets of $\mathbb{H}_2$ provided $k \geq 6$.  The following property of Poincar\'{e} series is well-known, and can be found for the case $N=1$ in for example \cite{Klingen1990}.  We include the argument for any level $N$ here since the value of the constant of proportionality will be important for our application.

\begin{lemma}\label{poincare-petersson-lemma}  Let $Q \in \mathbb{Q}^{2 \times 2}_{\text{sym}}$ be positive definite symmetric, $k \geq 6$ be even, and $N$ be a positive integer.  Let $G_{Q, N, k}$ be defined by (\ref{poincare-def}), and let $f = \sum_{T>0} a(T; f)e(\tr(TZ)) \in \mathcal{S}^{(2)}_k(N)$.  Then
\[\langle G_{Q, N, k}, f \rangle = \frac{2}{\vol(\Gamma_0(N) \backslash \mathbb{H}_2)}\sqrt{\pi} (4\pi)^{3-2k} \Gamma\left(k - \frac{3}{2}\right) \Gamma(k-2) \det(Q)^{-k + \frac{3}{2}}  \overline{a(Q; f)},\]
where $\langle, \rangle$ is the Petersson inner product defined by (\ref{petersson-definition}).\end{lemma}
\begin{proof}  Proceeding formally we have
\[\begin{aligned} \langle G_{Q, N, k}, f\rangle &= \frac{1}{\vol(\Gamma_0(N) \backslash \mathbb{H}_2)} \int_{\Gamma_0(N) \backslash \mathbb{H}_2} G_{Q, N, k}(Z) \overline{f(Z)} \det(Y)^{k} \frac{dX dY}{\det(Y)^3} \\
&= \frac{1}{\vol(\Gamma_0(N) \backslash \mathbb{H}_2)} \int_{\Gamma_0(N) \backslash \mathbb{H}_2} \sum_{M \in \Delta \backslash \Gamma_0(N)} j(M, Z)^{-k} e(\tr(Q\cdot M\langle Z\rangle)) \overline{f(Z)} \det(Y)^k \frac{dX dY}{\det(Y)^3}\end{aligned}\]
where $Z = X + iY$.  Now for $M \in \Gamma_0(N)$ we have $\det(\Im(Z))^k \overline{f(Z)} j(M, Z)^{-k} = \overline{f(MZ)}\det(\Im(MZ))^k$, so we can write
\[\begin{aligned} \langle G_{Q, N, k}, f\rangle &= \frac{1}{\vol(\Gamma_0(N) \backslash \mathbb{H}_2)} \int_{\Gamma_0(N) \backslash \mathbb{H}_2} \sum_{M \in \Delta \backslash \Gamma_0(N)}e(\tr(Q\cdot M\langle Z\rangle)) \overline{f(MZ)} \det(\Im(MZ))^k \frac{dX dY}{\det(Y)^3} \\
&= \frac{2}{\vol(\Gamma_0(N) \backslash \mathbb{H}_2)} \int_{\Delta \backslash \mathbb{H}_2} e(\tr(Q Z)) \overline{f(Z)} \det(Y)^k \frac{dX dY}{\det(Y)^3} \\
&= \frac{2}{\vol(\Gamma_0(N) \backslash \mathbb{H}_2)} \int_{Y>0} \int_{X \bmod 1} e(\tr(Q  Z)) \overline{f(Z)} \det(Y)^k \frac{dX dY}{\det(Y)^3}.\end{aligned}\]
Here the integral with respect to $X$ is over $\mathbb{R}^{n \times n}_{\text{sym}}$ with entries taken modulo $1$, and the integral with respect to $Y$ is over all positive definite matrices in $\mathbb{R}^{2 \times 2}_{\text{sym}}$.  The factor of $2$ appears in the second line because $-1_4 \notin \Delta$ but it acts trivially on $\mathbb{H}_2$. Substituting in the Fourier expansion $f(Z) = \sum_{T>0} a(T; f) e(\tr(TZ))$ and integrating with respect to $X$ we see that only the $T=Q$ term survives, giving
\[\langle G_{Q, N, k}, f\rangle = \frac{2\overline{a(Q; f)}}{\vol(\Gamma_0(N) \backslash \mathbb{H}_2)} \int_{Y>0} e^{-4\pi \tr(QY)} \det(Y)^{k-3} dY.\]
It remains to compute this integral.  However, this is well-known (or easily computable by induction), for example by (\cite{Maass1951}, (40)) we have
\[\int_{Y>0} e^{-4\pi \tr(QY)} \det(Y)^{k-3} dY = \sqrt{\pi} (4\pi)^{3-2k} \Gamma\left(k-\frac{3}{2}\right)\Gamma(k-2) \det(Q)^{-k+\frac{3}{2}}.\]\end{proof}

\noindent An immediate corollary of Lemma \ref{poincare-petersson-lemma} is that the $G_{Q, N, k}$ generate $\mathcal{S}_k(N)$ as $Q$ varies.  Thus one can obtain results on the growth of Fourier coefficients of Siegel cusp forms by studying the growth rate Fourier coefficients of Poincar\'{e} series.  Such studies were initiated in \cite{Kitaoka1984}, who considered the dependency on $\det(T)$ only.  In \cite{KowalskiSahaTsimerman2012} a saving with respect to the weight $k$ was obtained, and similarly in \cite{ChidaKatsuradaMatsumoto2011} for the level $N$.  We need a version which saves with both $k$ and $N$.  Our estimations are based on those of \cite{KowalskiSahaTsimerman2012} and in fact obtain a better decay than the $N^{-1/2}$ of \cite{ChidaKatsuradaMatsumoto2011}.  We shall prove:

\begin{theorem}  \label{bound-on-poincare-fourier-coefficients}  Let $k \geq 6$ be even and let $N$ be a positive integer.  Let $Q \in \mathbb{Q}^{2 \times 2}$ be positive definite and semi-integral.  Then for any positive definite semi-integral matrix $T$
\begin{equation}\label{poincare-coefficient-bound} a(T; G_{Q, N, k}) = \delta(T, Q) + E(N, k, T),\end{equation}
where
\[\delta(T, Q) = \#\{U \in \GL_2(\mathbb{Z}); U Q {}^tU = T\}\]
(which may equal zero), and the error term satisfies
\[E(N, k, T) \ll_{\epsilon, Q} N^{-1}k^{-\frac{2}{3}} \det(T)^{k/2- 1/4 + \epsilon}.\]\end{theorem}

\noindent It is easy to prove Proposition \ref{crly:applicable-version} from Theorem \ref{bound-on-poincare-fourier-coefficients}:

\begin{proof}[Proof of Proposition \ref{crly:applicable-version}]  Since $\mathcal{S}_k(N)^*$ is an orthogonal basis 
\[G_{Q, N, k} = \sum_{f \in \mathcal{S}_k(N)^*} \frac{\langle G_{Q, N, k}, f \rangle}{\langle f, f\rangle} f\]
and hence for any positive definite semi-integral $T \in \mathbb{Q}^{2 \times 2}_{\text{sym}}$ 
\begin{equation}\label{eqn:pp-lemma-on-basis} a(T; G_{Q, N, k}) = \sum_{f \in \mathcal{S}_k(N)^*} \frac{\langle G_{Q, N, k}, f \rangle}{\langle f, f \rangle} a(T; f).\end{equation}
We take $c \in \Cl_d(M)$, $c' \in \Cl_d$, and $T = \mathsf{S}_c^{L, M}$, $Q = \mathsf{S}_{c'}$.  Using $\det(\mathsf{S}_{c'}) = d/4$ we then have
that the right hand side of (\ref{eqn:pp-lemma-on-basis}) is
\[\frac{2}{\vol(\Gamma_0(N) \backslash \mathbb{H}_2)}\sqrt{\pi} (4\pi)^{3-2k} \Gamma\left(k - \frac{3}{2}\right) \Gamma(k-2) \left(\frac{d}{4}\right)^{-k + \frac{3}{2}} \sum_{f \in \mathcal{S}_k(N)^*} \frac{\overline{a(\mathsf{S}_{c'}; f)}a(\mathsf{S}_{c}^{L, M}; f)}{\langle f, f \rangle}.\]
But the left hand side of (\ref{eqn:pp-lemma-on-basis}) is estimated by Theorem \ref{bound-on-poincare-fourier-coefficients}, with error term
\[E(k, N, \mathsf{S}_c^{L, M}) \ll_\epsilon N^{-1}k^{-2/3} \det(\mathsf{S}_c^{L, M})^{\frac{k}{2} - \frac{1}{4} + \epsilon}.\]
One easily computes that $\abs{T} = L^2M^2d/4$, and since $d$ is treated as constant we obtain the statement of the corollary.
\end{proof}

\noindent Before embarking on the proof of Theorem \ref{bound-on-poincare-fourier-coefficients} let us consider whether it is possible to give a simpler qualitative proof of this ``asymptotic orthogonality'' of Poincar\'{e} series Fourier coefficients.  The motivation for this question is the argument in \cite{KowalskiSahaTsimerman2011}, which does precisely this in the $k$-aspect (the proof also works more generally for Siegel modular forms of degree $g$ not necessarily equal to $2$).  A sketch of the proof is as follows: one uses the fact that the Siegel fundamental domain can be characterised by a finite list of conditions along with
\begin{equation}\label{eqn:siegel-limit}\lim_{y \rightarrow \infty} \abs{\det(Cyi1_2 + D)} = +\infty\end{equation}
(valid for $(C, D)$ the bottom block-rows of any real symplectic matrix, where $C \neq 0$) to produce a positive real number $y_0$ and a set $U_g(y_0) = \{X + iy_01_2;\: X \in \mathbb{R}^{2 \times 2}_{\text{sym}};\: \abs{x_{ij}} \leq \frac{1}{2} \}$ such that, for $M = \left(\begin{smallmatrix} A & B \\ C & D \end{smallmatrix}\right) \in \Sp_4(\mathbb{Z})$ with $C \neq 0$ and $Z \in U_g(y_0)$, we have $\abs{j(M, Z)} > 1$.  Since $\abs{e(\tr(Q \cdot M\langle Z \rangle))} \leq 1$ (for any $M, Z$), an application of dominated congergence for series shows that for $Z \in U_g(y_0)$ the Poincar\'{e} series 
\[G_{Q, 1, k}(Z) = \sum_{M \text{ s.t.} C=0} j(M, Z)^{-k}e(\tr(Q\cdot M\langle Z \rangle)) + \sum_{\substack{M \text{ s.t. } C \neq 0}} j(M, Z)^{-k} e(\tr(Q \cdot M\langle Z \rangle))\] 
converges to the sub-series defined by the first sum, as $k \rightarrow \infty$.  On the other hand, using dominated convergence again, one sees that the Fourier coefficients in the limit can be computed by integrating this limiting sub-series over $U_g(y_0)$, and a simple computation (c.f. Lemma \ref{asymptotic-orthogonality-lemma1}) therefore gives the qualitative version of Theorem \ref{bound-on-poincare-fourier-coefficients}.\\

\noindent In order to extend this argument to deal with the case $N>1$ one can argue as in the proof of Proposition 2 of \cite{KowalskiSahaTsimerman2011}.  Let $U_g(Y_0)$ be as above and write
\[G_{Q, N, k}(Z) = \sum_{M \text{ s.t.} C=0} j(M, Z)^{-k} e(tr(Q \cdot M\langle Z \rangle)) + \sum_{M \text{ s.t.} C \neq 0} \Delta_N(C) j(M, Z)^{-k} e(\tr(Q \cdot M \langle Z \rangle)) \]
where $\Delta_N(C) = 1$ if $C \equiv 0_2 \bmod N$, and is zero otherwise.  The limit of each term in the second series is $0$ as $N + k \rightarrow \infty$: indeed the large $k$ limit was treated above, and $\Delta_N(C) = 0$ for $N$ sufficiently large.  Thus by dominated convergence for series we again reduce to the first sum, and can then argue as above.\\

\noindent  The remainder of this section is occupied with the (somewhat technical) proof of Theorem \ref{bound-on-poincare-fourier-coefficients}.  We will treat $Q$ as being fixed, and will therefore suppress the dependency of implied constants on $Q$.  To ease notation we will also write $\abs{\cdot}$ for $\det(\cdot)$.\\

\noindent  Now let $\mathfrak{h}_N$ be a complete set of representatives for $\Delta \backslash \Gamma_0(N) / \Delta$.  For $M \in \Gamma_0(N)$, let
\[\theta(M) = \left\{S \in \mathbb{Z}^{2 \times 2}_{\text{sym}};\: M\left(\begin{matrix} 1_2 & S \\ 0_2 & 1_2 \end{matrix}\right)M^{-1} \in \Delta\right\}.\]
Note that $\mathbb{Z}_{\text{sym}}^{2 \times 2}/\theta(M)$ is in bijection with $\Delta/(\Delta \cap M^{-1} \Delta M)$; we will identify these.  We then clearly have
\begin{equation}\label{decomposition-of-DeltaMDelta} \Delta M \Delta = \bigsqcup_{S \in \mathbb{Z}^{2 \times 2}_{\text{sym}} / \theta(M)} \Delta M \left(\begin{matrix} 1_2 & S \\ 0_2 & 1_2 \end{matrix}\right).\end{equation}
Define $H(M, \cdot) = H_{Q, N, k}(M, \cdot)$ by
\begin{equation} H(M, Z) = \sum_{S \in \mathbb{Z}^{2 \times 2}_{\text{sym}} / \theta(M)} j(M, Z+S)^{-k}e(\tr(Q\cdot M\langle Z+S\rangle))\end{equation}
so that by  (\ref{decomposition-of-DeltaMDelta})
\begin{equation}\label{poincare-in-double-cosets} G(Z) = \sum_{M \in \mathfrak{h}_N} H(M, Z)\end{equation}
where $G = G_{Q, N, k}$.  Let $h(M, T) = h_{Q, N, k}(M, T)$ be given by $h(M, T) = a(T; H(M, \cdot))$.  Then by (\ref{poincare-in-double-cosets}) we have
\begin{equation}\label{poincare-fc-in-double-cosets} a(T; G) = \sum_{M \in \mathfrak{h}_N} h(M, T).\end{equation}
In order to estimate (\ref{poincare-fc-in-double-cosets}) we split the sum over the subsets
\[\mathfrak{h}_N^{(i)} = \left\{M = \left(\begin{matrix} A & B \\ C & D \end{matrix}\right) \in \mathfrak{h}_N;\: \rk(C) = i\right\}\]
by defining, for $i=0, 1, 2$,
\begin{equation}\label{poincare-fc-summands} R_i = \sum_{M \in \mathfrak{h}_N^{(i)}} h(M, T).\end{equation}
In Lemmas \ref{asymptotic-orthogonality-lemma1}, \ref{asymptotic-orthogonality-lemma2}, and \ref{asymptotic-orthogonality-lemma3} we shall treat the cases $i=0$, $i=1$, and $i=2$ respectively.

\begin{lemma}\label{asymptotic-orthogonality-lemma1}  In the notation of (\ref{poincare-fc-summands}),
\[R_0 = \#\left\{U \in \GL_2(\mathbb{Z});\: U Q {}^tU = T\right\}\]
\end{lemma}
\begin{proof}  Straightfoward computation (see \cite{ChidaKatsuradaMatsumoto2011} Proposition 3.2).  \end{proof}

\begin{lemma}\label{asymptotic-orthogonality-lemma2}  Let $\epsilon > 0$.  In the notation of (\ref{poincare-fc-summands}),
\[\abs{R_1} \ll_\epsilon N^{-1} k^{-\frac{5}{6}} \abs{T}^{\frac{k}{2}-\frac{1}{4}+\epsilon}.\]  \end{lemma}
\begin{proof}  We choose our representatives in $\mathfrak{h}_N^{(1)}$ to be of the form
\[M = \left(\begin{matrix} * & * \\ U^{-1}\left(\begin{smallmatrix} Nc & 0 \\ 0 & 0 \end{smallmatrix}\right){}^tV & U^{-1}\left(\begin{smallmatrix} d_1 & d_2 \\ 0 & d_4 \end{smallmatrix}\right)V^{-1}\end{matrix}\right)\]
where
\[\begin{aligned} &U \in \left\{\left(\begin{matrix} * & * \\ 0 & * \end{matrix}\right) \in \GL_2(\mathbb{Z})\right\} \backslash \GL_2(\mathbb{Z}), \\
&V \in \GL_2(\mathbb{Z})/\left\{\left(\begin{matrix} 1 & * \\ 0 & * \end{matrix}\right) \in \GL_2(\mathbb{Z})\right\},\end{aligned}\]
$c \geq 1$, $d_4 = \pm 1$, $(Nc, d_1)=1$ and $d_1, d_2$ vary modulo $Nc$.  For such an $M$, we have
\[\theta(M) = \left\{S \in \mathbb{Z}^{2 \times 2}_\text{sym};\: {}^tVSV = \left(\begin{matrix} 0 & 0 \\ 0 & * \end{matrix}\right)\right\}.\]
When $N=1$ the set of such $M$ are the representatives for $\mathfrak{h}_1^{(1)}$ used by Kitaoka.  It easily follows that we have a complete set of representatives when $N>1$ as well.  These are also the representatives used in \cite{ChidaKatsuradaMatsumoto2011}.\begin{footnote}{Note that there is a typo in the statement of Lemma 4.1 of \cite{ChidaKatsuradaMatsumoto2011}: the conditions on $d_1, d_2$ should be as we have stated them (modulo $Nc_1$ in their notation), and following this their $a_1$ should also be an inverse modulo $Nc_1$.  However, during the subsequent computations the variables are taken in the correct ranges, so this does not affect the results of their computation.  In the statement of \cite{ChidaKatsuradaMatsumoto2011} Lemma 4.2 the last term in the exponential should have $-d_4$ in place of $d_4$, but this again has no effect.}\end{footnote}\\

\noindent  Consider now a fixed $M$ as above.  Let $P = \left(\begin{smallmatrix} p_1 & p_2/2 \\ p_2/2 & p_4 \end{smallmatrix}\right) = UQ{}^tU$, $S = \left(\begin{smallmatrix} s_1 & s_2/2 \\ s_2/2 & s_4 \end{smallmatrix}\right) = V^{-1} T {}^tV^{-1}$, and let $a_1$ be any integer such that $a_1 d_1 \equiv 1 \bmod Nc$.  By the discussion in the previous paragraph, we can apply \cite{Kitaoka1984} \S3 Lemma 1 to our representatives (a subset of Kitaoka's), which gives
\begin{equation} \label{kitaoka-R1-lemma} \begin{aligned}h(M, T) &= \delta_{p_4, s_4}(-1)^{\frac{k}{2}} \sqrt{2} \pi \left(\frac{\abs{T}}{\abs{Q}}\right)^{\frac{k}{2}-\frac{3}{4}}  s_4^{-\frac{1}{2}} (Nc)^{-\frac{3}{2}} J_{k-\frac{3}{2}}\left(4\pi \frac{\sqrt{\abs{T}\abs{Q}}}{Nc s_4}\right) \\
&\qquad\qquad\times e\left(\frac{a_1 s_4 d_2^2 - (a_1d_4 p_2 - s_2)d_2}{Nc} + \frac{a_1 p_1 + d_1 s_1}{Nc} - \frac{d_4 p_2 s_2}{2Nc s_4}\right),\end{aligned}\end{equation}
where $\delta$ is the Kronecker delta, and $J$ is the ordinary Bessel function.  As in \cite{Kitaoka1984} and \cite{ChidaKatsuradaMatsumoto2011} we sum (\ref{kitaoka-R1-lemma}) over $d_2 \bmod Nc$, using the well-known bound on quadratic Gauss sums 
\[\sum_{x \bmod c} e\left(\frac{ax^2 + bx}{c}\right) \ll (a,c)^{\frac{1}{2}}c^\frac{1}{2}\] 
(see \cite{ChidaKatsuradaMatsumoto2011} for a proof) for the first term in the exponential sum, and bounding the other two in absolute value by $1$ to get
\[\abs{\sum_{d_2 \bmod Nc} h(M, T)} \ll \delta_{p_4, s_4} \left(\frac{\abs{T}}{\abs{Q}}\right)^{\frac{k}{2}-\frac{3}{4}} s_4^{-\frac{1}{2}} (s_4, Nc)^{\frac{1}{2}}(Nc)^{-1} \abs{J_{k-\frac{3}{2}} \left(\frac{4\pi \sqrt{\abs{T}\abs{Q}}}{N c s_4}\right)}.\]
Now sum this over $d_1 \bmod Nc$ such that $(d_1, Nc)=1$, and $d_4 = \pm 1$.  Since the sum over $d_1$ has length $O(Nc)$ we have
\[\sum_{d_1, d_4} \abs{\sum_{d_2} h(M, T)} \ll \delta_{p_4, s_4} \left(\frac{\abs{T}}{\abs{Q}}\right)^{\frac{k}{2}-\frac{3}{4}} s_4^{-\frac{1}{2}} (s_4, Nc)^{\frac{1}{2}} \abs{J_{k-\frac{3}{2}}\left(\frac{4\pi \sqrt{\abs{T}\abs{Q}}}{Ncs_4}\right)}.\]
We next sum this over all possible $U$ and $V$ for a fixed $c$.  Let us write $U = \left(\begin{smallmatrix} * & * \\ u_3 & u_4 \end{smallmatrix}\right)$.  By definition of our choice of $M$, the choice of $(u_3, u_4)$ determines $U$ up to sign.  Note also that, writing $u = \left(\begin{smallmatrix} u_3 \\ u_4 \end{smallmatrix}\right)$, $p_4 = Q[u]$.  So  
\[\sum_U \sum_{d_1, d_4,} \abs{\sum_{d_2} h(M, T)} \ll r(s_4; Q)\left(\frac{\abs{T}}{\abs{Q}}\right)^{\frac{k}{2}-\frac{3}{4}} s_4^{-\frac{1}{2}} (s_4, Nc)^{\frac{1}{2}} \abs{J_{k-\frac{3}{2}}\left(\frac{4\pi \sqrt{\abs{T}\abs{Q}}}{N s_4 c}\right)}\]
where $r(s_4; Q) = \abs{\{\left(\begin{smallmatrix} u_3 \\ u_4 \end{smallmatrix}\right) \in \mathbb{Z}^{2 \times 2};\: (u_3, u_4) = 1;\: Q[u] = s_4\}}$.  Similarly, write $V = \left(\begin{smallmatrix} v_1 & * \\ v_2 & * \end{smallmatrix}\right)$.  The choice of $(v_1, v_2)$ determines $V$, by the definition of our representatives $M$.  Summing over all $(v_1, v_2)$ such that $\gcd(v_1, v_2)=1$ we get
\[\sum_{U, V} \sum_{d_1, d_4} \abs{\sum_{d_2} h(M, T)} \ll \sum_{m \geq 1} r(m; T) r(m; Q) \left(\frac{\abs{T}}{\abs{Q}}\right)^{\frac{k}{2}-\frac{3}{4}} m^{-\frac{1}{2}} (m, Nc)^{\frac{1}{2}} \abs{J_{k-\frac{3}{2}} \left(\frac{4\pi \sqrt{\abs{T}\abs{Q}}}{N m c}\right)}.\]
Now it is well known that the number of proper representations of $m$ by a primitive positive definite quadratic form is $\ll_\eta m^{\eta}$, for any $\eta > 0$.  Applying this with $\eta/2$ we have
\[\sum_U \sum_{d_1, d_4,} \abs{\sum_{d_2} h(M, T)} \ll_\eta \sum_{\substack{ m \geq 1}} \left(\frac{\abs{T}}{\abs{Q}}\right)^{\frac{k}{2}-\frac{3}{4}} m^{-\frac{1}{2}+\eta} (m, Nc)^{\frac{1}{2}} \abs{J_{k-\frac{3}{2}} \left(\frac{4\pi \sqrt{\abs{T}\abs{Q}}}{N m c}\right)}.\]
Finally we sum over $c \geq 1$ to get, for any $\eta>0$,
\begin{equation}\label{rank1-summand-basic-bound} \abs{R_1} \ll_\eta \left(\frac{\abs{T}}{\abs{Q}}\right)^{\frac{k}{2}-\frac{3}{4}} \sum_{c, m \geq 1} m^{-\frac{1}{2} + \eta} (m, Nc)^{\frac{1}{2}} \abs{J_{k-\frac{3}{2}}\left(\frac{4\pi \sqrt{\abs{T}\abs{Q}}}{Nmc}\right)}.\end{equation}
As in \cite{KowalskiSahaTsimerman2012} we split the sum on the right hand side of (\ref{rank1-summand-basic-bound}) up in to $R_{11} + R_{12} + R_{13}$, but where $R_{1i}$ now corresponds to
\[\begin{cases} \frac{4\pi \sqrt{\abs{T}\abs{Q}}}{N} \leq mc & \text{if }i=1, \\
\frac{4\pi \sqrt{\abs{T}\abs{Q}}}{\sqrt{k}N} \leq mc \leq \frac{4\pi \sqrt{\abs{T}\abs{Q}}}{N} & \text{if }i=2, \\
mc \leq \frac{4\pi \sqrt{\abs{T}\abs{Q}}}{\sqrt{k}N} & \text{if }i=3. \end{cases}\]
So by definition we have
\begin{equation} \label{R1-split} \abs{R_1} \ll_\eta \left(\frac{\abs{T}}{\abs{Q}}\right)^{\frac{k}{2}-\frac{3}{4}} (R_{11} + R_{12} + R_{13}),\end{equation}
and we proceed to estimate each $R_{1i}$ individually. \\

\noindent \textbf{Case $R_{11}$:}  We are estimating 
\[R_{11} = \sum_{\substack{c, m \geq 1 \\ mc \geq \frac{4\pi \sqrt{\abs{T}\abs{Q}}}{N}}} m^{-\frac{1}{2}+\eta} (m, Nc)^{\frac{1}{2}} \abs{J_{k-\frac{3}{2}}\left(\frac{4\pi \sqrt{\abs{T}\abs{Q}}}{Nmc}\right)}.\]
In this range we use the estimate
\begin{equation}\label{bessel-estimate-small-x} J_k(x) \ll \frac{x^k}{\Gamma(k)},\:\:\:\:\text{if }k \geq 1,\: 0 \leq x \ll \sqrt{k+1}, \end{equation}  
(i.e. \cite{KowalskiSahaTsimerman2012} (3.1.3)), to get
\[J_{k-\frac{3}{2}}\left(\frac{4\pi \sqrt{\abs{T}\abs{Q}}}{Nmc}\right) \ll \frac{1}{\Gamma(k-\frac{3}{2})}\left(\frac{4\pi \sqrt{\abs{T}\abs{Q}}}{N m c}\right)^{k-\frac{3}{2}}.\]
Substituting this in to $R_{11}$ gives
\[R_{11} \ll \frac{1}{\Gamma\left(k - \frac{3}{2}\right)} \sum_{\substack{m, c \geq 1 \\ mc \geq \frac{4\pi \sqrt{\abs{T}\abs{Q}}}{N}}} m^{-\frac{1}{2} + \eta} (m, Nc)^{\frac{1}{2}} \left(\frac{4\pi \sqrt{\abs{T}\abs{Q}}}{Nmc}\right)^{k-\frac{3}{2}}.\]
Since $(4\pi \sqrt{\abs{T}\abs{Q}})/(Nmc) \leq 1$ and $k \geq 6$ we can replace the exponent $k-\frac{3}{2}$ with $1 + \delta$, where $0 < \delta \leq 1$.  Doing this, and putting $\pi$ and $d$ in to the implied constant, we get
\[R_{11} \ll  \frac{N^{-1-\delta}\abs{T}^{\frac{1}{2}+\frac{\delta}{2}}}{\Gamma(k-\frac{3}{2})} \sum_{\substack{m, c \geq 1 \\ mc \geq \frac{4\pi \sqrt{\abs{T} \abs{Q}}}{N}}} m^{-\frac{1}{2} + \eta} (m, Nc)^{\frac{1}{2}} \left(\frac{1}{mc}\right)^{1+\delta} \]
Taking $\delta = 2\eta$ (assuming $\eta$ is sufficiently small) and using $(m, Nc)^{\frac{1}{2}} \leq m^{\frac{1}{2}}$ in the double sum gives
\[ \sum_{\substack{m, c \geq 1 \\ mc \geq \frac{4\pi \sqrt{\abs{T} \abs{Q}}}{N}}} m^{-\frac{1}{2}+\eta} (m, Nc)^{\frac{1}{2}}\left(\frac{1}{mc}\right)^{1+2\eta} \leq \sum_{\substack{m, c \geq 1 \\ mc \geq \frac{4\pi \sqrt{\abs{T}\abs{Q}}}{N}}} m^{-1-\eta} c^{-1-2\eta}\]
which is manifestly convergent.  Thus we have
\[R_{11} \ll \frac{N^{-1-2\eta} \abs{T}^{\frac{1}{2}+\eta} }{\Gamma\left(k-\frac{3}{2}\right)}.\]
Since the gamma function grows superexponentially we have $\Gamma(k-\frac{3}{2}) \gg k^{E}$ for any $E \geq 1$, so for any such $E$
\begin{equation}\label{R11-final-bound} R_{11} \ll_\eta N^{-1} k^{-E} \abs{T}^{\frac{1}{2}+\eta}.\end{equation}

\noindent \textbf{Case $R_{12}$:}  We are now estimating
\begin{equation}\label{R12-def} R_{12} = \sum_{\substack{m, c \geq 1 \\ \frac{4 \pi \sqrt{\abs{T} \abs{Q}}}{N \sqrt{k}} \leq mc \leq \frac{4 \pi \sqrt{\abs{T} \abs{Q}}}{N}}} m^{-\frac{1}{2} + \eta} (m, Nc)^{\frac{1}{2}} \abs{J_{k-\frac{3}{2}} \left(\frac{4 \pi \sqrt{\abs{T} \abs{Q}}}{Nmc}\right)}.\end{equation}
In this range we can still use the estimate (\ref{bessel-estimate-small-x}).  This, together with $\frac{4 \pi \sqrt{\abs{T} \abs{Q}}}{Nmc} \leq \sqrt{k}$ (in this range), gives
\[\begin{aligned} J_{k-\frac{3}{2}}\left(\frac{4 \pi \sqrt{\abs{T} \abs{Q}}}{Nmc}\right) &\ll \frac{1}{\Gamma(k-\frac{3}{2})} \left(\frac{4 \pi \sqrt{\abs{T} \abs{Q}}}{Nmc}\right)^{k-\frac{3}{2}} \\
&\ll \frac{k^{\frac{k}{2}-\frac{3}{4}}}{\Gamma(k-\frac{3}{2})}.\end{aligned}\]
Substituting this in to (\ref{R12-def}) we have
\begin{equation} \label{R12-post-bessel-bound} R_{12} \ll \frac{k^{\frac{k}{2}-\frac{3}{4}}}{\Gamma(k-\frac{3}{2})} \sum_{\substack{m, c \geq 1 \\ \frac{4 \pi \sqrt{\abs{T} \abs{Q}}}{N\sqrt{k}} \leq mc \leq \frac{4 \pi \sqrt{\abs{T} \abs{Q}}}{N}}} m^{-\frac{1}{2}+\eta} (m, Nc)^{\frac{1}{2}}.\end{equation}
Now we can easily see that, for any $\delta > 0$,
\begin{equation} \label{gcd-sum-bound} \begin{aligned} \sum_{\substack{m, c \geq 1 \\ mc \leq X}} m^{-\frac{1}{2} + \eta}(m, Nc)^{\frac{1}{2}} &= \sum_{r \leq X} \sum_{e \mid r} e^{-\frac{1}{2} + \eta} \left(e, \frac{Nr}{e}\right)^{\frac{1}{2}} \ll_\delta X^{1 + \eta + \delta}. \end{aligned}\end{equation}
Taking $X = \frac{4 \pi \sqrt{\abs{T} \abs{Q}}}{N}$ we can bound the sum in (\ref{R12-post-bessel-bound}), with $\delta = \eta$ this gives
\[R_{12} \ll_\eta \frac{k^{\frac{k}{2}-\frac{3}{4}}}{\Gamma(k-\frac{3}{2})} \left(\frac{\sqrt{\abs{T}}}{N}\right)^{1+2\eta}.\] 
Using Stirling's formula we see that, for any $E \geq 1$, $\frac{k^{\frac{k}{2}-\frac{3}{4}}}{\Gamma(k-\frac{3}{2})} \gg k^E$, so for any $E \geq 1$, $\eta>0$ we have
\begin{equation}\label{R12-final-bound} R_{12} \ll_\eta N^{-1}k^{-E} \abs{T}^{\frac{1}{2}+\eta}.\end{equation}

\noindent \textbf{Case $R_{13}$:}  Finally we consider
\begin{equation}\label{R13-def} R_{13} = \sum_{\substack{m, c \geq 1 \\ mc \leq \frac{4 \pi \sqrt{\abs{T}\abs{Q}}}{N \sqrt{k}}}} m^{-\frac{1}{2}+\eta} (m, Nc)^{\frac{1}{2}} \abs{J_{k-\frac{3}{2}}\left(\frac{4 \pi \sqrt{\abs{T}\abs{Q}}}{N m c}\right)}. \end{equation}
Here we use\begin{equation}\label{bessel-estimate-large-x} J_k(x) \ll \min(1, xk^{-1})k^{-\frac{1}{3}}, \:\:\:\:\text{if }k \geq 1,\: x \geq 1 \end{equation}
(i.e. \cite{KowalskiSahaTsimerman2012} (3.1.4)).  By definition of this range $\frac{4 \pi \sqrt{\abs{T}\abs{Q}}}{Nmc} \geq \sqrt{k} \geq 1$, so (\ref{bessel-estimate-large-x}) is applicable and gives
\[J_{k-\frac{3}{2}}\left(\frac{4\pi \sqrt{\abs{T} \abs{Q}}}{N m c}\right) \ll \left(k-\frac{3}{2}\right)^{-\frac{1}{3}} \ll k^{-\frac{1}{3}}.\]
Substituting this in to (\ref{R13-def}) we get
\[ R_{13} \ll k^{-\frac{1}{3}} \sum_{\substack{c, m \geq 1 \\ cm \leq \frac{4 \pi \sqrt{\abs{T}\abs{Q}}}{ N\sqrt{k}}}} m^{-\frac{1}{2}+\eta} (m, Nc)^{\frac{1}{2}} \]
Using (\ref{gcd-sum-bound}) with $X = \frac{4 \pi \sqrt{\abs{T}\abs{Q}}}{N \sqrt{k}}$ and $\delta = \eta$ gives
\begin{equation}\label{R13-final-bound} \begin{aligned} R_{13} &\ll_\eta k^{-\frac{1}{3}} \left(\frac{\sqrt{\abs{T}}}{N\sqrt{k}}\right)^{1+2\eta} \\
&\ll_\eta N^{-1} k^{-\frac{5}{6}} \abs{T}^{\frac{1}{2} + \eta}. \end{aligned}\end{equation} 

\noindent Combining (\ref{R11-final-bound}), (\ref{R12-final-bound}), and (\ref{R13-final-bound}) in (\ref{R1-split}) we have
\[\begin{aligned} R_1 &\ll_\eta \left(\frac{\abs{T}}{\abs{Q}}\right)^{\frac{k}{2}-\frac{3}{4}} \left(N^{-1} k^{-E} \abs{T}^{\frac{1}{2} + \eta} + N^{-1} k^{-E} \abs{T}^{\frac{1}{2}+\eta} + N^{-1} k^{-\frac{5}{6}} \abs{T}^{\frac{1}{2}+\eta}\right) \\
&\ll_\eta N^{-1} k^{-\frac{5}{6}} \abs{T}^{\frac{k}{2}-\frac{1}{4} + \eta}. \end{aligned}\] 
Taking $\eta = \epsilon$ this is precisely the statement of the lemma.  \end{proof}

\noindent Before proceeding let us remark that it is in the proof of this lemma that we obtain the improvement on \cite{ChidaKatsuradaMatsumoto2011}.  The relevant quantities to compare are our estimates of $R_{11}$, $R_{12}$, $R_{13}$ and \cite{ChidaKatsuradaMatsumoto2011} Lemma 4.4 (which is the result used for estimates which are $\ll \abs{T}^{k/2-1/4}$).  The bottleneck in their estimate is \cite{ChidaKatsuradaMatsumoto2011}(4.9), corresponding to our ($R_{12}$ and) $R_{13}$.  For small $x$ (\cite{ChidaKatsuradaMatsumoto2011}(4.10)) they use the same estimate for the Bessel function as we do and one can check that their exponent on $N$ can be made to improve by assuming larger $k$ as we have.  However, they estimate $J_k(x)$ for large $x$ (\cite{ChidaKatsuradaMatsumoto2011}(4.9)) by $x^{-1/2}$ which ultimately introduces a factor of $N^{1/2}$; we estimate $J_k(x)$ for large $x$ by $k^{-1/3}$ which avoids this, as well as giving the saving we require with respect to $k$.  \\

\noindent Finally we estimate the remaining term $R_2$:
 
\begin{lemma}\label{asymptotic-orthogonality-lemma3}  Let $\epsilon > 0$.  In the notation of (\ref{poincare-fc-summands}),
\[\abs{R_2} \ll_\epsilon N^{-2} k^{-\frac{2}{3}} \abs{T}^{\frac{k}{2}-\frac{1}{4}+\epsilon}.\] \end{lemma}
\begin{proof}  We choose
\[\mathfrak{h}_N^{(2)} = \left\{M = \left(\begin{matrix} * & * \\ NC & D \end{matrix}\right) \in \Sp_4(\mathbb{Z});\: \abs{C} \neq 0;\: D \bmod NC\right\},\]
for such $M$ we have $\theta(M) = 0$.  When $N=1$ our $\mathfrak{h}_1^{(2)}$ is the set of representative of \cite{Kitaoka1984} \S2 Lemma 5.  Again it easily follows that we have a complete set of representatives when $N>1$ as well, and also that Kitaoka's computations are applicable.  Note that these are once again the same as the representatives used in \cite{ChidaKatsuradaMatsumoto2011}.  We can then write
\begin{equation}\label{rank2-summand-basic-bound} R_2 = \sum_{\substack{C \in \mathbb{Z}^{2 \times 2} \\ \abs{C} \neq 0}} \sum_{D \bmod NC} h(M, T)\end{equation}
with $M = \left(\begin{smallmatrix} * & * \\ NC & D \end{smallmatrix}\right)$ as above.  Fix a matrix $C$ and consider the sum over all $M = \left(\begin{smallmatrix} * & * \\ NC & D \end{smallmatrix}\right) \in \mathfrak{h}_N^{(2)}$.  Using the arguments of \cite{Kitaoka1984} \S4 following Lemma 1 up to the second equation on p166, we obtain the following: let  \begin{itemize}
\item  $P(NC) := T{}^t(NC)^{-1}Q(NC)^{-1}$, 
\item  $\norm{NC}$ be the absolute value of $\abs{NC} = N^2\det{(C)}$,
\item  $K(Q, T; NC)$ be the matrix Kloosterman sum defined (and bounded) by Kitaoka (\cite{Kitaoka1984}, \S1)
\item  $0 < s_1 \leq s_2$ be such that $s_1^2, s_2^2$ are the eigenvalues of the positive definite matrix $P(NC)$, and write
\[\mathcal{J}_k(P(NC)) = \int_0^{\pi/2} J_{k-\frac{3}{2}}(4 \pi s_1 \sin (\theta)) J_{k-\frac{3}{2}}(4 \pi s_2 \sin(\theta)) \sin(\theta) d\theta.\]\end{itemize}
Then
\[\begin{aligned} \sum_{D \bmod NC} h(M, T) &= \frac{1}{2\pi^4} \left(\frac{\abs{T}}{\abs{Q}}\right)^{\frac{k}{2}-\frac{3}{4}} \norm{NC}^{-\frac{3}{2}} K(Q, T; NC) \mathcal{J}_k(P(NC)).\end{aligned}\]

\noindent Using principal divisors we can write $NC \in \mathbb{Z}^{2 \times 2}$ with $\abs{C} \neq 0$ uniquely (see \cite{Kitaoka1984} \S4 Lemma 1) as
\[NC = U^{-1} \left(\begin{matrix} Nc_1 & 0 \\ 0 & Nc_2 \end{matrix}\right)V^{-1}\]
where $1 \leq c_1$, $c_1 \mid c_2$, $U \in \GL_2(\mathbb{Z})$ and $V \in \SL_2(\mathbb{Z})/\Gamma^0(c_2/c_1)$.  Here 
\[\Gamma^0(m) = \left\{\left(\begin{smallmatrix} a & b \\ c & d \end{smallmatrix}\right) \in \SL_2(\mathbb{Z});\:b \equiv 0 \bmod m\right\}.\]  
We will thus consider our matrix $NC$ to be parameterised by $(U, Nc_1, Nc_2, V)$.  To handle the sum over the $NC$, first suppose that $(Nc_1, Nc_2, V)$ is fixed.  Pick $U_1 \in \GL_2(\mathbb{Z})$ such that
\begin{equation} \label{A-def} A = A(Nc_1, Nc_2, V) := {}^t\left(V \left(\begin{smallmatrix} Nc_1 & 0 \\ 0 & Nc_2 \end{smallmatrix}\right)^{-1} U_1\right)T\left(V\left(\begin{smallmatrix} Nc_1 & 0 \\ 0 & Nc_2\end{smallmatrix}\right)^{-1}U_1\right)\end{equation}
is Minkowski-reduced.  Clearly we have that the matrices $NC$ with parameters $(Nc_1, Nc_2, V)$ are precisely the matrices
\[NC = U^{-1}U_1^{-1} \left(\begin{matrix} Nc_1 & 0 \\ 0 & Nc_2 \end{matrix}\right) V^{-1}\]
as $U$ varies over $\GL_2(\mathbb{Z})$.  Hence we can write, for any $NC$ with parameters $(Nc_1, Nc_2, V)$,
\begin{equation} \begin{aligned} \label{P-in-U1-var}P(NC) &= T {}^t (NC)^{-1} Q (NC)^{-1} \\
&= T {}^t \left(V \left(\begin{smallmatrix} Nc_1 & 0 \\ 0 & Nc_2\end{smallmatrix}\right)^{-1}U_1 U\right) Q \left(V \left(\begin{smallmatrix} Nc_1 & 0 \\ 0 & Nc_2 \end{smallmatrix}\right)^{-1}U_1 U\right)\end{aligned}\end{equation}
From (\ref{A-def}) and (\ref{P-in-U1-var}) we immediately see $\abs{P(NC)} = \abs{Q}\abs{A}$.  On the other hand, $\abs{P(NC)} = s_1^2s_2^2$, by definition of $s_1, s_2$.  Now $A$, being positive definite symmetric, is diagonizable, say to
\[H = H(Nc_1, Nc_2, V) := \left(\begin{matrix} a & 0 \\ 0 & c \end{matrix}\right),\]
where $0 < a \leq c$.  Hence we have, recalling that $\abs{Q}$ is treated constant,
\begin{equation}\label{det-like-H} \abs{H} = ac \asymp s_1^2 s_2^2 = \abs{P(NC)}.\end{equation}
By computing the determinant in (\ref{A-def}) we have
\begin{equation} s_1^2 s_2^2 \asymp \frac{\abs{T}}{N^4c_1^2c_2^2}.\label{si-in-known-parameters}\end{equation}  Since $A$ is Minkowski-reduced, we also have
\begin{equation}\label{trace-like-H}\tr(P(NC)) \asymp \tr(A[U]) = \tr(H[U]).\end{equation}

\noindent Continuing to work with any $NC$ having parameters $(Nc_1, Nc_2, V)$, (\cite{Kitaoka1984} \S 1 Prop. 1) gives us
\[K(Q, T; NC) \ll_\epsilon N^{\frac{5}{2}+\epsilon}c_1^2c_2^{\frac{1}{2}+\epsilon} (Nc_2, {}^tvTv)^{\frac{1}{2}}\]
for any $\epsilon>0$, where $v$ is the second column of $V$.  Thus
\begin{equation}\label{kitaoka-h-bound} \begin{aligned} \sum_{D \bmod NC} h(M, T) &\ll_\epsilon \left(\frac{\abs{T}}{\abs{Q}}\right)^{\frac{k}{2}-\frac{3}{4}} N^{-\frac{1}{2}+\epsilon} c_1^{\frac{1}{2}} c_2^{-1+\epsilon} (Nc_2, {}^tvTv)^{\frac{1}{2}} \abs{\mathcal{J}_k(P(NC))}.\end{aligned}\end{equation}

\noindent We handle the different $NC$ according to the properties of $P(NC)$ by partitioning in to the following sets:
\[\begin{aligned} \mathcal{C}_1 &= \left\{NC \in \mathbb{Z}^{2 \times 2};\: \abs{C} \neq 0;\:\tr(P(NC)) < 1\right\}, \\
\mathcal{C}_2 &= \left\{NC \in \mathbb{Z}^{2 \times 2};\: \abs{C} \neq 0;\: \tr(P(NC)) \geq \max(2\abs{P(NC)}, 1)\right\}, \\
\mathcal{C}_3 &= \left\{NC \in \mathbb{Z}^{2 \times 2};\: \abs{C} \neq 0;\: 1 \leq \tr(P(NC)) < 2\abs{P(NC)}\right\}. \end{aligned}\]
Recall we had $0 < s_1 \leq s_2$ such that $s_1^2, s_2^2$ were the eigenvalues of $P(NC)$.  So $\tr(P(NC)) = s_1^2 + s_2^2$ and $\abs{P(NC))} = s_1^2s_2^2$.  For $\mathcal{C}_1$, $\tr(P(NC)) < 1$ implies $s_1^2, s_2^2 \leq 1$.  For $\mathcal{C}_2$, $s_1^2 + s_2^2 \geq 2 s_1^2 s_2^2$ and $s_2^2 \geq s_1^2$ imply $s_1^2 \leq 1$; in addition $s_1^2 + s_2^2 \geq 1$ then gives $s_2^2 \geq \max(1-s_1^2, s_1^2) \geq 1/2$.  For $\mathcal{C}_3$, we have $2s_1^2s_2^2 \geq s_1^2 +s_2^2$ which, together with the AM-GM inequality $s_1^2 + s_2^2 \geq 2 s_1s_2$ gives $s_1s_2 \geq 1$, so $s_2 \geq 1$; and $2s_1^2s_2^2 \geq s_1^2 + s_2^2$ also gives $s_1^2 \geq s_2^2/(2s_2^2 - 1)$, hence $s_1^2 \geq 1/2$.  It then follows that
\[\begin{aligned} \mathcal{C}_1 &\subset \left\{NC \in \mathbb{Z}^{2 \times 2};\: 0 < s_1 \leq s_2 \leq 1\right\}, \\
\mathcal{C}_2 &\subset \left\{NC \in \mathbb{Z}^{2 \times 2};\: 0 < s_1 \leq 1;\: s_2 \geq 1/\sqrt{2}\right\}, \\
\mathcal{C}_3 &\subset \left\{NC \in \mathbb{Z}^{2 \times 2};\: s_1 \geq 1/\sqrt{2};\: s_2 \geq 1\right\}. \end{aligned}\] 
This characterization of the $\mathcal{C}_i$ based on the values of the $s_i$ will be important in the following case analysis.  For now we also define
\[\mathcal{C}_i(Nc_1, Nc_2, V) = \{NC \in \mathcal{C}_i;\: NC\text{ has final three parameters }(Nc_1, Nc_2, V)\},\]
so $\bigcup_{(Nc_1, Nc_2, V)} \mathcal{C}_i(Nc_1, Nc_2, V) = \mathcal{C}_i$.  We recall the (weighted) sizes of these sets as proved in \cite{Kitaoka1984} \S4 Lemma 2 and stated in Lemma 3.4 of \cite{KowalskiSahaTsimerman2012}:  for any $\epsilon, \delta>0$,
\begin{align}
  \abs{\mathcal{C}_1(Nc_1, Nc_2, V)} &\ll_\epsilon (ac)^{-\frac{1}{2}-\epsilon} \label{size-of-c1}\\
  \sum_{NC \in \mathcal{C}_2(Nc_1, Nc_2, V)} \abs{A}^{1+\delta} \tr({}^tU A U)^{-\frac{5}{4} - \delta} &\ll_{\delta, \epsilon} \begin{cases} (ac)^{\frac{1}{2} + \delta  - \epsilon} & \text{if } ac<1\\ (ac)^{\frac{1}{4} + \epsilon} & \text{if }ac\geq 1\end{cases} \label{size-of-c2}\\
  \abs{\mathcal{C}_3(Nc_1, Nc_2, V)} &\ll_\epsilon (ac)^{\frac{1}{2}+\epsilon} \label{size-of-c3}
\end{align}
Note that again our $(Nc_1, Nc_2, V)$ are simply a subset of the $(c_1, c_2, V)$ considered in \cite{KowalskiSahaTsimerman2012}.  Finally, write 
\[R_2 = R_{21} + R_{22} + R_{33},\] 
where
\[R_{2i} = \sum_{NC \in \mathcal{C}_i} \sum_{D \bmod NC} h(M, T).\]
Then by (\ref{kitaoka-h-bound}) we have
\begin{equation}\label{R2i-starting-point} R_{2i} \ll_\epsilon \left(\frac{\abs{T}}{\abs{Q}}\right)^{\frac{k}{2}-\frac{3}{4}} N^{-\frac{1}{2} + \epsilon} \sum_{(Nc_1, Nc_2, V)} c_1^{\frac{1}{2}} c_2^{-1+\epsilon} (Nc_2, {}^tvTv)^{\frac{1}{2}} \mathcal{R}_{2i}(Nc_1, Nc_2, V)\end{equation}
where
\[\mathcal{R}_{2i}(Nc_1, Nc_2, V) = \sum_{NC \in \mathcal{C}_i(Nc_1, Nc_2, V)} \abs{\mathcal{J}_k(P(NC))}.\]
We will again bound each of these terms individually.  \\

\noindent \textbf{Case $R_{21}$:}  Here we have $s_1, s_2 \leq 1$.  Using the esimate (\ref{bessel-estimate-small-x}) we have
\[\begin{aligned} \abs{\mathcal{J}_k(P(NC))} &\ll \frac{1}{\Gamma\left(k-\frac{1}{2}\right)^2} (4 \pi s_1)^{k-\frac{3}{2}} (4 \pi s_2)^{k-\frac{3}{2}} \\
&\ll \frac{(s_1s_2)^{2+2\delta}}{k^E},\end{aligned}\]
where the final line holds for any reasonably small $\delta > 0$, $E \geq 1$, by using the fact that $k \geq 6$ and the superexponential growth of the gamma function.  Also, by (\ref{size-of-c1}),
\[\abs{\mathcal{C}_1(Nc_1, Nc_2, V)} \ll_\delta (ac)^{-\frac{1}{2}-\frac{\delta}{2}} \ll (s_1s_2)^{-1-\delta}.\]
So
\[\mathcal{R}_{21}(Nc_1, Nc_2, V) \ll_\delta \frac{(s_1s_2)^{1+\delta}}{k^E} \ll \frac{\abs{T}^{\frac{1}{2}+\frac{\delta}{2}}}{k^E(N^2c_1c_2)^{1+\delta}},\] 
using (\ref{si-in-known-parameters}).  Thus, with $c_1, c_2$ fixed,
\begin{equation}\label{V-sum} \sum_V c_1^{\frac{1}{2}} c_2^{-1+\epsilon} (Nc_2, {}^tvTv)^{\frac{1}{2}} \mathcal{R}_{21}(Nc_1, Nc_2, V) \ll_\delta \frac{\abs{T}^{\frac{1}{2} + \frac{\delta}{2}}}{k^E N^{2+2\delta}} \sum_V c_1^{-\frac{1}{2}-\delta} c_2^{-2-\delta+\epsilon} (Nc_2, {}^tvTv)^{\frac{1}{2}} .\end{equation}
By \cite{Kitaoka1984} \S1 Proposition 2 with $n=c_2/c_1$ we have, for any $\eta > 0$,
\[\sum_V \left(\frac{c_2}{c_1}, {}^tvTv\right)^{\frac{1}{2}} \ll_\eta \left(\frac{c_2}{c_1}\right)^{1+\eta} \left(\cont(T), \frac{c_2}{c_1}\right)^{\frac{1}{2}} \]
where, writing $T = \left(\begin{smallmatrix} t_1 & t_2/2 \\ t_2/2 & t_3 \end{smallmatrix}\right)$, $\cont(T) = \gcd(t_1, t_2, t_3)$.  Using $(Nc_2, {}^tvTv)^{\frac{1}{2}} \leq N^{\frac{1}{2}} c_1^{\frac{1}{2}} \left(\frac{c_2}{c_1}, {}^tvTv\right)^{\frac{1}{2}}$ in (\ref{V-sum}) then gives
\[\sum_V c_1^{\frac{1}{2}} c_2^{-1+\epsilon} (Nc_2, {}^tvTv)^{\frac{1}{2}} \mathcal{R}_{21}(Nc_1, Nc_2, V) \ll_{\delta, \eta} \frac{\abs{T}^{\frac{1}{2}+\frac{\delta}{2}}}{k^E}N^{-\frac{3}{2}-2\delta} c_1^{-1-\delta-\eta} c_2^{-1-\delta+ \epsilon+\eta} \left(\frac{c_2}{c_1}, \cont(T)\right)^{\frac{1}{2}}.\]
Substituting this in to (\ref{R2i-starting-point}), and writing $c_2 = nc_1$,
\[\begin{aligned} R_{21} &\ll_{\delta, \eta, \epsilon} \left(\frac{\abs{T}}{\abs{Q}}\right)^{\frac{k}{2}-\frac{3}{4}} \abs{T}^{\frac{1}{2}+\frac{\delta}{2}} N^{-2 - 2\delta} k^{-E} \sum_{c_1,n \geq 1}  c_1^{-2-2\delta+\epsilon} n^{-1-\delta+ \epsilon+\eta} \left(n, \cont(T) \right)^{\frac{1}{2}}. \end{aligned} \]
Take $\eta = \epsilon$, $\delta = 3\epsilon$.  Clearly the sum over $c_1$ is convergent.  For the sum over $n$ we note that $\sum_{n \geq 1} n^{-1} (n, \cont(T))^{\frac{1}{2}}$ may be written as $\sum_{e \mid \cont(T)} e^{-\frac{1}{2}}\sum_{m \geq 1} m^{-1} \ll \sum_{e \mid \cont(T)} e^{-\frac{1}{2}} \ll_\epsilon \cont(T)^\epsilon$.  Then using the inequality $\cont(T)^2 \leq 4\det(T)$ we see that the sum over $n$ is thus $\ll_\epsilon \abs{T}^\epsilon$.  Thus, after redefining $\epsilon$, we have for any $E \geq 1$
\[R_{21} \ll_\epsilon N^{-2} k^{-E} \abs{T}^{\frac{k}{2}-\frac{1}{4} + \epsilon}.\]

\noindent \textbf{Case $R_{22}$:}  This is the case $s_1 \leq 1$, $s_2 \gg 1$.  Now we have
\[\abs{\mathcal{J}_k(P(NC))} \ll \frac{s_1^{k-\frac{3}{2}}}{\Gamma(k-\frac{3}{2})} \frac{2^k}{s_2^{\frac{1}{2}}},\]
where we have used (\ref{bessel-estimate-small-x}) to bound the Bessel function involving $s_1$, and the estimate $J_{k}(x) \ll 2^kx^{-1/2}$ in the range $k \geq 1$, $x > 0$ (i.e. \cite{KowalskiSahaTsimerman2012} (3.1.5))  for the one involving $s_2$.  Let $NC$ have parameters $(U, Nc_1, Nc_2, V)$, and recall $A = A(Nc_1, Nc_2, V)$ defined by (\ref{A-def}).  We have $\abs{A} \asymp \abs{P(NC)} = s_1^2 s_2^2$, so 
\[\abs{\mathcal{J}_k(P(NC))} \ll \frac{2^k}{\Gamma(k-\frac{3}{2})}\frac{\abs{A}^{\frac{k}{2}-\frac{3}{4}}}{s_2^{k-1}}.\]
Also, by (\ref{trace-like-H}), $\tr({}^t U A U) \asymp \tr(P(NC)) = s_1^2 + s_2^2 \asymp s_2^2$, since $s_1 \leq 1$.  So
\[\abs{\mathcal{J}_k(P(NC))} \ll \frac{2^k}{\Gamma(k-\frac{3}{2})} \frac{\abs{A}^{\frac{k}{2}-\frac{3}{4}}}{\tr({}^tUAU)^{\frac{k-1}{2}}}.\]
For any $\delta > 0$ we may write $\abs{A}^{\frac{k}{2}-\frac{3}{4}} \tr({}^tUAU)^{\frac{1-k}{2}} = \abs{A}^{1+\delta} \tr({}^tU A U)^{\frac{5}{4}-\delta} \left(\frac{\abs{A}}{\tr({}^tUAU)}\right)^{\frac{k}{2}-\frac{7}{4} - \delta}$.  But $\frac{\abs{A}}{\tr({}^tUAU)} \asymp \frac{s_1^2s_2^2}{s_2^2} = s_1^2 \leq 1$.  Now $k \geq 6$ and we can assume $\delta$ is small, so
\[\abs{\mathcal{J}_k(P(NC))} \ll \frac{2^k}{\Gamma(k-\frac{3}{2})} \abs{A}^{1+\delta} \tr({}^tUAU)^{\frac{k}{2}-\frac{5}{4}-\delta}.\] 
Using (\ref{size-of-c2}) (with $\epsilon = \delta/2$) and the superexponential growth of the gamma function gives
\[\mathcal{R}_{22}(Nc_1, Nc_2, V) \ll k^{-E} \times \begin{cases} (ac)^{\frac{1}{2} + \frac{\delta}{2}} & \text{if }ac < 1, \\ (ac)^{\frac{1}{4}+\frac{\delta}{2}} & \text{if }ac \geq 1.\end{cases}\]
for any $E \geq 1$. Recalling from (\ref{det-like-H}) that $ac \asymp \abs{T}/(N^4c_1^2c_2^2)$ we can now write bound the sum for $R_{22}$ by
\[\begin{aligned} R_{22} \ll_\epsilon \left(\frac{\abs{T}}{\abs{Q}}\right)^{\frac{k}{2}-\frac{3}{4}} N^{-\frac{1}{2}+\epsilon}k^{-E} &\left(\sum_{c_1c_2 > \frac{\sqrt{\abs{T}}}{N^2}} \left(\frac{\sqrt{\abs{T}}}{N^2c_1c_2}\right)^{1+\delta} \sum_V c_1^{\frac{1}{2}} c_2^{-1+\epsilon}(Nc_2, {}^tvTv)^{\frac{1}{2}} \right.\\
&\left.\qquad+ \sum_{c_1c_2 \leq \frac{\sqrt{\abs{T}}}{N^2}} \left(\frac{\sqrt{\abs{T}}}{N^2c_1c_2}\right)^{\frac{1}{2}+\delta} \sum_V c_1^{\frac{1}{2}} c_2^{-1+\epsilon}(Nc_2, {}^tvTv)^{\frac{1}{2}} \right).\end{aligned}\]
Now in the second sum the base with exponent $\frac{1}{2}+\delta$ is larger than $1$, so we can certainly increase the exponent to $1+\delta$.  This then reduces to
\[\begin{aligned} R_{22} &\ll_\epsilon \left(\frac{\abs{T}}{\abs{Q}}\right)^{\frac{k}{2}-\frac{3}{4}}N^{-\frac{5}{2}-2\delta+\epsilon}k^{-E} \sum_{\substack{c_1, c_2 \geq 1 \\ c_1 \mid c_2}} \left(\frac{\sqrt{\abs{T}}}{c_1c_2}\right)^{1+\delta} \sum_V c_1^{\frac{1}{2}} c_2^{-1+\epsilon}(Nc_2, {}^tvTv)^{\frac{1}{2}} \\
&\ll_\epsilon \left(\frac{\abs{T}}{\abs{Q}}\right)^{\frac{k}{2}-\frac{3}{4}}N^{-\frac{5}{2}-2\delta+\epsilon}k^{-E} \abs{T}^{\frac{1}{2} + \frac{\delta}{2}} \sum_{c_1, c_2, V} c_1^{-\frac{1}{2}-\delta}c_2^{-2-\delta+\epsilon} (Nc_2, {}^tvTv)^{\frac{1}{2}}\end{aligned}\]
But the sum over $c_1, c_2, V$ is now exactly the same as the sum appearing in (\ref{V-sum}) (more precisely summed over $c_1, c_2$, as we proceeded to do there).  Thence we conclude that this sum over $c_1, c_2, V$ is $\ll_\delta N^{\frac{1}{2}} \abs{T}^\frac{\delta}{2}$, so taking $\delta = \epsilon$ we obtain
\[R_{22} \ll_\epsilon \abs{T}^{\frac{k}{2}-\frac{1}{4} + \epsilon} N^{-2} k^{-E}\]
for any $E \geq 1$ as before. \\

\noindent \textbf{Case $R_{23}$:}  In this case $1 \ll s_1 \leq s_2$.  Let $M_1 = \{\theta \in [0, 2\pi);\: 4\pi s_2 \sin\theta \leq 1\}$ (note that if $\theta \in M_1$ then $4 \pi s_1 \sin\theta \leq 1$ as well), and let $M_2 = \{\theta \in [0, 2\pi);\: 4 \pi s_1 \sin\theta \geq 1\}$ (and note that if $\theta \in M_2$ then $4 \pi s_2 \sin\theta \geq 1$ as well).  Then
\[\abs{\mathcal{J}_k(P(NC))} \ll \left(\int_{M_1} + \int_{M_2}\right) \abs{J_{k-\frac{3}{2}}(4 \pi s_1 \sin\theta) J_{k-\frac{3}{2}}(4 \pi s_2 \sin\theta)\sin\theta}d\theta.\] 
We estimate using (\ref{bessel-estimate-small-x}) and (\ref{bessel-estimate-large-x}) on $M_1$ and $M_2$ respectively.  Since the argument of the Bessel functions is $\leq 1$ on $M_1$, we may replace the exponent $k-\frac{3}{2}$ by $\delta$ for any $\delta > 0$.  Since the gamma functions grow superexponentially we may replace these by $2^{-k}$, giving
\[\abs{\mathcal{J}_k(P(NC))} \ll_\delta \frac{(s_1s_2)^{\delta}}{2^k} + k^{-\frac{2}{3}},\]
hence
\[\abs{\mathcal{J}_k(P(NC))} \ll_\delta k^{-\frac{2}{3}}(s_1s_2)^{\delta}.\]
Also, from (\ref{size-of-c3}) and (\ref{det-like-H}), $\abs{\mathcal{C}_3(Nc_1, Nc_2, V)} \ll_\epsilon (ac)^{\frac{1}{2}+\epsilon} \ll (s_1s_2)^{1+2\epsilon}$, so taking $\epsilon = \delta$ we have
\[\mathcal{R}_{23}(Nc_1, Nc_2, V) \ll_\delta k^{-\frac{2}{3}} (s_1s_2)^{1+3\delta}.\]
Replacing $\delta$ by $\delta/3$ and recalling (\ref{si-in-known-parameters}) gives
\[\mathcal{R}_{23}(Nc_1, Nc_2, V) \ll_\delta k^{-\frac{2}{3}} N^{-2-2\delta} {\abs{T}}^{\frac{1}{2}+\frac{\delta}{2}} (c_1c_2)^{-1-\delta},\]
hence
\[R_{23} \ll_\delta \left(\frac{\abs{T}}{\abs{Q}}\right)^{\frac{k}{2}-\frac{3}{4}} {\abs{T}}^{\frac{1}{2}+\frac{\delta}{2}} k^{-\frac{2}{3}} N^{-\frac{5}{2}-2\delta+\epsilon} \sum_{c_1, c_2, V} c_1^{-\frac{1}{2}-\delta} c_2^{-2-\delta+\epsilon} (Nc_2, {}^tvTv)^{\frac{1}{2}}.\]
The sum over $c_1, c_2, V$ is once again the sum we dealt with for $R_{21}$, so again taking $\delta = \epsilon$ we have
\[R_{23} \ll_\epsilon \abs{T}^{\frac{k}{2}-\frac{1}{4}+\epsilon} N^{-2} k^{-\frac{2}{3}}.\]

\noindent Putting these three cases in to (\ref{rank2-summand-basic-bound}) we obtain the result.  \end{proof}

\section{The main theorem}\label{main-theorem}

Fix $d, \Lambda$ and a finite set of primes $S$.  Recall the definitions of the spaces $X_S$ and $Y_S$ from (\ref{XS-YS-definition}).  Recall also the measures $d\nu_{S, N, k}$ and $d\mu_{S}$ defined by (\ref{nu-def}) and (\ref{mu-def}) respectively.  Our aim was to prove Theorem \ref{local-equidistribution-theorem-qualitive}; that is, for any choice of $d$ and $\Lambda$, the measure $\nu_{S, N}$ converges weak-$*$ to the measure $\mu_S$ as $k$ and $N$ vary admissibly.  

\begin{proposition}\label{convergence-on-U_p}  Let $S$ be a finite set of primes, and let $l = (l_p)_{p \in S}$, $m = (m_p)_{p \in S}$ be tuples of non-negative integers.  Define $L = \prod_{p \in S} p^{l_p}$, $M = \prod_{p \in S} p^{m_p}$.  Let $\mathcal{S}_k(N)^*$ be an orthogonal basis of $\mathcal{S}_k(N)$ consisting of eigenforms for $\mathcal{H}_p$ when $p \in S$.  Then
\[\sum_{f \in \mathcal{S}_k(N)^*} \omega_{f, N, k} \prod_{p \in S} U_p^{l_p, m_p} (a_p(f), b_p(f)) = \delta(l,m) + O_{d, \epsilon}\left(N^{-1} k^{-\frac{2}{3}} L^{1+\epsilon}M^{\frac{3}{2}+\epsilon}\right),\] 
where
\[\delta(l, m) = \begin{cases} 1 & \text{if }l_p = m_p = 0\text{ for all }p \in S, \\ 0 & \text{otherwise,} \end{cases}\]
and the functions $U_p^{l_p, m_p} \in C(Y_S)$ are as in Theorem \ref{sugano-formula}.  \end{proposition}

\begin{proof}  Recall the definition of $a(d, \Lambda; f)$ given by (\ref{fourier-coeff-class-gp-average}).  Computing, using this definition for the first and third line and the crucial formula (\ref{local-bessels-in-terms-of-sugano}) for the second, 
\[\begin{aligned} \frac{\abs{a(d, \Lambda; f)}^2}{\langle f, f \rangle} \prod_{p \in S} U_p^{l_p, m_p}(a_p(f), b_p(f)) &= \frac{\overline{a(d, \Lambda; f)}}{\langle f, f \rangle} \sum_{c \in \Cl_d} \overline{\Lambda(c)}a(\mathsf{S}_c; f) \prod_{p \in S} U_p^{l_p, m_p}(a_p(f), b_p(f)) \\
&= \frac{\overline{a(d, \Lambda; f)}}{\langle f, f \rangle}\frac{L^{\frac{3}{2}-k}M^{2-k}\abs{\Cl_d}}{\abs{\Cl_d(M)}} \sum_{c \in \Cl_d(M)} \overline{\Lambda(c)} a(\mathsf{S}_c^{L, M}; f) \\
&= \frac{L^{\frac{3}{2}-k}M^{2-k}\abs{\Cl_d}}{\abs{\Cl_d(M)}} \sum_{\substack{c' \in \Cl_d \\ c \in \Cl_d(M)}} \Lambda(c')\overline{\Lambda(c)} \frac{\overline{a(\mathsf{S}_{c'}; f)}a(\mathsf{S}_c^{L, M}; f)}{\langle f, f \rangle}.\end{aligned}\] 
Including the full weight $\omega_{f, N, k}$ given by (\ref{weight-def}) and summing over our basis $\mathcal{S}_k(N)^*$ we obtain
\[\begin{aligned} &\sum_{f \in \mathcal{S}_k(N)^*} \omega_{f, N, k} \prod_{p \in S} U_p^{l_p, m_p}(a_p(f), b_p(f)) \\
&\qquad= \frac{\abs{\Cl_d} L^{\frac{3}{2}-k} M^{2-k}}{\abs{\Cl_d(M)}\vol(\Gamma_0(N) \backslash \mathbb{H}_2)} \sum_{\substack{c' \in \Cl_d \\ c \in \Cl_d(M)}} \Lambda(c')\overline{\Lambda(c)} c_k^{d, \Lambda} \sum_{f \in \mathcal{S}_k(N)^*} \frac{\overline{a(\mathsf{S}_{c'}; f)} a(\mathsf{S}_c^{L, M}; f)}{\langle f, f \rangle}. \end{aligned}\]
Using Corollary \ref{crly:applicable-version},
\begin{equation}\label{weighted-sugano-analytically} \begin{aligned} &\sum_{f \in \mathcal{S}_k(N)^*} \omega_{f, N, k} \prod_{p \in S} U_p^{l_p, m_p}(a_p(f), b_p(f)) \\
&\qquad= \frac{\abs{\Cl_d} M^{2-k}L^{\frac{3}{2}-k}}{\abs{\Cl_d(M)}} \frac{d_\Lambda}{2w(-d) \abs{\Cl_d}} \sum_{\substack{c' \in \Cl_d \\ c \in \Cl_d(M)}} \Lambda(c')\overline{\Lambda(c)} [\delta(c, c', L, M) + E(N, k; c, c', L, M)].\end{aligned}\end{equation}
If $LM=1$ then the right hand side of (\ref{weighted-sugano-analytically}) is
\[\frac{d_\Lambda}{2w(-d)\abs{\Cl_d}} \sum_{c, c' \in \Cl_d} \Lambda(c') \overline{\Lambda(c)} [\delta(c, c', 1, 1) + E(k, N; c, c', 1, 1)].\]
Using \cite{KowalskiSahaTsimerman2012} Lemma 3.7 (note that our $\delta$ includes the number of the $\GL_2(\mathbb{Z})$-automorphisms in its definition) we evaluate this as
\[1 + \frac{d_\Lambda}{2w(-d)\abs{\Cl_d}} \sum_{c, c' \in \Cl_d} \Lambda(c') \overline{\Lambda(c)}E(N, k; c, c', 1, 1) = 1 + O_{d, \epsilon}(N^{-1}k^{-\frac{2}{3}}).\]
If $LM>1$ then $\det(S_{c}^{L, M}) = \det(S_{c'})(LM)^2$ and it is clear that $\delta(c, c', L, M) = 0$.  So using Corollary \ref{crly:applicable-version} again the right hand side of (\ref{weighted-sugano-analytically}) is simply 
\[\frac{\abs{\Cl_d} M^{2-k} L^{\frac{3}{2}-k}}{\abs{\Cl_d(M)}} \frac{d_\Lambda}{2w(-d)\abs{\Cl_d}}\sum_{\substack{c' \in \Cl_d \\ c \in \Cl_d(M)}} \Lambda(c') \overline{\Lambda(c)} E(N, k; c, c', L, M) = O_{d, \epsilon}(N^{-1} k^{-\frac{2}{3}}L^{1 + \epsilon} M^{\frac{3}{2} + \epsilon}).\]  \end{proof}

\begin{proposition}\label{plancherel-integral-U_p}  Let $S$ be a finite set of primes, and let $l = (l_p)_{p \in S}$, $m = (m_p)_{p \in S}$ be tuples of non-negative integers.  Let $\mu_S$ be the measure on $Y_S$.  Then
\[\int_{Y_S} \prod_{p \in S} U_p^{l_p, m_p}(a_p, b_p) d\mu_S = \delta(l, m),\]
where $\delta(l, m)$ is as in Proposition \ref{convergence-on-U_p}. \end{proposition}
\begin{proof}  This is \cite{KowalskiSahaTsimerman2012} Proposition 4.2.  \end{proof}

\noindent  It is now simple to obtain the quantitative version of our local equidistribution statement:

\begin{proof}  [Proof of Theorem \ref{local-equidistribution-theorem-qualitive}]  By Weyl's criterion (\cite{IwaniecKowalski2004} \S21.1) it suffices to show that the claimed convergence holds for all $\varphi$ in a set of continuous functions whose linear combinations span $C(Y_S)$.  As $(l_p)_{p \in S}$ and $(m_p)_{p \in S}$ vary over all tuples of non-negative integers, $U_p^{l_p, m_p}$ describes such a family.  The result then follows immediately from Propositions \ref{convergence-on-U_p} and \ref{plancherel-integral-U_p}.  \end{proof}

\begin{theorem}  [Local equidistribution and independence, quantitative version]\label{local-equidistribution-theorem-quantitative}  Fix any $d$ and $\Lambda$, and finite set of primes $S$.  Let $\varphi = \prod_p \varphi_p$ be a product function on $Y_S$ such that $\varphi_p$ is a Laurent polynomial in $(a, b, a^{-1}, b^{-1})$ invariant under the action of the Weyl group generated by (\ref{weyl-group-generators}) and of total degree $d_p$ as a polynomial in $(a + a^{-1}, b + b^{-1})$.  Write $D = \prod_{p \in S} p^{d_p}$.  Then, for all $\epsilon>0$,
\[\sum_{f \in \mathcal{S}_k(N)^*} \omega_{f, k, N} \varphi((a_p(f), b_p(f))_{p \in S}) = \int_{Y_S} \varphi\: d\mu_S + O_{d, \epsilon}(N^{-1} k^{-\frac{2}{3}} D^{1 + \epsilon} \norm{\varphi}_\infty),\]
where $\norm{\varphi}_{\infty} = \max_{X_S} \abs{\varphi}$.  \end{theorem}
\begin{proof}   We may assume (by working with a smaller $S$ if necessary) that each $\varphi_p$ is non-constant (i.e. $d_p \geq 1$).  Since the functions $U_p^{l_p, m_p}$ linearly generate $C(Y_p)$
\[\varphi_p = \sum_{0 \leq l_p \leq e_p} \sum_{0 \leq m_p \leq f_p} \widehat{\varphi_p}(l_p, m_p) U_p^{l_p, m_p},\]
where at least one of $e_p, f_p$ is $\geq 1$.  Note that by Proposition \ref{plancherel-integral-U_p} 
\begin{equation}\label{mu-integral-of-expansion} \int_{Y_S} \varphi\: d\mu_S = \prod_{p \in S} \widehat{\varphi_p}(0, 0).\end{equation}  
Moreover,
\[\begin{aligned} \sum_{f \in \mathcal{S}_k(N)^*} \omega_{f, N, k} \varphi((a_p(f), b_p(f))_{p \in S}) &= \prod_{p \in S} \sum_{\substack{0 \leq l_p \leq e_p \\ 0 \leq m_p \leq f_p}} \widehat{\varphi_p}(l_p, m_p) \sum_{f \in \mathcal{S}_k(N)^*} \omega_{f, N, k} U_p^{l_p, m_p}(a_p(f), b_p(f)) \\
&= \prod_{p \in S} \widehat{\varphi_p}(0, 0) + N^{-1}k^{-2/3} R, \end{aligned}\]
where, using Proposition \ref{convergence-on-U_p}, we have the following bounds on $R$:  write $L_\varphi = \prod_{p \in S} p^{e_p}$, $M_\varphi = \prod_{p \in S} p^{m_p}$, then
\[R \ll_{\epsilon} \sum_{L \mid L_\varphi} \sum_{M \mid M_\varphi} L^{1+\epsilon} M^{3/2+\epsilon} \prod_{p \in S} \abs{\widehat{\varphi_p}(v_p(L), v_p(M))}.\]  
Comparing with (\ref{mu-integral-of-expansion}) it suffices to show from this that $R \ll D^{1+\epsilon} \norm{\varphi}_{\infty}$.  This is carried out in the proof of Theorem 1.6 of \cite{KowalskiSahaTsimerman2012} and we do not repeat the details.  \end{proof}

\section{Background on $L$-functions and low-lying zeros}\label{sctn:background-on-l-functions}

\noindent For the remainder of the paper we restrict to modular forms of squarefree level $N$, and take the weight $\omega_{f, N, k}$ to be defined with $d=4$ and $\Lambda = \mathbf{1}$.\\

\noindent \textbf{Background on $L$-functions.}  Given an irreducible automorphic representation $\pi$ of $\GSp_4$, one can form the Langlands $L$-function $L(s, \pi, r)$ for any representation of the dual group $r : \GSp_4(\mathbb{C}) \to \GL_n(\mathbb{C})$.  We take $n=4$ and the representation $r$ to be the tautological one, whence we obtain the so-called spin $L$-function of $\pi$.  We will also restrict our attention to representations $\pi$ which are self-dual, since the representations generated by modular forms with trivial character are self-dual.  For $p$ a finite prime we write the local Euler factor as
\[L_p(s, \pi) = \prod_{i=1}^4 (1-\alpha_i(p)p^{-s})\]
so that the (finite part of) the $L$-function is
\[L(s, \pi) = \prod_p L_p(s, \pi)^{-1}.\]
The $\alpha_i(p)$ are the local factors, defined via the local Langlands correspondence for $\GSp_4$.  At the unramified primes (those where $\pi_p$ is spherical) these are the Satake parameters.  Using the notation of \S\ref{the-equidistribution-problem} these are $(a_p(\pi), b_p(\pi)) = (\sigma(p), \sigma(p)\chi_1(p))$.  Thus labelling appropriately we have
\[\begin{aligned} \alpha_1(p) &= \alpha_2(p)^{-1} = a_p(\pi),\\
\alpha_3(p) &= \alpha_4(p)^{-1} = b_p(\pi). \end{aligned}\]
At the ramified primes (those where $\pi_p$ is not spherical) the $\alpha_i(p)$ can be zero; it is a delicate question to say precisely what the local factor are in these case.  For our consideration of low-lying zeros attached to these $L$-functions in \S\ref{sctn:low-lying-zeros} we will require some bounds on these quantities.  Whilst the Ramanujan conjecture, proved by Weissauer, provides the optimal bound for the local parameters at unramified places (certainly the most important case in general) for non-CAP representations, we are not aware of such results for ramified places in the literature.  We make the following assumption: if $\pi$ is non-CAP then there exists $0 \leq \theta<1/2$ such that
\begin{equation}\label{eqn:bound-towards-ramanujan}\abs{\alpha_i(p)} \leq p^\theta.\end{equation}
We suspect this might be known, expecially given that we are assuming squarefree level.  It certainly follows if we assume transfer of $\pi$ to $\GL_4$ (which has been proven for $N=1$ in \cite{PitaleSahaSchmidt2013}), as non-CAP representations will have cuspidal transfer so one can use \cite{MuellerSpeh2004} Proposition 3.3 (the ramified analogue of \cite{LuoRudnikSarnak1999} Theorem 2) to take $\theta = \frac{1}{2} - \frac{1}{4^2+1}$.\\

\noindent  However we must also take in to account some CAP representations, since the representations attached to Saito--Kurokawa lifts are so.  These are certain cuspidal automorphic representations of $\PGSp_4$ whose local factors do not satisfy the Ramanujan conjecture: at almost all places some of the local factors are as large as $p^{1/2}$.  For these representations the expected transfer to $\GL_4$ is no longer cuspidal (and in particular (\ref{eqn:bound-towards-ramanujan}) will not hold). It turns out that the ramified local factors for these representations are large enough that we have to handle these representations exceptionally.  We will explain our resolution of this issue in \S\ref{sctn:saito-kurokawa-lifts}.  Although we restrict to squarefree level to deal with this, we expect this issue should really be minor in any case.  At the unramified places the Saito--Kurokawa contribution is already handled in Theorem \ref{local-equidistribution-theorem-quantitative}.\\
 
\noindent We now continue with the definition of the $L$-function.  For the infinite place we have a gamma factor determined by the representation type of $\pi_\infty$.  When $\pi = \pi_f$ is an irreducible constituent of the representation generated by a Siegel cusp form $f$ of weight $k$ the gamma factor is
\begin{equation}\begin{aligned}\label{eqn:weight-k-gamma-factor}\gamma(s, \pi_f) &= (2\pi)^{-2s} \Gamma\left(s + \frac{1}{2}\right)\Gamma\left(s + k - \frac{3}{2}\right) \end{aligned}\end{equation}
We shall assume the existence of a ``nice $L$-function theory'': there exists an integer $q(\pi)$, divisible only by ramified primes of $\pi$, such that the completed $L$-function
\[\Lambda(s, \pi) = q(\pi)^{s/2} \gamma(s, \pi)L(s, \pi)\]
extends to a meromorphic function satisfying the functional equation
\[\Lambda(s, \pi) = \varepsilon(\pi) \Lambda(1-s, \pi).\]
Here $\varepsilon(\pi) \in \{\pm 1\}$ is determined by the local $\varepsilon$-factors, in turn defined by the local Langlands correspondence.  A ``nice $L$-function theory'' would follow from (\cite{Piatetski-Shapiro1997}), once it has been verified in all cases that the local factors defined there agree with those of defined by the local Langlands correspondence.  Given such an $L$-function we define the analytic conductor to be $C(\pi) = q(\pi)q_\infty(\pi)$, where $q(\pi)$ is the factor appearing in the functional equation, and for $\pi$ the representation generated by a weight $k$ Siegel modular form $q_\infty(\pi) := k^2$.\\

\noindent \textbf{Background on low-lying zeros.}  We are interested in the low-lying zeros of the $L(s, \pi)$ on the critical line $s=1/2$.  The key to this is an explicit formula: for example from \cite{IwaniecKowalski2004} Theorem 5.12 we have, for $h : \mathbb{R} \to \mathbb{R}$ an even Schwartz function with Fourier transform $\widehat{h}$,
\begin{equation}\label{eqn:explicit-formula} \begin{aligned}\sum_\rho h\left(\frac{\gamma}{2\pi}\right) &= \widehat{h}(0) \log q(\pi) + \frac{1}{2\pi} \int_\mathbb{R} \left(\frac{\gamma'}{\gamma}\left(\frac{1}{2}+it, \pi\right) + \frac{\gamma'}{\gamma}\left(\frac{1}{2}-it, \pi\right)\right)h\left(\frac{t}{2\pi} \right) dt \\
&\qquad-2\sum_p \log p  \sum_{m \geq 1} c(\pi, p^m) p^{-m/2} \widehat{h}\left(m \log p\right),\end{aligned}\end{equation}
where the sum on the left hand side is over zeros $\rho = \frac{1}{2} + i\gamma$, and the double sum on the right involves moments of the local factors of the representation:
\begin{equation}\label{eqn:moments-of-local-factors-def} c(\pi, p^m) = \sum_{i=1}^4 \alpha_i(p)^m.\end{equation}
However, in order to have enough zeros to do a meaningful statistical study we will average over a suitable family of representations $\pi$ as above, which we now describe:  let $\mathcal{S}_k(N)^{\#}$ be an orthogonal basis of $\mathcal{S}_k(N)$ consisting of eigenfunctions of all $T(p)$ and $T_1(p^2)$ when $p \nmid N$.  Then for any $f \in \mathcal{S}_k(N)^\#$ we have an associated cuspidal automorphic representation of $\GSp_4(\mathbb{A})$.  Let $\pi_f$ be any irreducible consitutent of this, and write $C(\pi_f)$ be the analytic conductor as above.\\

\noindent We will consider the representations we obtain as we vary $f \in \mathcal{S}_k^{\#}(N)$, in particular there is no restriction to ``newforms''.  It may be possible to set up the problem in terms of newforms using the description in \cite{Schmidt2005a}, but we opt not to so that we can apply Theorem \ref{local-equidistribution-theorem-quantitative} directly.  As described in the introduction, this means that as we vary over $f \in \mathcal{S}_k(N)^{\#}$, the (isomorphism class of) a representation may be repeated.\\

\noindent In any case when working with forms that are not necessarily ``new'' the $q(\pi_f)$ is by no means the same for each element in our family.  It is therefore prudent to introduce a log-average conductor, defined by
\[\log C_{k, N} =  \frac{1}{\sum_{f \in \mathcal{S}_k(N)^\#} \omega_{f, k, N}} \sum_{f \in \mathcal{S}_k(N)^\#} \omega_{f, k, N} \log C(\pi_f).\]
Recall that $N$ is squarefree.  From Table 3 of \cite{Schmidt2005a}, particularly the fact that the conductors of representations which have invariant vectors for $P_1$ (the local version of $\Gamma_0(N)$) have conductor $\leq 2$, it easily follows that $C_{k, N} \ll N^2$.  By using the fact that representations containing newforms for $P_1$ have conductor $\geq 1$ one can argue by induction to obtain a lower bound and deduce that 
\begin{equation}\label{eqn:conductor-is-polynomial}\log C_{k, N} \asymp \log N.\end{equation}
Finally, let $\Phi$ be an even Schwartz function (the Fourier transform of which we will eventually assume to have sufficiently small compact support), and let
\[D(k, N; \Phi) = \frac{1}{\sum_{f \in \mathcal{S}_k(N)^\#} \omega_{f, k, N}} \sum_{f \in \mathcal{S}_k(N)^\#} \omega_{f, k, N} D(\pi_f; \Phi),\]
where
\[D(\pi_f; \Phi) = \sum_\rho \Phi\left(\frac{\gamma}{2\pi} \log C_{k, N} \right).\]
The quantity $D(k, N; \Phi)$ measures the low-lying zeros of the $L$-functions associated to the representations in our family.  \\

\section{Saito--Kurokawa lifts}\label{sctn:saito-kurokawa-lifts}

\noindent Recall that we stated in the preceding section that certain representations, namely CAP representations, require special treatment.  To this end we begin by recalling the description of Saito--Kurokawa lifts from \cite{Schmidt2005b}; at the end of this section we will show that these essentially exhaust all problem cases in our context.  First take an irreducible cuspidal automorphic representation $\pi$ of $\PGL_2$, and assume that $\pi$ corresponds to a holomorphic cusp form of weight $2k-2$, so that $\pi_\infty$ is the discrete series representation with lowest weight $2k-2$.  Let $\Sigma$ be the set of places at which $\pi$ is a discrete series.  We pick a set $S$ with $\infty \in S \subset \Sigma$ such that $(-1)^{\abs{S}} = \varepsilon(\pi)$, with the usual $\varepsilon$-factor of the cuspidal automorphic representation $\pi$.  Define a representation $\pi_S$ of $\GL_2$ by
\[\pi_S = \begin{cases} 1_v & \text{ if } v \notin S, \\ \St_v & \text{ if } v \in S, \end{cases}\]
where $\St_v$ denotes the Steinberg representation.  At the infinite place this is taken to mean the lowest discrete series representation.  $\pi_S$ is in fact a constituent of a globally induced representation, so it is automorphic.  For any choice of $S$ as above a lift $\Pi(\pi \times \pi_S)$ can be defined; it is an irreducible cuspidal automorphic representation of $\PGSp_4$.\\

\noindent Most importantly for us is a case when $\pi$ corresponds to a newform $g \in \mathcal{S}_{2k-2}^{(1)}(M)$ of squarefree level $M$ (the superscript $^{(1)}$ refers to modular forms of degree $1$, i.e. on $\SL_2$) considered in detail in \cite{Schmidt2007}, where $S$ is chosen to be the set of primes $p \mid M$ for which the newform $g$ has Atkin--Lehner eigenvalue $-1$.  The lift $\SK(\pi) = \Pi(\pi \times \pi_S)$ is then an irreducible cuspidal automorphic representation of $\PGSp_4$.  The local component $\SK(\pi)_\infty$ is the holomorphic discrete series representation of $\PGSp_4(\mathbb{R})$ with scalar minimal $K$-type of weight $(k, k)$.  This is the $\infty$-type of the representation attached to a holomorphic Siegel modular form; in fact it follows from Theorem 5.2 of \cite{Schmidt2007} that there is a unique (up to scalars) modular form $f \in \mathcal{S}_k(M)$ such that $\Phi_f$ generates the representation $\Pi(\pi \times \pi_S)$.  Indeed this function $f$ is the classical Saito--Kurokawa lift $\SK(g)$ of the newform $g$ as defined in \cite{ManickamRamakrishnanVasudevan1993}.  Using results from \cite{Piatetski-Shapiro1983}, it is also described in the proof of Theorem 5.2 of \cite{Schmidt2007} how the representation $\SK(\pi)$ occurs with multiplicity one in the space of automorphic forms on $\PGSp_4$.\\

\noindent Our $L$-functions however are formed from $\mathcal{S}_k(N)^{\#}$ and therefore we must take in to account that whilst there is a unique modular form $f$ of level $M \mid N$ whose representation is $\pi_f$, there will be more forms of level $N$ describing the same representation.  We shall now count how many vectors in the representation $\SK(\pi)$ give rise to modular forms of level $N$:

\begin{lemma}\label{lem:number-of-SK-oldforms}  Let $\pi$ be the cuspidal automorphic representation of $\PGL_2$ associated to a classical newform $g$ of level $M \mid N$ ($N$ squarefree), and $\SK(\pi)$ its Saito--Kurokawa lift.  Then the vector space consisting of modular forms $f \in \mathcal{S}_k(N)$ such that $\Phi_f \in \SK(\pi)$ has dimension $3^r$, where $r$ is the number of prime divisors of $N/M$.\end{lemma}
\begin{proof}  Set 
\[P_1(p) = \left\{ \left(\begin{matrix} A & B \\ C & D \end{matrix}\right) \in \GSp_4(\mathbb{Z}_p);\: C \equiv 0 \bmod N\mathbb{Z}_p \right\}.\]
To count the number of vectors in $\SK(\pi)$ which come from level $N$ modular forms it suffices to count the number of vectors invariant under $\prod_p P_1(p) = \prod_{p \mid N} P_1(p) \prod_{p \nmid N} \GSp_4(\mathbb{Z}_p)$.  It is shown in \cite{Schmidt2007} that $\Pi(\pi \times \pi_S)_p$ has an essentially unique (i.e. up to scalars) vector under the right-action of $P_1(p)$ for each $p \mid M$.  For $p \nmid N$ there is an essentially unique vector for the action of $\GSp_4(\mathbb{Z}_p)$.  The case $p \mid N/M$ is not written down in the work of Schmidt but follows easily from it: we know when $p \mid N/M$ that $\pi_p = \pi(\chi, \chi^{-1})$ is a spherical principal series representation of $\PGL_2(\mathbb{Q}_p)$, and by \cite{Schmidt2005b} \S7 we have $\Pi(\pi \times \pi_S)_p \simeq \chi 1_{\GL_2} \rtimes \chi^{-1}$ (in the notation of \cite{SallyTadic1993}).  By Table 3 in \cite{Schmidt2005a} this has three linearly independent vectors invariant under $P_1(p)$.  Piecing this together for each prime dividing $N/M$ we obtain the statement of the lemma.\end{proof}

\noindent Lemma \ref{lem:number-of-SK-oldforms} does not give us the modular forms $f \in \mathcal{S}_k(N)$ explicitly, but we can easily provide a basis for the vector space it considers via classical means:

\begin{lemma}\label{lem:basis-for-SK}  Let $\pi$ be the cuspidal automorphic representation of $\PGL_2$ associated to a classical newform $g$ of level $M$, and $\SK(\pi)$ its Saito--Kurokawa lift.  Let $\SK(g)$ be the classical Saito--Kurokawa lift of $g$.  Define\begin{footnote}{The subscripts are thus to be consistent with the notation of \cite{Schmidt2005a}.}\end{footnote} the following maps on Fourier coefficients:
\[f(Z) = \sum_{T > 0} a(T; f) e(\tr(TZ)) \mapsto \begin{cases} \sum_{T > 0} a(T; f) e(\tr(pTZ)) =: T_1(p, f), \\
\sum_{T > 0} a(pT; f) e(\tr(pTZ)) =: T_3(p, f).\end{cases}\]
Define a set for squarefree multiples of $M$ inductively as follows: $\mathcal{B}_M = \{\SK(g)\}$, and if $N'$ is a squarefree multiple of $M$ and $p \nmid N'$ is a prime set $\mathcal{B}_{N'p} = \{f, T_1(p, f), T_3(p, f);\: f \in \mathcal{B}_{N'}\}$.  Then, for any squarefree multiple $N$ of $M$, $\mathcal{B}_{N}$ is a basis for the space of modular forms $f \in \mathcal{S}_k(N)$ such that $\Phi_f \in \SK(\pi)$.\end{lemma}
\begin{proof}  It suffices to prove that $\mathcal{B}_{N}$ is a linearly independent set since if so it by construction has the dimension required by Lemma \ref{lem:number-of-SK-oldforms}.  By writing out a dependence relation and picking off leading Fourier coefficients we see that proving linear independence boils down to showing that there are no nontrivial dependence relations of the form
\begin{equation}\label{eqn:dependence-relation}\sum_{e \mid d} c_e a(eT; \SK(g)) = 0, \text{ for all }T>0\end{equation}
where $d$ is a fixed divisor of $N/M$.  Suppose we have such a nontrivial relation involving a minimal number of divisors $e$.  Now for any $p \nmid M$ we have that $\SK(g)$ is an eigenfunction of $T(p)$, hence there is $\lambda \in \mathbb{C}$ such that
\[\lambda a(T; \SK(g)) = a(pT; \SK(g)) + p^{k-1} a(T; \SK(g)) + p^{2k-3}a(T; \SK(g)).\]
This follows from using the formula for the action of $T(p)$ on Fourier expansions and the fact that the Fourier coefficients of a Saito--Kurokawa lift depend only on the determinant of the indexing matrix.  Repeatedly using this allows us to derive from (\ref{eqn:dependence-relation}) a dependence relation involving fewer $e$, and thence a contradiction.  \end{proof}

\noindent Now we use a result of Brown and the structure of the basis in \ref{lem:basis-for-SK} to show that the weights $\omega_{f, k, N}$ are small for any $f$ this basis:

\begin{theorem}\label{thm:brown}[Brown, \cite{Brown2007} Theorem 1.1]  Let $M$ be a squarefree positive integer with $m$ prime divisors, $g \in \mathcal{S}_k^{(1)}(M)$ be a newform, and let $\SK(g) \in \mathcal{S}_k(M)$ be the classical Saito--Kurokawa lift of $g$.  Write $\Sh(g)$ for the Shimura lift of $g$, and $a(n; \Sh(g))$ for its Fourier coefficients.  Let $D<0$ be a fundamental discriminant such that $\gcd(M, D)=1$ and $a(\abs{D}, \Sh(g)) \neq 0$.  Then
\begin{equation}\label{eqn:SK-inner-product-relation}\langle \SK(g), \SK(g) \rangle = \mathcal{B}_{k, M}  \frac{\abs{a(\abs{D}; \Sh(g))}^2 L(1, \pi_g)}{\pi \abs{D}^{k-\frac{3}{2}} L(\frac{1}{2}, \pi_g \times \chi_D)} \langle g, g\rangle,\end{equation}
where 
\[\mathcal{B}_{k, M} = \frac{M^k(k - 1) \prod_{i=1}^m (p_i^4+1)}{2^{m+3} 3[\Sp_4(\mathbb{Z}):\Gamma_0(M)][\Gamma_0(M):\Gamma_0(4M)]}.\]
\end{theorem}

\begin{corollary}\label{crlry:to-brown}  Let $M$ be a squarefree positive integer and $g \in \mathcal{S}_k^{(1)}(M)$ be a newform, and let $\SK(g) \in \mathcal{S}_k(M)$ be the classical Saito--Kurokawa lift of $g$.  Let $\mathcal{S}_k^{(1)}(M)^{\#}_{\text{new}}$ denote an orthogonal basis for the space of newforms.  Then, for any $\delta > 0$,
\[\sum_{g \in \mathcal{S}_k^{(1)}(M)^{\#}_{\text{new}}} \omega_{\SK(g), M, k} \ll_\delta \frac{1}{M^{5-\delta} k^{2-\delta}}.\]\end{corollary}
\begin{proof}  Let $g \in \mathcal{S}_k^{(1)}(M)^{\#}_{\text{new}}$, and assume for now $a(1_2; \SK(g)) \neq 0$.  By the construction of the classical Saito--Kurokawa lifting we have $a(4; \Sh(g)) = a(1_2; \SK(g))$, so we can apply Theorem \ref{thm:brown} with $D=-4$.  Substituting this in to the formula for $\omega_{\SK(g), M, k} = \omega_{\SK(g), M, k}^{4, \mathbf{1}}$ we have
\[\omega_{\SK(g), M, k} = \frac{\pi^2}{2\vol(\Gamma_0^{(2)}(M) \backslash \mathbb{H}_2) \mathcal{B}_{k, M} (k-2)}\frac{\Gamma(2k-3)}{(4\pi)^{2k-3}\langle g, g\rangle}\frac{L(\frac{1}{2}, \pi_g \times \chi_D)}{L(1, \pi_g)}.\]
If $a(1_2; \SK(g)) = 0$ then clearly the weight is zero.  In any case the sum we are trying to bound is majorized by a constant (depending on $k$ and $M$) multiplied by
\[\sum_{g \in \mathcal{S}_k^{(1)}(M)_{\text{new}}^{\#}} \frac{\Gamma(2k-3)}{(4\pi)^{2k-3} \langle g, g\rangle} \frac{L(\frac{1}{2}, \pi_g \times \chi_D)}{L(1, \pi_g)}.\]
We can now argue as in \cite{KowalskiSahaTsimerman2012} \S5.3 (where $M=1$) to see that this sum is $\ll \log(Mk)$.  Note that the factor of $[\Sp_4(\mathbb{Z}):\Gamma_0(M)]$ cancels out the normalisation in $\vol(\Gamma_0(M) \backslash \mathbb{H}_n)$, but the $M^k$ in the numerator and our ubiquitous assumption that $k \geq 6$ give us (after sacrificing a power of $M$ to the $2^{m+3}$ in the denominator) the claimed bound.  \end{proof}

\noindent Finally we must show that the Saito--Kurokawa lifts exhaust all problematic cases.  Thus suppose $f \in \mathcal{S}_k(N)^{\#}$ is such that $\pi_f$ has a local parameter with absolute value $p^{1/2}$ at some prime $p \mid N$.  We will show that there exists an irreducible cuspidal automorphic representation $\pi$ of $\PGL_2$, corresponding to a newform $g \in \mathcal{S}_{2k-2}^{(1)}(M)$, such that $\Phi(f) \in \SK(\pi)$.\\

\noindent  By our assumption (\ref{eqn:bound-towards-ramanujan}) $\pi_f$ is CAP -- in fact it follows from \cite{PitaleSchmidt2009} Corollary 4.5 that $\pi_f$ is associated to the Siegel parabolic $P$.  Fix an additive character $\psi$ of $\mathbb{Q} \backslash \mathbb{A}$, and write $\theta(\cdot, \psi)$ for the theta lifting from $\widetilde{\SL}_2$ to $\PGSp_4$.  Then by \cite{Piatetski-Shapiro1983} Theorem 2.2 $\pi_f = \theta(\widetilde{\pi}, \psi)$ for some irreducible cuspidal automorphic representation $\widetilde{\pi}$ of $\widetilde{\SL}_2$.  The representation $\widetilde{\pi}$ is not $\psi$-generic (c.f. \cite{Piatetski-Shapiro1983} Theorem 2.4), which implies that it does not participate in the theta correspondence with $\PGL_2$.  \\

\noindent On the other hand, let $S$ be the (finite) set of places at which $\widetilde{\pi}_v$ is the non-generic element in the fiber of the local Waldspurger correspondence between $\widetilde{\SL}_2$ and $\PGL_2$.  Replacing $\widetilde{\pi}_v$ with the generic element in the fiber we will obtain a globally $\psi$-generic representation of $\widetilde{\SL}_2$ which does have a non-vanishing theta lift to $\PGL_2$; write $\pi$ for this lift.  By the definition in \cite{Schmidt2005b} (and multiplicity one for theta lifts from $\widetilde{\SL}_2$) we have $\pi_f = \Pi(\pi \times \pi_S)$, with $S$ as above.  \\

\noindent It remains to see that $\pi$ in fact corresponds to a holomorphic newform $g \in \mathcal{S}_{2k-2}^{(1)}(M)$ where $M \mid N$ (the choice of $S$ is then forced to be the one defining $\SK(\pi)$ by table (30) of \cite{Schmidt2007}).  By examining Table 2 of \cite{Schmidt2005b} we easily deduce that $\pi$ has the correct $\infty$-type (and that $\infty \in S$) by knowing the $\infty$-type $\pi_f$.  Similarly knowing that all the local components of $\pi_f$ must have Iwahori-spherical vectors we deduce that $\pi$ is nowhere supercuspidal.  Finally we see that the set of finite primes at which $\pi$ is a discrete series is a subset of the set of finite primes at which $\pi_f$ is not a principal series.  Thus $\pi$ corresponds to a holomorphic newform $g$ as above.

\begin{remark}  The preceding paragraph only shows that our problem cases are contained in the Saito--Kurokawa cases.  Certain Saito--Kurokawa representations may not be a problem: for example an elliptic modular form of squarefree level with all Atkin--Lehner eigenvalues equal to $-1$ will have small local factors at ramified primes.  It will have large local factors at unramified primes, but these are dealt with by Theorem \ref{local-equidistribution-theorem-quantitative}.\end{remark}

\begin{corollary}\label{crlry:bounding-weights-on-problem-cases}  Let $\mathcal{P} = \{f \in \mathcal{S}_k(N)^{\#}; \text{(\ref{eqn:bound-towards-ramanujan}) does not hold for }\pi_f\}$.  Then, for any $\delta > 0$,
\[\sum_{f \in \mathcal{P}} \omega_{f, N, k} \ll_\delta \frac{1}{N^{3}k^{2-\delta}}\]\end{corollary}
\begin{proof}  Let $f \in \mathcal{P}$.  By the preceding discussion we know that there exists an irreducible cuspidal automorphic representation $\pi$ of $\PGL_2$ corresponding to a newform $g$ such that $\Phi_f \in \SK(\pi)$.  Thus $f$ is a sum of the basis elements of $\mathcal{B}_N$ from Lemma \ref{lem:basis-for-SK}.  Normalising (recall $\omega_{f, N, k}$ is invariant under rescaling) we may assume that the coefficient of $\SK(g)$ (if nonzero) is one.  Since all elements $f'$ other than $\SK(g)$ of the basis clearly have $a(1_2; f') = 0$, and hence $\omega_{f', N, k} = 0$, it follows that $\omega_{f, N, k}$ is either zero (if the coefficient of $\SK(g)$ is) or we have $\omega_{f, N, k} = \omega_{\SK(g), N, k}$.  The result then follows from Corollary \ref{crlry:to-brown} and the fact that $\omega_{\cdot, N, k} \asymp \frac{1}{(N/M)^3}\omega_{\cdot, M, k}$.\end{proof}
\section{Low lying zeros}\label{sctn:low-lying-zeros}

\noindent We now proceed with the proof of Theorem \ref{thm:low-lying-zeros}, beginning with the computations at the archimedean place.  If $f \in \mathcal{S}_k(N)^*$ then the gamma factor of the $L$-function of the representation $\pi_f$ is given by (\ref{eqn:weight-k-gamma-factor}).  As before let $\Phi$ be an even Schwartz function, and now consider the expression
\[\begin{aligned} &\frac{1}{2\pi} \int_\mathbb{R} \left(\frac{\gamma'}{\gamma}\left(\frac{1}{2}+it, \pi_f\right) + \frac{\gamma'}{\gamma}\left(\frac{1}{2}-it, \pi_f\right)\right)\Phi\left(\frac{t}{2\pi} \log C_{k, N} \right) dt \\
&= \frac{1}{\log C_{k, N}}\int_\mathbb{R} \left(\frac{\gamma'}{\gamma}\left(\frac{1}{2} + \frac{2 \pi i x}{\log C_{k, N}}, \pi_f\right) + \frac{\gamma'}{\gamma}\left(\frac{1}{2} - \frac{2 \pi i x}{\log C_{k, N}}, \pi_f\right)\right)\Phi(x) dx.\end{aligned}\]
Arguing from (\ref{eqn:weight-k-gamma-factor}) as in \cite{DuenezMiller2006} we see that 
\[\frac{1}{\log C_{k, N}}\int_\mathbb{R} \left(\frac{\gamma'}{\gamma}\left(\frac{1}{2} + \frac{2 \pi i x}{\log C_{k, N}}, \pi_f\right) + \frac{\gamma'}{\gamma}\left(\frac{1}{2} - \frac{2 \pi i x}{\log C_{k, N}}, \pi_f\right)\right)\Phi(x) dx = \widehat{\Phi}(0)\frac{\log k^2}{\log C_{k, N}} + O\left(\frac{1}{\log  C_{k, N}} \right).\]
Now setting $h(x) = \Phi(x \log C_{k, N})$ (and hence $\widehat{h}(t) = \frac{1}{\log C_{k, N}} \widehat{\Phi}\left(\frac{t}{\log C_{k, N}}\right)$) in the explicit formula (\ref{eqn:explicit-formula}), using the above archimedean computation and
\[\widehat{\Phi}(0)\frac{\log q(\pi_f)}{\log C_{k, N}} + \widehat{\Phi}(0)\frac{\log k^2}{\log C_{k, N}}  + O\left(\frac{1}{\log  C_{k, N}} \right) = \frac{\log C_{\pi_f}}{\log C_{k, N}}\widehat{\Phi}(0) + O\left(\frac{1}{\log C_{k, N}}\right),\]
we get
\[\sum_\rho \Phi\left(\frac{\gamma}{2 \pi} \log C_{k, N}\right) = \frac{\log C_{\pi_f}}{\log C_{k, N}}\widehat{\Phi}(0) - \frac{2}{\log C_{k, N}} \sum_p \log p  \sum_{m \geq 1} c(\pi, p^m) p^{-m/2} \widehat{\Phi}\left(\frac{m \log p}{\log C_{k, N}}\right) + O\left(\frac{1}{\log C_{k, N}}\right).\]
Averaging over $f \in \mathcal{S}_k(N)^\#$ we therefore obtain
\begin{equation}\label{eqn:average-explicit-formula} D(k, N; \Phi) = \widehat{\Phi}(0) - \frac{1}{\sum_{f} \omega_{f, k, N}} \sum_{f}\omega_{f, k, N} \frac{2}{\log C_{k, N}} \sum_p \log p \sum_{m \geq 1} c(\pi_f, p^m)p^{-m/2} \widehat{\Phi}\left(\frac{m \log p}{\log C_{k, N}}\right) + O\left(\frac{1}{\log C_{k, N}}\right)\end{equation}
It remains to deal with the term involving the triple sum.  It is not difficult to see that for each $m \geq 3$ the sum over primes (even without the cutoff provided by $\widehat{\Phi}$) is finite and therefore the whole term can be absorbed in to the $O(1/\log C_{k, N})$.  Thus it suffices to estimate the sum over primes when $m=1$ and $m=2$.  \\

\noindent First consider $m=1$.  When $p$ is an unramified prime we argue as in \cite{KowalskiSahaTsimerman2012}: use the definition (\ref{eqn:U_p-definition}) and Proposition \ref{convergence-on-U_p} to see 
\[\begin{aligned} \frac{1}{\sum_f \omega_{f, k, N}} \sum_{f} \omega_{f, k, N} c(\pi_f, p) &= \frac{1}{\sum_f \omega_{f, k, N}}\sum_f \omega_{f, k, N}\left( U_p^{1, 0}(a_p(\pi_f), b_p(\pi_f)) + \lambda_p p^{-1/2} \right) \\
&= \lambda_p p^{-1/2} + O_\epsilon(N^{-1} k^{-2/3} p^{1+\epsilon}).\end{aligned}\]
When $p$ is a ramified prime, using Corollary \ref{crlry:bounding-weights-on-problem-cases} (and its notation)
\[\begin{aligned} \abs{\frac{1}{\sum_{f \in \mathcal{S}_k(N)^\#} \omega_{f, k, N}} \sum_{f \in \mathcal{S}_k(N)^\#} \omega_{f, k, N} c(\pi_f, p)} &\leq \frac{4}{\sum_{f \in \mathcal{S}_k(N)^\#} \omega_{f, k, N}} \left[\sum_{\substack{f \in \mathcal{S}_k(N)^\# \\ f \notin \mathcal{P}}} \omega_{f, k, N} p^\theta + \sum_{\substack{f \in \mathcal{S}_k(N)^\# \\ f \in \mathcal{P}}} \omega_{f, k, N} p^{1/2}\right] \\
&\ll 4p^\theta + \frac{p^{1/2}}{N^{3}k^{2-\delta}}.\end{aligned}\]
Thus, assuming that $\widehat{\Phi}$ is supported in $[-\alpha, \alpha]$,
\begin{equation}\label{eqn:m=1-collected}\begin{aligned}&\frac{1}{\sum_{f} \omega_{f, k, N}} \sum_{f}\omega_{f, k, N} \frac{2}{\log C_{k, N}} \sum_p \log (p)  c(\pi_f, p)p^{-1/2} \widehat{\Phi}\left(\frac{\log p}{\log C_{k, N}}\right) \\
&\qquad= \frac{2}{\log C_{k, N}} \left(\sum_{p\nmid N} \frac{\lambda_p \log p}{p} \widehat{\Phi}\left(\frac{\log p}{\log C_{k, N}}\right) + O_\epsilon\left(\frac{1}{N k^{2/3}}\sum_{\substack{p \leq C_{k, N}^{\alpha}}} p^{\frac{1}{2} + \epsilon}\right) + O\left(\sum_{p \mid N} \log(p) p^{\left(\theta-\frac{1}{2}\right)}\right)\right).\end{aligned}\end{equation}
We have left out the contribution at ramified primes from $f \in \mathcal{P}$ because this is clearly negligible.  For the remaining sum over ramified primes, the hypothesis $\theta<1/2$ and the fact that $\#\{p \mid N\} = o(\log N)$ show that the sum is $o(\log N)$.  By the hypothesis (\ref{eqn:conductor-is-polynomial}) this is in turn $o(\log C_{k, N})$, and so the sum over $p \mid N$ is negligible due to the presence of the $\frac{1}{\log C_{k, N}}$ factor in front.  By choosing $\alpha$ small enough we will show that the second term is negligible as well.  For the first term note that $\lambda_p$ takes the value $0$ or $2$ each on sets of primes of asymptotic density $1/2$, so by the prime number theorem
\[\begin{aligned}\frac{2}{\log C_{k, N}} \sum_p \frac{\lambda_p \log p}{p} \widehat{\Phi}\left(\frac{\log p}{\log C_{k, N}}\right) &= 2\int_{1}^\infty \widehat{\Phi}\left(\frac{\log x}{\log C_{k, N}}\right) \frac{1}{\log C_{k, N}} \frac{dx}{x} + o(1) \\
&= 2\int_0^\infty \widehat{\Phi}(x) dx + o(1) \\
&= \Phi(0) + o(1)\end{aligned}\]
where the last equality follows from the factor that $\Phi$ is even.  Now the left hand side is the same as the first sum in (\ref{eqn:m=1-collected}) except that we imposed the restriction $p \nmid N$ in the latter: the difference between the two is easily seen to be $O\left(\frac{1}{\log C_{k, N}}\right)$ (remembering the constant factor $\frac{1}{\log(C_{k, N})}$ in front), so we conclude
\[\begin{aligned}&\frac{1}{\sum_{f} \omega_{f, k, N}} \sum_{f}\omega_{f, k, N} \frac{2}{\log C_{k, N}} \sum_p \log (p)  c(\pi_f, p)p^{-1/2} \widehat{\Phi}\left(\frac{\log p}{\log C_{k, N}}\right) \\
&\qquad= {\Phi}(0) + O_\epsilon\left(\frac{1}{\log C_{k, N} N k^{2/3}}\sum_{\substack{p \leq C_{k, N}^{\alpha}}} p^{\frac{1}{2} + \epsilon}\right) + O\left(\frac{1}{\log C_{k, N}}\right).\end{aligned}\]

\noindent Next consider $m=2$.  When $p$ is unramified we again argue as in \cite{KowalskiSahaTsimerman2012}: begin with the formula
\[c(\pi_f, p^2) = U_p^{2, 0}(a_p(\pi_f), b_p(\pi_f)) + \frac{\lambda}{\sqrt{p}}U_p^{1, 0}(a_p(\pi_f), b_p(\pi_f)) - \tau(a_p(f), b_p(f)) - 1 - \frac{1}{p}\left(\frac{d}{p}\right).\]
Averaging this over $f$ with the help of Proposition \ref{convergence-on-U_p} we have
\[\frac{1}{\sum_{f} \omega_{f, k, N}} \sum_f \omega_{f, k, N} c(\pi_f, p^2) = -1 - \frac{1}{\sum_{f} \omega_{f, k, N}} \sum_f \omega_{f, k, N} \tau(a_p(f), b_p(f)) + O_\epsilon\left(\frac{p^{2+\epsilon}}{Nk^{2/3}}\right) + O_\epsilon\left(\frac{p^{1+\epsilon}}{Nk^{2/3}}\right) + O\left(\frac{1}{p}\right).\]
Appealing to the definitions \ref{eqn:U_p-definition} and Proposition \ref{convergence-on-U_p} with $U_p^{0,1}$ we have that
\[\frac{1}{\sum_{f} \omega_{f, k, N}} \sum_f \omega_{f, k, N} \tau(a_p(f), b_p(f)) = O_\epsilon\left(\frac{p^{\frac{3}{2} + \epsilon}}{Nk^{2/3}}\right).\]
For the ramified primes we argue as before and obtain the same result with $p^\theta$ replaced by $p^{2\theta}$ in the first term on the RHS, and $p^{1/2}$ replace by $p$ in the second.  Again the ramified contribution from $f \in \mathcal{P}$ is clearly negligible and we obtain
\[\begin{aligned}&\frac{1}{\sum_{f} \omega_{f, k, N}} \sum_{f}\omega_{f, k, N} \frac{2}{\log C_{k, N}} \sum_p \log (p)  c(\pi_f, p^2)p^{-1} \widehat{\Phi}\left(\frac{2\log p}{\log C_{k, N}}\right) \\
&\qquad= \frac{2}{\log C_{k, N}} \left(-\sum_{p\nmid N} \frac{\lambda_p \log p}{p} \widehat{\Phi}\left(\frac{2\log p}{\log C_{k, N}}\right) + O_\epsilon\left(\frac{1}{N k^{2/3}}\sum_{\substack{p \leq C_{k, N}^{\alpha/2}}} p^{1 + \epsilon}\right) + O\left(\sum_{p \mid N} \log(p) p^{\left(\theta-1\right)}\right)\right).\end{aligned}\]
The sum over $p \mid N$ is even more negligible than before.  We postpone choosing $\alpha$ sufficiently small for a little longer and consider the main term, which similarly to before is a negligible distance from
\[-\frac{2}{\log C_{k, N}}\sum_p \frac{\lambda_p \log p}{p} \widehat{\Phi}\left(\frac{2 \log p}{\log C_{k, N}}\right) = -\frac{1}{2} \Phi(0) + o(1)\]
(using the prime number theorem as before).  Thus
\[\begin{aligned}&\frac{1}{\sum_{f} \omega_{f, k, N}} \sum_{f}\omega_{f, k, N} \frac{2}{\log C_{k, N}} \sum_p \log (p)  c(\pi_f, p^2)p^{-1} \widehat{\Phi}\left(\frac{2\log p}{\log C_{k, N}}\right)\\
&= -\frac{1}{2}\Phi(0) + O_\epsilon\left(\frac{1}{\log C_{k, N} N k^{2/3}}\sum_{\substack{p \leq C_{k, N}^{\alpha/2}}} p^{1 + \epsilon}\right) + O\left(\frac{1}{\log C_{k, N}}\right).\end{aligned}\]
Finally it remains to choose $\alpha$ small enough such that the two sums
\[\frac{1}{\log C_{k, N} N k^{2/3}} \sum_{p \leq C_{k, N}^{\alpha/2}} p^{1 + \epsilon} = O\left(\frac{C_{k, N}^{\alpha+\epsilon}}{\log C_{k, N} N k^{2/3}}\right) \]
and
\[\frac{1}{\log C_{k, N} N k^{2/3}} \sum_{p \leq C_{k, N}^{\alpha}} p^{\frac{1}{2} + \epsilon} = O\left(\frac{C_{k, N}^{\frac{3\alpha}{2}+\epsilon}}{\log C_{k, N} N k^{2/3}}\right)\] 
are, say, $O(1/\log C_{k, N})$.  Since $C_{k, N} \ll N^2k^2$ we can do this with $\alpha < 2/9$.\\

\noindent  We end with a few remarks regarding Theorem \ref{thm:low-lying-zeros}.  Firstly, it should be possible to improve the range of $\alpha$ (which is typically desirable in low-lying zeros questions) with better estimation in the above.  If one were to study families with orthogonal symmetry then one would require $\alpha > 1$ to distinguish the type of orthogonal symmetry (c.f. \cite{IwaniecLuoSarnak2000} \S1 Remark D), but our $\alpha = 2/9$ is large enough to bear witness to the symplectic-type distribution of the low-lying zeros of our weighted family of $L$-functions.\\

\noindent Given that symplectic-type distribution was observed in the weight aspect alone version of this problem in \cite{KowalskiSahaTsimerman2012}, the result of Theorem \ref{thm:low-lying-zeros} is of course expected.  However this is in contrast to what one observes if one uses a constant weight in place of $\omega_{f, N, k}$.  It follows from \cite{ShinTemplier2013}, since the Frobenius--Schur indicator of the tautological representation of $\GSp_4(\mathbb{C})$ is $-1$, that we see even orthogonal symmetry in this case.  This statement holds in either the weight or level aspect version of our problem, and should hold in both simultaneously.  Thus the difference in symmetry type must be due to the weighting $\omega_{f, k, N}$.  As we have mentioned before this has been interpreted in \cite{KowalskiSahaTsimerman2012} \S5.4 as evidence for a version of B\"{o}cherer's conjecture, and a similar discussion is applicable in the context of increasing levels.

\end{document}